%% file: 00_main_NS-CH.tex
\documentclass[reqno,a4paper]{article}
\usepackage{amssymb,amsmath,epsfig,graphics,graphicx,color,subfig,blindtext}
\definecolor{lightblue}{rgb}{0.22,0.45,0.70}
\usepackage[colorlinks=true,breaklinks=true,linkcolor=lightblue,citecolor=lightblue]{hyperref}
\hypersetup{hypertexnames=false}
\usepackage{style_update}
\begin{document}

\title{Divergence-free and mass-conservative virtual element methods for the Navier--Stokes-Cahn--Hilliard system}
\author{Alberth Silgado\thanks{GIMNAP, Departamento de Ciencias B\'asicas, Universidad del B\'io-B\'io, Chill\'an, Chile.
E-mail: {\tt asilgado@ubiobio.cl}},\quad 
Giuseppe Vacca\thanks{Dipartimento di Matematica, Universit\`a degli Studi di Bari, via E. Orabona, 4, 70125,  Bari, Italia.
	E-mail: {\tt giuseppe.vacca@uniba.it}}  
}
\date{}
\maketitle

\begin{abstract}
In this work, we design and analyze semi/fully-discrete virtual element approximations for the time-dependent Navier--Stokes-Cahn--Hilliard equations, modeling the dynamics of two-phase incompressible fluid flows with diffuse interfaces.
A new variational formulation is derived involving solely the velocity, pressure, and phase field, together with corresponding a priori energy estimates. The spatial discretization is based on the coupling divergence-free and $C^1$-conforming elements of high-order, while the time discretization employs a classical backward Euler scheme. By introducing a novel skew-symmetric trilinear form to discretize the convective term in the Cahn--Hilliard equation, we propose discrete schemes that satisfy mass conservation and energy bounds. Moreover, optimal error estimates are provided for both formulations. Finally, two numerical experiments are presented to support our theoretical findings and to illustrate the good performance of the proposed schemes for different polynomial degrees and polygonal meshes.
\end{abstract}



\maketitle
\input{01_introduction}
\setcounter{equation}{0}
\input{02_continuousProblem}
\setcounter{equation}{0}
\input{03_VEMapproximation}
\setcounter{equation}{0}
\input{04_discrete_formulations}
\setcounter{equation}{0}
\input{05_Error_analysis}
\setcounter{equation}{0}
\input{06_numericalResults}
\setcounter{equation}{0}
\input{07_Conclusions}
{\small
\subsection*{Acknowledgements} 	
A.S and G.V. have been partially funded by  
PRIN2022 n. \verb+2022MBY5JM+  \emph{``FREYA - Fault REactivation: a hYbrid numerical Approach''} research grant
funded by the Italian Ministry of Universities and Research (MUR) and  by the European
Union through Next Generation EU, M4C2.
G.V. has been partially funded by 
PRIN2022PNRR n. \verb+P2022M7JZW+\emph{``SAFER MESH - Sustainable mAnagement oF watEr Resources: ModEls and numerical MetHods''} research grant,
funded by the Italian Ministry of Universities and Research (MUR) and  by the European
Union through Next Generation EU, M4C2.
G.V. is member of the Gruppo Nazionale Calcolo Scientifico (GNCS) at Istituto Nazionale di Alta Matematica (INdAM), Italia.
This research was carried out primarily during the postdoctoral fellowship of A.S. at the Università degli studio di Bari.}

\bibliographystyle{abbrv}
\bibliography{references_new}


\end{document}

%% file: 01_introduction.tex
\section{Introduction}
\paragraph{Scope and contributions.}  The Navier--Stokes-Cahn--Hilliard (NS-CH) system is a widely used mathematical model for describing the dynamics of two-phase incompressible fluid flows with diffuse interfaces. It couples the incompressible Navier--Stokes (NS) equations, which govern fluid motion, with the Cahn--Hilliard (CH) equation, which models phase separation and interfacial dynamics in a binary mixture. The coupling is realized through an additional phase-induced stress term in the NS equations and a convective transport term in the CH equation driven by the fluid velocity. This coupled system finds applications in processes involving complex multiphase flows, which play a crucial role in engineering and industrial contexts.  These include spinodal decomposition in binary mixtures, interfacial instabilities, phase inversion in emulsions, as well as the behaviour of polymer and microfluidic flows (see, for instance, \cite{Gross_2011,sussumanJCP2007,KIM2004,Diegel2017} and the references therein).  
Moreover, the NS-CH system provides an alternative to classical sharp interface models, where the interface is treated as a free-moving curve or surface evolving with the fluid. The classical interface problems are often challenging from both theoretical and computational perspectives due to the presence of surface tension and the occurrence of topological changes during the evolution, such as self-intersection, fattening, pinch-off, and splitting. In contrast, the diffuse-interface approach has been widely used to address these difficulties without requiring explicit interface tracking (for further details of this aspect, we refer to \cite{Feng2006}).

Despite its versatility, however, the NS-CH system also presents inherent mathematical and numerical challenges. The model is strongly nonlinear and involves fourth-order spatial derivatives due to the CH component, which makes its analysis and discretization nontrivial. The classical coupling between velocity, pressure, phase field, and chemical potential introduces additional stiffness and requires careful treatment to ensure stability, mass conservation, and energy dissipation at the discrete level. Moreover, the presence of small interfacial thickness parameters necessitates fine spatial and temporal resolution, further increasing the computational cost. 

Over the last years, considerable effort has been devoted to developing efficient numerical methods for the NS-CH system, mainly based on mixed Galerkin discretizations. For instance, in~\cite{Feng2006} the authors developed and analyzed a  FEM for the NS-CH equations that satisfies both a mass conservation law and a discrete energy law for the phase field. The time derivative is discretized using a backward Euler scheme, while the $\P_2-\P_0$ mixed finite element pair is employed for the spatial discretization. Well-posedness and error estimates for the fully-discrete formulation are provided. A similar approach was independently presented in~\cite{KSW2008}. More recently, in~\cite{Diegel2017}, the authors proposed a novel second-order-in-time scheme for the velocity formulation, with error estimates for the fully discrete problem. Other numerical schemes for related systems can be found in~\cite{Diegel2015,MR3324579,MR3606459,MR4921258}.

We emphasize that the NS and CH problems represent active lines of research, even when the two systems are considered separately. In particular, recent years have witnessed several contributions focused on the development of numerical schemes for the NS equations that exactly preserve the divergence-free constraint, which is essential for accurately capturing incompressible flow dynamics. This approach positively influences the computation of the velocity field in various challenging scenarios, including asymptotic pressure-robustness, hydrostatic and convection-dominated flows, and other situations where standard methods may suffer from loss of accuracy or stability~\cite{GN2014,div-free-SIAMreview}.  On the other hand, various numerical techniques have been developed to solve the CH system, including nonconforming and mixed finite element methods \cite{EF89,FPNM2004}, Discontinuous Galerkin and Hybrid High-Order schemes, isogeometric analysis, among others (see for instance~\cite{Elliot-NM,Sulli-SINUM,Xia2007,CDMP2016,GCBYH2008} and reference therein). A central difficulty in designing numerical schemes for the CH equation lies in the presence of fourth-order spatial derivatives, which require either $C^1$-conforming discretizations or a reformulation that introduces auxiliary variables—both of which increase the complexity of the method.

In recent years, several Virtual Element Methods (VEMs) have been developed and analyzed to solve the NS and CH models separately. In particular, divergence-free and $C^1$-conforming VEMs have been employed on general polygonal meshes to approximate the NS equations \cite{BLV-SINUM2018,MR4367665,BMV2019} and the CH systems~\cite{ABSV-SINUM2016,ASVV-2022}, respectively. These approaches exploit the method's ability to construct discrete subspaces that satisfy the incompressibility condition exactly and to handle higher-order operators naturally, without the need to introduce auxiliary variables. For alternative formulations we also refer to \cite{LHC2020,DH2024}, where mixed and nonconforming VEMs have been proposed and analyzed.

Initially introduced in \cite{BBCMMR-basicVEM}, the VEM is a Galerkin-type approach on polytopal meshes that enables the construction of stable discretizations without the need for explicitly defined basis functions. Its versatility is evidenced by remarkable features, such as the design of exactly divergence-free spaces for incompressible flows~\cite{ABMV-SINUM2016,BLV-M2AN2017,BLV-SINUM2018} and $H^2$-conforming approximations for fourth-order PDEs~\cite{BM-CMAME2013,CM-CAMWA2016,ABSV-SINUM2016,Brenner-VEM-M3AS,AMNS2021}. Moreover, the method exhibits strong robustness with respect to mesh distortion and offers natural adaptability to complex geometries, which makes it particularly well-suited for applications in fluid dynamics and phase-field models. Over the years, the VEM has been successfully applied to a wide range of problems, and its analysis has been extended to different others~\cite{VEM-elasticity-SINUM,VEM-aposteriori-NM,VEM-div-curl-NM,VEM-Carstensen-NM,VEM-Long-MC,VEM-NS-PNP-JSC,MR4849440}, demonstrating its flexibility across a broad range of PDEs. For a comprehensive account of the construction, theoretical properties of these spaces, and related developments, we refer the reader to \cite{VEM-ActaNum2023}.

According to the above discussion, in the present contribution we further explore the ability of the VEM to approximate challenging coupled nonlinear fluid-flow and fourth-order problems. In particular, we design and analyze a numerical scheme for the time-dependent NS-CH system posed on both standard simplicial meshes and more general polygonal grids.

At the continuous level, we propose a novel variational formulation involving only the velocity, pressure, and phase-field variables. This approach avoids the introduction of the chemical potential as an auxiliary unknown, which is customary in classical mixed formulations and typically leads to systems involving four strongly coupled variables with increased analytical complexity due to their mixed nature. In addition, for the resulting formulation, we establish suitable energy estimates and prove a mass conservation law.

At the discrete level, by coupling exactly divergence-free and $C^1$-conforming virtual element spaces, we construct a semi-discrete scheme, which is subsequently combined with a backward Euler time discretization to obtain a fully discrete formulation. As a key contribution of this work, we introduce a novel skew-symmetric trilinear form for the discretization of the convective term in the CH equation. This choice allows us to preserve discrete mass conservation and to derive energy bounds at both the semi- and fully-discrete levels. In addition, we establish essential boundedness properties of the discrete forms, which play a crucial role in the subsequent error analysis.
We emphasize that the analysis primarily focuses on the spatial discretization, which constitutes the core novelty of the method, with particular attention to the coupling of divergence-free and  $C^1$-conforming virtual element spaces, as well as to the proposed mass conservation law. For completeness, some error estimates for the fully-discrete scheme are also provided, while the temporal discretization is handled using standard techniques.

Building on these results, we introduce suitable Stokes- and Cahn--Hilliard-type Ritz projections, which map functions from the continuous solution spaces onto their virtual discrete counterparts, and establish their approximation properties. In particular, we derive new error estimates for the Cahn--Hilliard-type Ritz projection within this high-order conforming framework. We then derive new error estimates that explicitly account for the variational crimes arising from the discretization of the trilinear forms. By combining these estimates with suitable discrete Gr\"{o}nwall-type inequalities, we obtain optimal convergence rates for all variables in the corresponding energy norms, under standard regularity assumptions. Finally, we propose a simple projection-based postprocessing to approximate the chemical potential (cf. Remark~\ref{rm:chemical}) and present numerical experiments that validate the theoretical findings and illustrate the effectiveness of the proposed fully-discrete scheme.

To the best of our knowledge, this is the first work coupling exactly divergence-free and $C^1$-conforming Galerkin schemes for the NS-CH system using only the velocity, pressure, and phase-field variables. This yields a structure-preserving discretization that enforces incompressibility exactly, preserves mass, avoids additional unknowns, and naturally handles higher-order derivatives. Such properties enhance both accuracy and stability, as previously observed when the two systems were considered separately (cf.~\cite{BLV-SINUM2018,ABSV-SINUM2016}).

\paragraph{Notation.}
Let $\Omega \subset \R^2$ be a polygonal domain with boundary $\Gamma := \partial \Omega$, and let
$\bn = (n_i)_{1 \le i \le 2}$ denote the outward unit normal vector on $\Gamma$.
We denote by $\partial_{\bn}$ the normal derivative and by
$\boldsymbol{t} = (t_i)_{1 \le i \le 2}$ the unit tangential vector on $\Gamma$, oriented such that
$t_1 = -n_2$ and $t_2 = n_1$.

With a usual notation   
the symbols $\nabla$ and $\Delta$, $\bD^2$  
denote the gradient, the Laplacian and the Hessian for scalar functions, while 
$\boldsymbol{\nabla} $, $\boldsymbol{\Delta}$,   ${\rm div}$
denote the gradient, the  Laplacian  and the divergence operator
for vector fields.

For any open bounded domain $\mD \subseteq \Omega$, we use the standard notation for the Banach spaces
$L^q(\mD)$ and the Sobolev spaces $W_q^s(\mD)$, with $s \ge 0$ and $q \in [1,\infty]$, endowed with their
usual seminorms and norms.
In particular, for $q=2$ we identify $W_2^s(\mD) = H^s(\mD)$.
The space $L^2(\mD)$ is equipped with the inner product
\[
(\phi,\psi)_{\mD} := \int_{\mD} \phi\,\psi .
\]
For $s \ge 0$, the the seminorm and norm in $H^s(\mD)$ are given by
\[
|\phi|_{s,\mD} := \left( \sum_{|\boldsymbol{\alpha}| = s} \|\partial^{\boldsymbol{\alpha}} \phi\|_{0,\mD}^2 \right)^{1/2}, \quad \|\phi\|_{s,\mD} := \left( \sum_{|\boldsymbol{\alpha}| \le s} \|\partial^{\boldsymbol{\alpha}} \phi\|_{0,\mD}^2 \right)^{1/2},
\]
where $\boldsymbol{\alpha} \in \N_0^2$ is a multi-index and
$\partial^{\boldsymbol{\alpha}} := \partial_{x_1}^{\alpha_1}\partial_{x_2}^{\alpha_2}$.

We introduce the following functional spaces:
\[
H_0^1(\Omega)
:=
\{ v \in H^1(\Omega) : v = 0 \text{ on } \partial\Omega \},
\qquad
L_0^2(\Omega)
:=
\left\{ v \in L^2(\Omega) :  (v,1)_{\Omega} = 0 \right\}.
\]

Let $t \in (0,T]$ denote the time variable, with $T>0$ a final time.
We denote with $\partial_t$ the derivative with respect to the time variable.
For any Hilbert space $W$ endowed with the norm $\|\cdot\|_{W}$, we define the Bochner spaces $L^2(0,t;W)$ and $L^\infty(0,t;W)$ with norms
\[
\|\phi\|_{L^2(0,t;W)}
:= \left( \int_0^t \|\phi(\cdot, \sigma)\|_{W}^2\,d \sigma \right)^{1/2},
\qquad
\|\phi\|_{L^\infty(0,t;W)}
:= \operatorname*{ess\,sup}_{0 \le \sigma \le t} \|\phi(\cdot, \sigma)\|_{W}.
\]

Finally, given a scalar-valued space $W$, we denote by $\bW$ its vector-valued counterpart. Examples include $L^2(\O)$ and $\bL^2(\O)$.

\paragraph{Outline.} The remainder of this work is organized as follows. In Section~\ref{sec:cont_problem} the NS-CH model and its continuous weak formulation is introduced.  Section~\ref{sec:disc:problem} develops the conforming VE approximations and key spatial discretization properties. Subsequently, in Section~\ref{sec:discrete:formulations} we define the semi- and fully-discrete schemes and analyze their mass conservation and energy stability. Error analysis and convergence results are presented in Section~\ref{sec:error:analysis}. In Section~\ref{sec:numericalResults}, we provide some numerical details,
together with two numerical experiments that validate the theoretical results and illustrate the performance of the proposed scheme. Finally, in Section~\ref{sec:conclusions}, we provide concluding remarks and outline possible directions for future research.

%% file: 02_continuousProblem.tex
\section{The model problem and continuous weak formulation}\label{sec:cont_problem}
We consider a phase-field model for two immiscible and incompressible fluids with comparable densities, which are taken equal to one. The model is described by the following time-dependent coupled system; see, e.g., \cite{LS2003,Feng2006}
\begin{equation}\label{model_a}
	\left\{
	\begin{aligned}	
	\partial_t\bu - \nu \bDelta\bu + ( \bnabla \bu) \bu  +\nabla p + \lambda  \div (\nabla \varphi\otimes \nabla \varphi ) &= \bg   &\quad \qin  \O \times (0,T], \\
	\div \bu &= 0 &\quad \qin \O \times (0,T],\\
	\partial_t\varphi + \div (\varphi \bu)+ \gamma \Delta (\Delta  \varphi -\varepsilon^{-2} f(\varphi)) &= 0 &\quad \qin \O \times (0,T]. 
		\end{aligned} \right. 
	\end{equation}
Here $f(\varphi) = F'(\varphi)$, where $F(\varphi) = \frac{1}{4}(\varphi^2 - 1)^2$ and the positive constants $\nu$,   $\lambda$, $\gamma$, and $\varepsilon$ denote the fluid viscosity, the surface tension, the elastic relaxation time, and the capillary width, respectively and we assume the external load $\bg \in L^2(0,T; \bL^2(\Omega))$.
The system is complemented with the following  homogeneous boundary conditions
\begin{equation}\label{bou:data}
	\bu = \0,    \qquad \partial_{\bn} \varphi= \partial_{\bn} \Delta\varphi =0\,  \quad \qin  \Gamma \times (0,T] \,
\end{equation}
and initial conditions
\begin{equation}\label{initial:data}
\bu(\cdot, 0 ) = \bu_0, \qquad \varphi(\cdot, 0) = \varphi_0  \quad \:\: \quad\quad \qin  \O \,.
\end{equation}

In the equations~\eqref{model_a}-\eqref{initial:data}, the vector $\bu(\bx, t) \in \R^2$ and the scalar $p(\bx, t) \in \R$ denote the velocity and the pressure of the fluid mixture at the space-time point $(\bx, t)$, respectively, while the scalar function $\varphi$ denotes a phase function and is used to indicate the fluid phase. We note that on the boundary of the domain we impose a no-flux type condition for both $\varphi$ and $\Delta \varphi$.
Moreover, $\nabla \varphi \otimes \nabla \varphi$ stands for the $2 \times 2$ rank-one matrix $(\nabla \varphi)^T \nabla \varphi$ with entries $\varphi_{x_i} \varphi_{x_j}$. 

Now, by using some vector calculus identities and the incompressibility condition in \eqref{model_a}, we have 
\begin{equation*}
	\begin{split}
	 \div (\nabla \varphi\otimes \nabla \varphi ) &= \nabla \varphi \: \div(\nabla \varphi) + (\nabla (\nabla \varphi)) \nabla \varphi  	 = 	 \nabla \varphi \Delta \varphi  + \frac{1}{2} \nabla |\nabla \varphi|^2 \,,\\
	 \div (\varphi \bu) &= \varphi \: \div \bu + \bu \cdot \nabla  \varphi =  \bu \cdot \nabla  \varphi \,. 
	 	\end{split}
\end{equation*} 
Then, setting $\widehat{p}=p +\frac{\lambda}{2} |\nabla \varphi|^2$, the system~\eqref{model_a} is equivalent to the following problem
\begin{equation}\label{Mmodel_a}
	\left\{
	\begin{aligned}	
\partial_t \bu - \nu \bDelta\bu +   (\bnabla \bu) \bu  +\nabla \widehat{p} + \lambda  \nabla \varphi \Delta \varphi &= \bg   \quad \qin  \O \times (0,T], \\
	\div \bu &= 0 \quad \qin \O \times (0,T],\\
	\partial_t\varphi +\bu \cdot \nabla  \varphi+ \gamma \Delta (\Delta  \varphi -\varepsilon^{-2} f(\varphi)) &= 0 \quad \qin \O \times (0,T].
\end{aligned} \right. 
\end{equation}
Let us consider the following compact notation for the Sobolev spaces
\begin{equation*}
\bV := \bH^1_0(\O)\,, \qquad  Q := L_0^{2}(\O)\,,  \quad H = \{\phi \in H^2(\O): \partial_{\bn} \phi =0 \qon \Gamma  \}\,, 
\end{equation*}
representing the velocity field space,  the pressure space and the phase filed space respectively, endowed  with the norms $$\|\bv\|_{\bV} := |\bv|_{1,\O}, \quad  \|q\|_{Q} := \|q\|_{0,\O}, \quad
  \|\phi\|_{H} := |\phi|_{2,\O}   \quad \forall (\bv,q,\phi) \in \bV \times Q \times  H.$$
We define the following forms:
\begin{equation*}
	\begin{aligned}
m_{F}(\bu,\bv) &:= (\bu, \bv)_{\O}, &\qquad 
m(\varphi, \phi  ) &:=(\varphi , \phi)_{\O},\\ 		
a_{\bnabla}(\bu,\bv) &:= (\bnabla\bu, \bnabla \bv)_{\O}, &\qquad a_{\bD}(\varphi, \phi  )&:=(\bD^2 \varphi , \bD^2 \phi)_{\O}, 
\\ 
c_{F}(\bw;\bu,\bv) &:= \frac{1}{2}  ( (\bnabla \bu) \bw, \bv)_{\O} - \frac{1}{2}( (\bnabla \bv) \bw, \bu)_{\O}, 
&\qquad
c(\bu;\varphi,\phi) &:=\frac{1}{2} ( \bu \cdot \nabla \varphi, \phi)_{\O} - \frac{1}{2}( \bu \cdot \nabla \phi, \varphi)_{\O},
\\	
d(\varphi;\phi, \bv) &:= ( \Delta \varphi\nabla \phi, \bv)_{\O},  &\qquad 
r(\eta;\varphi,\phi) &:= ( f^{\prime}(\eta)\nabla \varphi, \nabla\phi)_{\O},
\\
b(\bv,q) &:= (q,\div \bv)_{\O},
\end{aligned}
\end{equation*}
for all $\bu,\bv,\bw \in \bV$,  $q \in Q$ and $\eta,\varphi,\phi \in H$.

Thus, multiplying the equations~\eqref{Mmodel_a}  by test functions $(\bv,  q,\phi) \in  \bV \times Q \times H$, integrating by parts, exploiting the boundary conditions in~\eqref{bou:data} and the incompressibility condition $\div \bu =0$, we obtain the following variational formulation: 	seek 
\[
\bu \in    L^{\infty}(0,T;\,  \bL^2(\Omega))  \cap L^{2}(0,T; \bV) \,, 
\quad
\widehat{p} \in  L^{2}(0,T; Q)\,,  
\quad
\varphi \in    L^{\infty}(0,T;\,  L^2(\Omega))  \cap L^{2}(0,T; H)
\]
such that for a.e. $t \in (0, T]$
\begin{equation}\label{form1}
	\left\{
	\begin{aligned}
		m_F(\partial_t\bu,\bv) + \nu a_{\bnabla}(\bu, \bv) +  c_{F}(\bu;\bu,\bv)  - b(\bv,\widehat{p})  + \lambda  d(\varphi;\varphi, \bv) &= (\bg,\bv)_{\O}   & \forall \bv \in \bV,\\ 
		b(\bu,q)  &= 0  & \forall q \in Q,\\
		m(\partial_t \varphi, \phi )+\gamma a_{\bD}(\varphi, \phi )  +c(\bu;\varphi,\phi)+ \gamma \varepsilon^{-2} r(\varphi;\varphi,\phi)  &= 0 & \forall \phi \in H,		
	\end{aligned}
	\right.
\end{equation}
coupled with the initial conditions in~\eqref{initial:data}.
 
Now, let us introduce the kernel 
\[
\bZ := \{ \bv \in  \bV : \div \bv =0 \qin \O\}\,.
\]
Moreover, we introduce the concept of the \textit{Helmholtz--Hodge projector} (see for instance \cite[Lemma 2.6]{div-free-SIAMreview}). For any $\bv \in \bL^2(\Omega)$, there exist $\bv_0 \in H(\text{div}; \Omega)$ and $\chi \in H^1(\Omega)/\mathbb{R}$, satisfying
\begin{equation}
	\bv = \bv_0 + \nabla \chi,
	\label{eq:HHdecomp}
\end{equation}
where $\bv_0$ is $L^2$-orthogonal to the gradients, i.e., 
$(\bv_0, \nabla \zeta) = 0$ for all $\zeta \in H^1(\Omega)$, which implies, in particular, that $\div \bv_0 = 0$.
The orthogonal decomposition \eqref{eq:HHdecomp} is unique and is called the \textit{Helmholtz--Hodge decomposition}, and $\mathcal{H}(\bv) := \bv_0$ is the \textit{Helmholtz--Hodge projector} of $\bv$.

Therefore, problem~\eqref{form1} can be reformulated in the following kernel formulation: seek 
\[
\bu \in   L^{\infty}(0,T;\,  \bL^2(\Omega))  \cap L^{2}(0,T; \bZ)\,, 
\quad 
\varphi \in  L^{\infty}(0,T;\,  L^2(\Omega))  \cap L^{2}(0,T; H) 
\]
such that for a.e. $t \in (0,T]$
\begin{equation}\label{form:kernel}
	\left\{
	\begin{aligned}
		m_F(\partial_t \bu,\bv) + \nu a_{\bnabla}(\bu, \bv) +  c_{F}(\bu;\bu,\bv)  + \lambda  d(\varphi;\varphi, \bv) &=(\mathcal{H}(\bg),\bv)_{\O}  &  \forall \bv \in \bZ,\\ 
		m(\partial_t \varphi, \phi )+\gamma a_{\bD}(\varphi, \phi  )  +c(\bu;\varphi,\phi)+ \gamma \varepsilon^{-2} r(\varphi;\varphi,\phi)  &= 0& \forall \phi \in H,	
	\end{aligned}
	\right.
\end{equation}
coupled with the initial conditions in~\eqref{initial:data}.

\subsection{A priori energy estimates}
In this subsection we will present some energy estimates satisfied by the solution of problem~\eqref{Mmodel_a}. Let us introduce the following notation~\cite{Feng2006}:
\begin{equation*}
\mathcal{J}_{\varepsilon,\lambda}(\bv,\phi) := \frac{1}{2}\|\bv\|_{0, \O}^2 + \frac{\lambda}{2} |\phi|_{1, \O}^2 +  \lambda\varepsilon^{-2} (F(\phi), 1)_{\O}.
\end{equation*}
For any $w \in H^{-1}(\Omega)$, we define $\Delta^{-1} w \in H^1_0(\Omega)$
	as the unique solution of
	\begin{equation}
	\label{eq:lapm1}
	(\nabla \Delta^{-1} w, \nabla v)_\Omega = -(w,v)_\Omega
	\qquad \forall v \in H^1_0(\Omega) \,.
	\end{equation}

Therefore for all $w \in H^{-1}(\Omega)$ we have
\[
\|w\|_{-1, \Omega} = \sup_{v \in H^1_0(\Omega)} \frac{(w, v)_\Omega}{\|\nabla v\|_\Omega}
= 
\sup_{v \in H^1_0(\Omega)} -\frac{(\nabla \Delta^{-1} w, \nabla v)}{\|\nabla v\|_\Omega} = \|\nabla \Delta^{-1} w\|_{\Omega} \,,
\]	
and thus
\begin{equation}
\label{eq:utile}
\|w\|_{-1, \Omega}^2 = -(w, \Delta^{-1} w)_{\Omega} \,.
\end{equation}
With the above notation, we state the following result.

\begin{proposition}[Continuous conservation laws]
\label{prop:cont:cons}
Suppose that the initial values  $(\bu_0, \varphi_0)$ in \eqref{initial:data} satisfy  
$\mathcal{J}_{\varepsilon,\lambda}(\bu_0,\varphi_0) < \infty $. Then, every regular solution $(\bu, \varphi)$ of problem~\eqref{Mmodel_a} satisfies the following estimates:
\begin{align}
&(\varphi(\bx, t),1)_{\O} = (\varphi_0(\bx),1)_{\O} \qquad \text{for a.e. $t \in (0,T)$,} \label{mass:conserv}\\ 
&\underset{0 \leq t \leq T}{\text{ess sup}} \,\, \mathcal{J}_{\varepsilon,\lambda}(\bu(\cdot, t),\varphi(\cdot, t)) \leq C_{{\rm stab}}, \label{energy:1} \\
&\int_{0}^T \nu |\bu|^2_{1,\O} + \int_{0}^T  \lambda \gamma^{-1} \|\varphi_t +\bu \cdot \nabla  \varphi\|^2_{-1,\O}  \leq C_{{\rm stab}}, \label{energy:2}
\end{align}
where $C_{{\rm stab}}$ is a positive constant depending on the data and physical parameters.
\end{proposition}

\begin{proof}
The mass conservation property~\eqref{mass:conserv} is obtained by integrating by parts the third equation in \eqref{Mmodel_a}, exploiting the boundary conditions in~\eqref{bou:data} and the incompressibility condition $\div \bu =0$,
and then integrating over the interval $(0,t)$.

Now, in order to obtain bound~\eqref{energy:1}, we proceed as follows. We preliminary consider the following identities:
\begin{equation}
\label{eq:utileF}
\nabla F(\varphi) = \nabla \varphi \, f(\varphi) \,,
\qquad 
\partial_t F(\varphi) = \partial_t \varphi \, f(\varphi) \,.
\end{equation}
From the momentum equation~\eqref{Mmodel_a}, and using the first identity above, we have 
\begin{equation*}
\partial_t \bu - \nu \Delta\bu +   (\bnabla \bu) \bu  +\nabla (\widehat{p}+ \lambda\varepsilon^{-2} F(\varphi)) + \lambda  \nabla \varphi \big(\Delta \varphi- {\varepsilon^{-2}}f(\varphi) \big) = \bg. 
\end{equation*}	
Next, testing by $\bu$, integrating by parts, using that $\div \bu =0$ and the boundary conditions \eqref{bou:data}, we  infer
\begin{equation}\label{Eq:velo}
\frac{1}{2} \partial_t\|\bu\|_{0,\O}^2 + \nu |\bu|^2_{1,\O}  + \lambda  \big( \Delta \varphi- {\varepsilon^{-2}}f(\varphi), \bu \cdot \nabla \varphi \big)_{\O}= (\mathcal{H}({\bg}), \bu)_{\O}. 
\end{equation}	
On the other hand, multiplying the third equation of~\eqref{Mmodel_a} by $-\lambda \gamma^{-1} \Delta^{-1}(\partial_t \varphi+\bu \cdot \nabla  \varphi)$, integrating by parts, using \eqref{eq:utile} and \eqref{eq:lapm1}, and employing  the second identity in \eqref{eq:utileF}, we have 
\begin{equation}\label{Eq:Cahn}
\begin{aligned}
0 &=   -\lambda \gamma^{-1}  \big( \partial_t \varphi +\bu \cdot \nabla  \varphi,  \Delta^{-1}(\partial_t \varphi +\bu \cdot \nabla  \varphi) \big)_{\O}- \lambda \big( \Delta (\Delta  \varphi -\varepsilon^{-2} f(\varphi)),  \Delta^{-1}(\partial_t \varphi +\bu \cdot \nabla  \varphi)  \big)_{\O} \\
&= \lambda \gamma^{-1} \| \partial_t \varphi +\bu \cdot \nabla  \varphi\|^2_{-1, \O} - \lambda  \big( \Delta  \varphi -\varepsilon^{-2} f(\varphi),  \partial_t \varphi +\bu \cdot \nabla  \varphi  \big)_{\O}\\
&= 
\lambda \gamma^{-1}  \| \partial_t \varphi +\bu \cdot \nabla  \varphi\|^2_{-1, \Omega} + \frac{\lambda}{2} \partial_t |\varphi|^2_{1,\O}
+\lambda \varepsilon^{-2} \partial_t (F(\varphi), 1)_{\O}  - \lambda  \big( \Delta \varphi- {\varepsilon^{-2}}f(\varphi), \bu \cdot \nabla \varphi \big)_{\O}.
\end{aligned}
\end{equation}
Now, adding equations~\eqref{Eq:velo} and \eqref{Eq:Cahn} we obtain the following  basic energy law
\begin{equation}\label{energy:law}
\partial_t \mathcal{J}_{\varepsilon,\lambda}(\bu, \varphi) + \nu |\bu|^2_{1,\O}  +\lambda \gamma^{-1}  \| \varphi_t +\bu \cdot \nabla  \varphi\|^2_{-1, \Omega} 
=(\mathcal{H}(\bg), \bu)_{\O}.
\end{equation}
Thus, bounds \eqref{energy:1} and \eqref{energy:2}, follows from the identity \eqref{energy:law}, Young and Gr\"{o}nwall inequalities (cf. Lemma \ref{cont:gronwall}).
\end{proof}

We refer the reader to~\cite{LS2003} (see also~\cite{Chupin2003,FY-mathComp2007,Feng2006}) for the existence and uniqueness analysis of the weak solution to problem~\eqref{Mmodel_a}.
In what follows, we will assume the following regularity condition on the  phase solution $\varphi$ of problem~\eqref{form1}.

\begin{assumption}
[Regularity assumption on the continuous phase solution]
\label{Assump:cont:phase}
The solution $\varphi$ of  problem \eqref{form:kernel} satisfies for a.e. $t \in (0,T]$,
	\[
	\|\varphi(\cdot, t)\|_{W^{2,\infty}(\O)} \leq C_T,
	\] 
for a constant $C_T$ depending on the final time $T$. 
\end{assumption}

We conclude this section with the following two remarks.

\begin{remark}
	$W^{s,\infty}$-regularity assumptions of the type stated in Assumption~\ref{Assump:cont:phase} are commonly employed in the numerical analysis of phase-field models. In particular, analyses based on mixed formulations typically require boundedness in weaker norms, such as
	$\|\varphi(\cdot,t)\|_{L^{\infty}(\Omega)}$; see, for instance,~\cite{KSW2008}. On the other hand, for Cahn--Hilliard-type problems treated within a primal formulation, $W^{1,\infty}$-regularity assumptions have also been considered in the literature (see, e.g.,~\cite{EF89}).
	
	In the present work, the analysis is conducted within a primal (for the phase-field equation) and fully coupled framework. As a consequence, stronger regularity requirements on the phase variable naturally arise and are necessary to ensure a consistent treatment of the coupling terms appearing in the model. To the best of our knowledge, the explicit use of a $W^{2,\infty}$-regularity assumption in this context has not been previously addressed in numerical analyses, where weaker norms are typically sufficient due to the adoption of mixed formulations.
	
	This assumption will be instrumental in deriving the error estimates; see Lemmas~\ref{lemma:crimen:d-B} and~\ref{lemma:momentum:error}.
\end{remark}

\begin{remark}\label{Remark:mass:conserv}
The relations~\eqref{mass:conserv} and~\eqref{energy:law} are known  as the mass conservation law and the basic energy law for the system \eqref{Mmodel_a}, respectively. Moreover, it is worth highlighting that the formulation \eqref{Mmodel_a} also allows the derivation of an energy bound of the form~\eqref{form1}, similar to the one derived in mixed formulations (see, e.g., estimate (3.3) in~\cite{Feng2006}).

We emphasize that designing a discrete scheme that preserves all these properties is not straightforward. In this work, we propose a novel discrete convective term for the phase-field equation, specifically designed to ensure mass conservation and a form of energy bound. Preserving the discrete energy law remains a challenging task and an interesting direction for future research (see Remark~\ref{remark:properties} below).
\end{remark}

%% file: 03_VEMapproximation.tex
\section{Conforming virtual element approximation}\label{sec:disc:problem}
In this section we shall introduce the div-free and  $C^1$-conforming VEMs,
of arbitrary order $k \geq 1$ and $\ell\geq 2$,  for the velocity  and phase field variables, respectively.  
We start by recalling notation and basic setting for the virtual element approach in two dimensions. By combining the advancements that will be presented here and the construction in~\cite{BDM3AS2020,BDRCAMWA2020} it is possible the  design of the three dimensional case for the NS-CH system \eqref{form1}.

\subsection{Polygonal decomposition and polynomial projections}
\label{subsec:meshassump}
Henceforth, $\E$ denotes a generic polygon, while $h_\E$, $|\E|$, and $\bx_{\E}$ denote its diameter, area, and barycenter, respectively.
We define $\bn_{\E}$ as the unit outward normal vector to $\E$ and $\bt_{\E}$ as the unit tangent vector to $\partial \E$.
Moreover, we denote by $e$ a generic edge of $\partial \E$ with length $h_e$, and we use the notation $\bn_e$ and $\bt_e$ for the unit normal and tangential vectors to the edge $e$, respectively.

Let $\fOh$ be a family of decompositions of $\O$ into non-overlapping polygons $\E$, where  $h:=\max_{\E\in\Oh} h_\E.$ 
For the theoretical analysis, we suppose that $\Oh$ satisfies the following assumption.

\begin{assumption}[Assumption on the mesh]\label{Assump:mesh}
There exists a  positive constant $\rho>0$ such that for any $\E\in \Oh$ it holds: (i) $\E$ is star-shaped with respect to every point of a  ball
of radius  $ \ge  \rho h_\E$; (ii) any  edge $e \subset \partial \E$ has length $h_e\ge \rho h_\E$.
\end{assumption}

For any $\mD \subset \R^2$ and any integer $n\geq 0$, we introduce the following spaces: $\P_{n}(\mD)$ as the space of
polynomials of degree  up to $n$ defined on $\mD$ and we denote by $\bP_{n}(\mD)$ 
its vectorial version, i.e.,  $\bP_{n}(\mD):=[\P_{n}(\mD)]^2$. However, for tensors we will keep $[\P_{n}(\mD)]^{2 \times 2}$.  A natural basis associated with the space $\P_n(E)$ is the set of normalized
monomials: $	\M_n(E) := \left\{ 
m_{\boldsymbol{\alpha}},
\,\,\,  \text{with} \,\,\,
|\boldsymbol{\alpha}| \leq n
\right\}$
where, for any multi-index $\boldsymbol{\alpha} = (\alpha_1, \alpha_2) \in \N_0^2$ 
{\small
\begin{equation*}
m_{\boldsymbol{\alpha}} :=
\prod_{i=1}^2  
\left(\frac{x_i - x_{E, i}}{h_{E}} \right)^{\alpha_i}
\qquad \text{and} \qquad
|\boldsymbol{\alpha}|:= \sum_{i=1}^2 \alpha_i \,.
\end{equation*}}
Moreover for any $m \leq n$, we denote with $\P_{n \setminus m}(E) = {\rm span}  
\left\{ 
m_{\boldsymbol{\alpha}}: m +1 \leq |\boldsymbol{\alpha}| \leq n
\right\}.$

Furthermore for any $n \in \N$ and for any $\E \in \O_h$,   we define the following polynomial projections:
\begin{itemize}
	\item the $L^2$\textbf{-projection} $\Pi_{\E}^{0,n} \colon L^2(E) \to \P_n(\E)$, given by
	\begin{equation*}
	( v - \, {\Pi}_{\E}^{0,n}  v, q_n)_{\E} = 0 \qquad  \text{for all $q_n \in \P_n(\E)$,  for all $v \in L^2(\E)$.} 
	\end{equation*} 
	
Analogously, we define the extension for vector functions $\boldsymbol{\Pi}_{\E}^{0,n} \colon \bL^2(E) \to \bP_n(E)$ and 
tensor functions $\boldsymbol{\Pi}_{\E}^{0,n}  \colon [L^2(\E)]^{2\times 2} \to [\P_n(E)]^{2\times 2}$;
	
	\item the $H^1$\textbf{-seminorm projection} $\PinnablaKb  \colon \bH^1(E) \to \bP_n(\E)$, defined by 
	\begin{equation*}\label{eq:Pinabla}
		\left\{
		\begin{aligned}
			  (\bnabla( \bv - \, \PinnablaKb  \bv), \bnabla  \bq_n)_{E} &= 0 \quad  \text{for all  $\bq_n \in \bP_n(\E)$,} \\
			 (\bv - \,  \PinnablaKb \bv,{\bf 1})_{\partial \E} &= 0 \, ,
		\end{aligned}
		\right.
		\quad  \text{for all $\bv \in \bH^1(\E)$,}
	\end{equation*}

	\item the $H^2$\textbf{-seminorm projection} $\Pi_{\E}^{\bD,n}  \colon H^2(E) \to \P_n(\E)$, defined by 
\begin{equation*}\label{eq:PiD}
	\left\{
\begin{aligned}
	(\bD^2( v - \, \Pi_{\E}^{\bD,n}   v), \bD^2 q_{n})_{\E}  &= 0 \quad  \text{for all $q_n \in \P_n(\E)$,} \\
	(v - \, \Pi_{\E}^{\bD,n} v,1 )_{\partial \E} &= 0 \\
	(\nabla (v -\Pi_{\E}^{\bD,n} v),{\bf 1})_{\partial \E}&=\0 \,,
\end{aligned}
	\right.
	\quad \text{for all $v \in H^2(\E)$.}
\end{equation*}
\end{itemize}

In the following $C$ will denote a generic positive constant, independent of the mesh size $h$ and the time step size $\tau$ (cf. Subsection \ref{subsec:fully}), and possibly depending on $\Omega$, the final time $T$,  
the order  $(k, \ell)$ of the method (cf. Subsection \ref{subsec:space:velo-pressure} and Subsection \ref{subsec:space:phase}), and the mesh regularity constant $\rho$ in Assumption \ref{Assump:mesh}.
We explicitly track the dependence of $C$ on 
the problem data  $\nu$, $\lambda$, $\gamma$, $\varepsilon$, $\bg$, $\bu_0$ and $\varphi_0$, 
the problem solution $(\bu, \widehat{p}, \varphi)$, 
the constants $C_T$, $\widetilde{C}_T$,  $\widehat{C}_T$, $C_{\O}$ in Assumptions \ref{Assump:cont:phase},  \ref{Assump:disc:phase}, \ref{Assump:f-disc:phase} and \ref{Assump:reg:aux}.  
We use the notation $\lesssim$ (resp. $\approx$) to indicate a bound (resp. equivalence) up to a multiplicative factor $C$.

We now recall the following polynomial approximation result (see, e.g.,~\cite{Brenner-Scott-Book}).

\begin{proposition}[Bramble-Hilbert]\label{approx:prop:L2}
Let Assumption \ref{Assump:mesh} hold.  
Let $n \in \N$ and let $0 \leq r \leq s \leq n+1$, and $1 \leq p,q \leq \infty$ such that $s - 2/q \ge r - 2/p$.
Then, for any $E \in \Oh$ and for each $w \in W^{s,q}(E)$ the following hold
\[
|w - \Pi^{0,n}_E w|_{W^{r,p}(E)} \lesssim
h_E^{s-r + 2/p - 2/q} \,|w|_{W^{s,q}(E)}
\,.
\] 	
\end{proposition}

\subsection{Virtual element spaces for the velocities and pressures}\label{subsec:space:velo-pressure}
In this  subsection we outline an overview of the divergence-free VE spaces for the velocity and pressure variable of the NS equation~\cite{ABMV-SINUM2016, BLV-SINUM2018}. For each polygonal element $\E \in \O_h$ and any $k \geq 1$, we consider the index
 $\widehat{k} = \max\{k, 2\}$ and the following ``enhanced'' virtual space 
\begin{equation}\label{eq:localVirtual-Velo}
	\begin{aligned}
\VK := \biggl\{  
		\bv_h \in \bC^0(\overline{E})  \: \: \colon \: \:
		(i)& \,\,  
		\bDelta    \bv_h  +  \nabla s \in \bx^{\perp} \P_{k-1}(E), 
		\,\,\text{ for some $s \in L_0^2(E)$,} 
		\\
		(ii)&  
		\,\,   \div \, \bv_h \in \P_{k-1}(E) \,, 
		\\
		(iii) &		\,\,  \bv_h|_e \cdot \bt_e \in \P_k(e), \:\: \bv_h|_e \cdot \bn_e \in  \P_{\widehat{k}}(e)\,\,\, \, \forall e \in \dE, 
		\\
		(iv) &
		\,\,   (\bv_h - \PinablaKb\bv_h, \, \bx^{\perp} \, \widehat{p}_{k-1} )_E = 0
		\,\,\,\, \text{$\forall \widehat{p}_{k-1} \in \widehat{\P}_{k-1 \setminus k-3}(E)$}
		\,\,\biggr\} \,,
	\end{aligned}
\end{equation}
where $\bx^{\perp} = (x_2, -x_1)$. Now, we introduce following linear operators $\boldsymbol{D_{V}}$ for the space $\VK$: 
\begin{itemize}
	\item[$\boldsymbol{D_{V}1}$:] the values of $\bv_h(\ba_i)$  for all vertex $\ba_i$ of the polygon $\E$;
	\item[$\boldsymbol{D_{V}2}$:] the values of $\bv_h \cdot \bn_e$ and $\bv_h \cdot \bt_e$ at $k - 1$ and $\widehat{k} - 1$ internal distinct points, respectively, on each edge $e \in \partial E$;
	\item[$\boldsymbol{D_{V}3}$:] for $k \geq 3$, the moments of $\bv_h$ 
	$$
	|\E|^{-1}  (\bv_h \cdot \boldsymbol{m}^{\perp} , m_{\boldsymbol{\alpha}})_{\E} 
	\quad 
	\text{$\forall m_{\boldsymbol{\alpha}} \in \M_{k-3}(E)$,}
	$$
	where $\boldsymbol{m}^{\perp} := h^{-1}_E (\bx^{\perp}- \bx^{\perp}_{\E})$;
	\item[$\boldsymbol{D_{V}4}$:] for $k \geq 2$, the moments of $\div  \bv_h$ 
	$$
	 h_{\E}^{-1}  (\div \bv_h , m_{\boldsymbol{\alpha}})_{\E}  
	\quad \text{$\forall  m_{\boldsymbol{\alpha}} \in \M_{k-1}(E)$ with $|\boldsymbol{\alpha}| > 0$.}
	$$
\end{itemize}

We here summarize the main properties of the space $\VK$
(we refer to \cite{BLV-SINUM2018,AVV-IMAJNA2023} for a deeper analysis).
	\begin{itemize}
		\item [\textbf{(P1)}] \textbf{Polynomial inclusion:} $\bP_k(E) \subseteq \VK$;
		\item [\textbf{(P2)}] \textbf{Unisolvency:} the triplet $\big(\E,\: \VK, \: \boldsymbol{D_{V}} \big)$ is a finite element in the sense of Ciarlet;
		\item  [\textbf{(P3)}] \textbf{Polynomial projections:}
		the DoFs $\boldsymbol{D_{V}}$ allow us to compute the following linear operators:
		\[
\PikKb  \colon \VK \to \bP_{k}(\E)\,, \quad
\PikMoneKb  \colon \boldsymbol{\nabla} \VK \to [\P_{k-1}(\E)]^{2 \times 2}\,, \quad 
\iota \colon \div \VK \to \P_{k-1}(\E) \,.
		\] 
	\end{itemize}

The global velocity space is defined by gluing the local spaces and using the boundary conditions, with the obvious
associated set of global DoFs:
\begin{equation*}
	\Vh:= \big\{\bv_h \in \bV \colon \: {\bv_h}|_{E} \in \VK \quad \text{for all $\E \in \Omega_h$} \big\}\,.
\end{equation*}
The discrete  space for the pressures is given by the piecewise polynomial functions of degree $k-1$, i.e.,
\begin{equation}
	\label{eq:global:pressure}
\Qh := \big\{q_h \in Q \colon \: {q_h}|_{{\E}} \in \P_{k-1}(\E) \quad \text{for all $\E \in \Omega_h$} \big\}\,.
\end{equation}
The pair of spaces $(\Vh, \, \Qh)$ is inf-sup stable~\cite{BLV-SINUM2018}, i.e., there exists $\widetilde{\beta}>0$ such that
\begin{equation}\label{eq:infsup}
	\sup_{\bv_h \in \Vh \setminus \{\0\}} \frac{b(\bv_h, q_h)}{\Vert \bv_h\Vert_{\bV}} \geq \widetilde{\beta} \|q_h\|_{Q}
	\qquad \text{for any $q_h \in \Qh$.}
\end{equation}
Next, let us introduce the discrete kernel
\begin{equation*}
	\Zh := \big\{ \bv_h \in \Vh  \: \colon \:   b(\bv_h, q_h) = 0 \quad \text{for all $q_h \in \Qh$} \big\},
\end{equation*}
then recalling $(ii)$ in the definition of the space $\VK$~\eqref{eq:localVirtual-Velo} and the definition \eqref{eq:global:pressure}, the following kernel inclusion holds $	\Zh \subseteq \bZ$. Therefore, the functions belonging to the discrete kernel are exactly divergence-free.

\subsection{Virtual element space for the phase variable}\label{subsec:space:phase}
In this subsection, we outline an overview  of the \emph{enhanced} $C^1$-conforming VE space~\cite{BM-CMAME2013,CM-CAMWA2016,ABSV-SINUM2016}. For every polygon $\E\in\O_h$  and  any integer $\ell\ge 2$, we consider the index $\widehat{\ell}:=\max\{\ell,3\}$ 
and the following finite  dimensional space

\begin{equation}
	\label{eq:localVirtial-phase}
	\begin{aligned}
\HK
 := \biggl\{ \phi_h\in C^1(\overline{\E})  \: \colon \: 
(i)& \,\,  \Delta^2\phi_h\in\P_{\ell}(\E) \: , \\
(ii) &		\,\,  {\phi_h}|_{e} \in \P_{\widehat{\ell}}(e) \quad  \text{and} \quad \partial_{\bn_e} \phi_h \in\P_{\ell-1}(e)  \,\,\, \, \forall e \in \dE \: , 
\\
(iii) & \,\,   \big(\phi_h-\PiK \phi_h, \, \widehat{q}_{\ell} \big)_{\E}=0 \,\,\,\,  \text{$\forall \widehat{q}_{\ell} \in \widehat{\P}_{\ell \setminus \ell-4}(E)$}
\,\,\biggr\} \,.
\end{aligned}
\end{equation}
Next, for  $\phi_{h}\in \HK$, we introduce the following set
of linear operators $\boldsymbol{D_{H}}$:
\begin{itemize}
	\item $\DXu:$ the values of $\phi_{h}(\ba_i)$ and $h_{\ba_i}\nabla \phi_{h}(\ba_i)$  for all vertex $\ba_i$ of the polygon $\E$, where $h_{\ba_i}$ denotes the average of the diameters of all surrounding elements with $\ba_i$ as vertex;
    \item $\DXd:$ for $\ell \ge 3$, the values of $h_e\partial_{\bn_{e}}\phi_h$ at $\ell-2$ distinct internal points of every edge $e \subset \partial \E$;
  \item $\DXt:$ for $\ell \ge 4$, the values of $\phi_h$ at $\ell-3$ distinct internal points of every edge $e \subset \partial \E$;
	\item $\DXc:$ for $\ell \geq 4$, the  internal moments:
	$$ |E|^{-1} (m_{\boldsymbol{\alpha}},  \phi_h)_{0,\E}\,
	\quad \forall m_{\boldsymbol{\alpha}} \in \M_{\ell-4}(\E).$$
\end{itemize}
We now summarize the main properties of the space $\HK$. We refer to \cite{CM-CAMWA2016,ASVV-2022} for further details.
\begin{itemize}
	\item [\textbf{(P4)}] \textbf{Polynomial inclusion:} $\P_{\ell}(E) \subseteq \HK$;
	\item [\textbf{(P5)}] \textbf{Unisolvency:} the triplet $\big(\E,\: \HK, \: \boldsymbol{D_{H}} \big)$ is a finite element in the sense of Ciarlet;
	\item  [\textbf{(P6)}] \textbf{Polynomial projections:}
	the DoFs $\boldsymbol{D_{H}}$ allow us to compute the following linear operators:
	\[
	\begin{aligned}
		\boldsymbol{\Pi}^{0,\ell-2}_{\E} &: \bD^2 \HK \rightarrow [\P_{\ell-2}(\E)]^{2 \times 2}, \qquad	\Pi^{0,\ell-2}_{\E} : \Delta \HK \rightarrow \P_{\ell-2}(\E), \\
		\boldsymbol{\Pi}^{0,\ell-1}_{\E} &: \nabla \HK \rightarrow \bP_{\ell-1}(\E), \qquad\qquad \:	\Pi^{0,\ell}_{\E}   : \HK \rightarrow \P_{\ell}(\E).
	\end{aligned}
	\]
\end{itemize}

Now, for every decomposition $\O_h$ of $\O$ into polygons $\E$ and for any $\ell \geq 2$, we define
the global virtual space to  the numerical approximation of the phase field, as follows:
\begin{equation*}
	\Hh:=\big\{\phi_h\in H  \: \colon \:  \: \phi_h|_{\E}\in \HK  \quad \text{for all $\E \in \Omega_h$} \big\}.
\end{equation*}

\subsection{Virtual element forms}\label{subsec:virtual:forms}
In this subsection we introduce discrete approximations of the continuous forms defined in Section~\ref{sec:cont_problem}, that are computable 
which are computable using only the  DoFs $\boldsymbol{D_{V}}$ and $\boldsymbol{D_{H}}$  (cf. property \textbf{(P3)} and Property \textbf{(P6)}).
We begin with the construction of the VEM stabilizing forms.

\paragraph{Stabilizing forms} 
We consider arbitrary symmetric positive definite bilinear forms 
\[
\begin{aligned}
s_{0,F}^{\E}(\cdot, \cdot) \,, 
s_{\bnabla}^{\E}(\cdot, \cdot) &\colon 
\VK \times \VK \to \R \,,  &\qquad
s_0^{\E}(\cdot, \cdot) \,, s_{\bD}^{\E}(\cdot, \cdot) &\colon 
\HK \times \HK \to \R
\end{aligned}
\]
satisfying  the following equivalences
\begin{equation}\label{term:stab:SK}
\begin{aligned}
s_{0,F}^{\E}(\bv_h,\bv_h)  &\approx (\bv_h,\bv_h)_E  
& \qquad
s_{\bnabla}^{\E}(\bv_h,\bv_h)  &\approx 
 (\bnabla \bv_h,\bnabla \bv_h)_E 
& \qquad \forall \bv_h \in  \mathrm{Ker}(\PikKb)\,,
\\
s_{0}^{\E}(\phi_h,\phi_h)   &\approx (\phi_h,\phi_h)_E 
& \qquad
s_{\bD}^{\E}(\phi_h,\phi_h) &\approx 
 (\bD^2 \phi_h, \bD^2 \phi_h)_E  
& \qquad \forall \phi_h \in  \mathrm{Ker}(\PikK)\,.
\end{aligned}
\end{equation}
Many examples of stabilizing forms can be found in the VEM literature.
Under  Assumption \ref{Assump:mesh}, we choose the forms as the Euclidean scalar product associated with the DoFs, scaled by  suitable scaling factors. We refer to \cite{BLV-SINUM2018,ASVV-2022,AVV-IMAJNA2023}, for further details.

Next, by using the above stabilizing forms we proceed to define the first-, second- and fourth-order bilinear forms as follows: 

\paragraph{Discrete first-order forms} 
\begin{align*}
m_{F,h}^{\E}(\bu_h,\bv_h)&:=
 \big(\PikKb \bu_h,\PikKb \bv_h\big)_E +s^{\E}_{0,F}\big(({\rm I}-\PikKb) \bu_h,({\rm I}-\PikKb) \bv_h\big)\,, \\
m_{h}^{\E}(\varphi_h,\phi_h)&:=
 \big(\PikK \varphi_h,\PikK \phi_h\big)_E +s_0^{\E}\big(({\rm I}-\PikK) \varphi_h,({\rm I}-\PikK) \phi_h\big)\,. 
\end{align*}

\paragraph{Discrete second- and fourth-order forms:} 
\begin{align*}
	a_{\bnabla,h}^{\E}(\bu_h,\bv_h)&:=
	\big( \PikMoneKb  \bnabla \bu_h,\PikMoneKb  \bnabla  \bv_h \big)_{\E} + s^{\E}_{\bnabla}\big((\boldsymbol{{\rm I}}-\PinablaKb) \bu_h,(\boldsymbol{{\rm I}}-\PinablaKb) \bv_h\big)\,,
	\\ 
	a_{\bD,h}^{\E}(\varphi_h,\phi_h)&:=
 \big(	\boldsymbol{\Pi}^{0,\ell-2}_{\E}\bD^2 \varphi_h, 	\boldsymbol{\Pi}^{0,\ell-2}_{\E}\bD^2 \phi_h\big)_{\E} +s_{\bD}^{\E}\big(({\rm I}-\PiK) \varphi_h,({\rm I}-\PiK) \phi_h\big)\,.	
\end{align*}
Next, we present the construction of trilinear and semi-linear forms.
\paragraph{Discrete convective forms:} 
\begin{align*}%
\widetilde{c}_{F,h}^{\E}(\bw_h; \, \bu_h, \bv_h) &:= 
 \bigl((\PikMoneKb  \bnabla \bu_h )  \PikKb \bw_h , \PikKb \bv_h \big)_{\E}\,,
	\\
\widetilde{c}_{h}^E(\bu_h; \, \varphi_h, \phi_h) &:=  
\bigl( \PikKb \bu_h \cdot \nabla \PikK \varphi_h , \PikK \phi_h \bigr)_{\E} +  \big( (\bu_h  \cdot \bn_{\E}) ({\rm I}-\PikK) \varphi_h , \PikK \phi_h\big)_{\partial\E}\,,
\end{align*}
and their skew-symmetric formulations 
\begin{align*}
	c_{F,h}^{\E}
	(\bw_h; \, \bu_h,  \bv_h) &:= \frac{1}{2} \bigl(\widetilde{c}_{F,h}^E(\bw_h; \, \bu_h, \bv_h) - \widetilde{c}_{F,h}^E(\bw_h; \, \bv_h, \bu_h) \bigr) \,, 
	\\
	c_{h}^{\E}(\bu_h; \, \varphi_h,  \phi_h) &:= \frac{1}{2} \bigl (\widetilde{c}_{h}^E(\bu_h; \, \varphi_h, \phi_h) - \widetilde{c}_{h}^E(\bu_h; \, \phi_h, \varphi_h) \bigr) \,. 
\end{align*}

\paragraph{Discrete stress and semi-linear forms:} 
 \begin{align*}
d_{h}^{\E}(\varphi_h;\phi_h,\bv_h)&:=  \big( \PiMtwoK \Delta \varphi_h  \PikMoneKp \nabla\phi_h ,\PikKb \bv_h\bigr)_{\E}\,, \\ 
r_h^{\E}(\eta_h;\varphi_h,\phi_h)&:=  \big( f'(\PikK \eta_h) \PikMoneKp  \nabla \varphi_h ,  \PikMoneKp  \nabla \phi_h \big)_{\E}\,.  
\end{align*}
 
\paragraph{Discrete load term:}  we consider the following approximation for the load term
\begin{equation*}
(\bg_h, \bv_h)_{\E} := (\boldsymbol{\Pi}_E^{0,k} \bg, \bv_h)_{\E} =  ( \bg, \boldsymbol{\Pi}_E^{0,k} \bv_h)_{\E}.
\end{equation*}

We define the associated global forms in the usual way, by summing the local forms on all mesh elements. Moreover, we recall that all the forms defined above are computable using the degrees of freedom $\boldsymbol{D_{V}}$ and $\boldsymbol{D_{H}}$.
In addition we observe that
the mixed term $b(\bv_h, q_h)$ can be exactly computed by the DoFs for any $\bv_h \in \Vh$ and $q_h \in \Qh$ (cf. properties $\mathbf{(P3)}$).
With a slight abuse of notation, in the following we extend the definition of discrete forms to sufficiently regular functions.

In next results we summarize the main properties of the forms defined above.

\begin{proposition}
[Properties of the discrete forms]
\label{lemma:bound:ch}
Let Assumption~\ref{Assump:mesh} hold. 
Then the following are satisfied:
\begin{enumerate}[label=\roman*)]
\item the forms $m_{F, h}(\cdot, \cdot)$, $a_{\bnabla, h}(\cdot, \cdot)$
 are continuous and coercive, i.e. for all $\bu_h$, $\bv_h \in \Vh$ 
\[
\begin{aligned}
m_{F, h}(\bu_h, \bv_h) &\lesssim \|\bu_h\|_{0, \O} \, \|\bv_h\|_{0, \O} \,,
& \quad
\|\bv_h\|_{0, \O}^2 & \lesssim m_{F, h}(\bv_h, \bv_h)  \,,
\\
a_{\bnabla, h}(\bu_h, \bv_h) &\lesssim \|\bu_h\|_{\bV} \, \|\bv_h\|_{\bV} \,,
& \quad
\|\bv_h\|_{\bV}^2 & \lesssim a_{\bnabla, h}(\bv_h, \bv_h) \,.
\end{aligned}
\] 
\item the forms $m_{h}(\cdot, \cdot)$, $a_{\bD, h}(\cdot, \cdot)$
 are continuous and coercive, i.e. for all $\varphi_h$, $\phi_h \in \Hh$ 
\[
\begin{aligned}
m_{h}(\varphi_h, \phi_h) &\lesssim \|\varphi_h\|_{0, \O} \, \|\phi_h\|_{0, \O} \,,
& \quad
\|\phi_h\|_{0, \O}^2 & \lesssim m_{h}(\phi_h, \phi_h)  \,,
\\
a_{\bD, h}(\varphi_h, \phi_h) &\lesssim \|\varphi_h\|_{H} \, \|\phi_h\|_{H} \,,
& \quad
\|\phi_h\|_{H}^2 & \lesssim a_{\bnabla, h}(\phi_h, \phi_h)\,.
\end{aligned}
\]  
\item the  form $c_{F,h}(\cdot; \cdot, \cdot)$ is skew-symmetric and for any $\bw$, $\bu$, $\bv \in \bV$ the following bounds hold
\begin{align}
c_{F,h}(\bw;\bu,\bv) &\lesssim  \|\bw\|_{\bV}  \|\bu\|_{\bV}  \|\bv\|_{\bV}\,,  \label{boundcF:one} \\ 
c_{F,h}(\bw;\bu,\bv)& \lesssim \, \|\bw\|^{\frac{1}{2}}_{0,\O}\|\nabla \bw\|^{\frac{1}{2}}_{0,\O} \|\bu\|_{\bV} \|\bv\|_{\bV} \,. \label{boundcF:two} 
\end{align}	
\item the  form $c_{h}(\cdot; \cdot, \cdot)$ is skew-symmetric and the following bounds hold
\begin{align}
c_h(\bv; \varphi,\phi) &\lesssim \|\bv\|_{\bV} \|\varphi\|_{H} \|\phi\|_{H}, & \qquad &\text{$\forall \bv \in \bV$, $\forall \varphi,\phi\in H$,} 
\label{c-bound-1}
\\ 
c_h(\bv;\varphi,\phi) &\lesssim \|\varphi\|_{W^{1,\infty}(\O)}\|\bv\|_{0,\O}\|\phi\|_{1,\O}  
&\qquad &\text{$\forall \bv \in \bZ$,  $\forall \varphi \in W^{1,\infty}(\O)$, $\forall \phi\in H$.}	 
\label{c-bound-2}
\end{align}	
\item the form $d_{h}(\cdot; \cdot, \cdot)$
 is continuous i.e. for all $\bv \in \bV$, and for all $\varphi$, $\phi \in H$ 
\[
\begin{aligned}
d_{h}(\varphi; \phi, \bv) & \lesssim \|\varphi\|_{H} \, \|\phi\|_{H} \,
\|\bv\|_{\bV} 
\,.
\end{aligned}
\]  
\end{enumerate}
\end{proposition}

\begin{proof}
The continuity and coercivity properties stated in items $i)$ and $ii)$ can be established by standard arguments; see, e.g., \cite{BLV-SINUM2018,ABSV-SINUM2016}.
Bound \eqref{boundcF:one} in item $iii)$  was proved in \cite[Proposition 3.3]{BLV-SINUM2018}, whereas  bound \eqref{boundcF:two}  can be obtained by extending the arguments presented in \cite[Lemma 4.1]{VK2023} to the high-order case.

Let us now focus on item $iv)$  and prove bound \eqref{c-bound-1}. By definition of the form $c_h(\cdot;\cdot,\cdot)$ we have 
\begin{equation}\label{eq:c-bound}
	\begin{split}
	c_h(\bv;\varphi,\phi) & = \frac{1}{2} \sum_{\E \in \O_h}  \bigl (\widetilde{c}_{h}^E(\bv; \, \varphi, \phi) - \widetilde{c}_{h}^E(\bv; \, \phi, \varphi) \bigr)=:\frac{1}{2} \sum_{\E \in \O_h} (\mu^{\E}_1 - \mu^{\E}_2 ).
	\end{split}
\end{equation}
Applying  the divergence theorem, simple computations yield
\begin{equation}\label{key-c-bound}
	\begin{split}
		\mu^{\E}_1  & =  \bigl(\PikKb \bv \cdot \nabla \PikK \varphi , \PikK \phi\bigr)_{\E} +   \bigl(( \bv  \cdot \bn_{\E}) ({\rm I}-\PikK) \varphi , \PikK \phi\bigr)_{ \dE} 	\\
		& = \bigl( \PikKb \bv \cdot \nabla \PikK \varphi, \PikK \phi   \bigr)_{\E}+ \bigl( {\rm div}(\bv ({\rm I}-\PikK) \varphi  \PikK \phi), 1\bigr)_{E}	\\
		& = \bigl( \PikKb \bv \cdot \nabla \PikK \varphi, \PikK \phi   \bigr)_{\E}+ 
\bigl( ({\rm div}\bv) ({\rm I}-\PikK) \varphi , \PikK \phi \bigr)_{E}+
\\
& \quad + 
\bigl( \bv \cdot \nabla ({\rm I}-\PikK) \varphi , \PikK \phi \bigr)_{E} +
\bigl( \bv ({\rm I}-\PikK) \varphi , \nabla  \PikK \phi \bigr)_{E} 
\\
& = 
\bigl( \bv \cdot \nabla \varphi, \PikK \phi   \bigr)_{\E} + 
\bigl( (\PikKb - {\rm I}) \bv \cdot \nabla \PikK \varphi, \PikK \phi   \bigr)_{\E}+ \\
& \qquad +
\bigl( ({\rm div}\bv) ({\rm I}-\PikK) \varphi , \PikK \phi \bigr)_{E}+
\bigl( \bv ({\rm I}-\PikK) \varphi , \nabla  \PikK \phi \bigr)_{E}.
	\end{split}
\end{equation}
Analogously for the term $\mu^E_2$ we infer
\begin{equation}\label{key-c-bound2}
	\begin{split}
	\mu^{\E}_2  &=  \bigl( \bv \cdot \nabla \phi, \PikK \varphi   \bigr)_{\E} + 
\bigl( (\PikKb - I) \bv \cdot \nabla \PikK \phi, \PikK \varphi   \bigr)_{\E}+ \\
& \qquad +
\bigl( ({\rm div}\bv) ({\rm I}-\PikK) \phi , \PikK \varphi \bigr)_{E}+
\bigl( \bv ({\rm I}-\PikK) \phi , \nabla  \PikK \varphi \bigr)_{E} \,.
\end{split}
\end{equation}

Thus, using the H\"older inequality and the continuity of the $L^2$-projection (cf. Proposition \ref{approx:prop:L2}),  each local term $\mu^{\E}_{1}-\mu^{\E}_{2}$ can be bounded as follows
\begin{equation*}
\begin{aligned}
\mu^{\E}_{1} - \mu^{\E}_{2} &\lesssim 
\|\bv \|_{0,E} 
\|\varphi\|_{W^{1,4}(E)}
\|\phi\|_{L^{4}(E)}
+
\|\bv\|_{0,E} 
\|\varphi\|_{L^{4}(E)}
\|\phi\|_{W^{1,4}(E)}
+
|\bv|_{1,E} 
\|\varphi\|_{L^{4}(E)}
\|\phi\|_{L^{4}(E)}
\\
& \lesssim 
\|\bv\|_{1,E} 
\|\varphi\|_{W^{1,4}(E)}
\|\phi\|_{W^{1,4}(E)}.
\end{aligned} 
\end{equation*}
Thus,
applying the H\"older inequality (for sequences) along with the Sobolev inclusions we obtain the desired result~\eqref{c-bound-1}.

Next, we prove bound \eqref{c-bound-2}. To do that, we apply H\"older inequality in equation \eqref{key-c-bound} and \eqref{key-c-bound2} with $\bv \in \bZ$ and $\varphi \in W^{1,\infty}(\O)$, obtaining  
\begin{equation*}
\begin{aligned}
\mu^{\E}_{1} - \mu^{\E}_{2}  &\lesssim 
\|\bv \|_{0,E} 
\|\varphi\|_{W^{1,\infty}(E)}
\|\phi\|_{0,E}
+
\|\bv\|_{0,E} 
\|\varphi\|_{L^{\infty}(E)}
\|\phi\|_{1,E}
 \lesssim 
\|\bv\|_{0,E} 
\|\varphi\|_{W^{1,\infty}(E)}
\|\phi\|_{1,E}.
\end{aligned} 
\end{equation*}
The desired result \eqref{c-bound-2} follows by inserting the above estimate in \eqref{eq:c-bound} and by applying the H\"older inequality (for sequences). 

We finally prove item $v)$. Employing the H\"older inequality, the continuity of the $L^2$-projection (cf. Proposition \ref{approx:prop:L2}), the  H\"older inequality for sequence and Sobolev inclusions we have
\[
\begin{aligned}
d_h(\varphi; \phi; \bv) & = 
\sum_{E \in \Omega_h} d_h^E(\varphi; \phi; \bv) 
\leq 
\sum_{E \in \Omega_h} \| \PiMtwoK \Delta \varphi \|_{0,E} \, \|\PikMoneKp \nabla\phi \|_{L^4(E)} \, \|\PikKb \bv\|_{L^4(E)}
\\
& \lesssim 
\sum_{E \in \Omega_h} \|  \Delta \varphi \|_{0,E} \, \| \nabla\phi \|_{L^4(E)} \, \|\bv\|_{L^4(E)}
 \lesssim
\|  \varphi \|_{H} \, \| \phi \|_{H} \, \|\bv\|_{\bV} \,.
\end{aligned}
\]
\end{proof}

We finish this subsection with following remark regarding the VEM presented in this work.
\begin{remark}\label{remark:properties}
We emphasize that the construction of the convective trilinear form  $c_{h}^{\E}(\cdot; \, \cdot,  \cdot)$ is non-standard (see \cite{BDLV-M2AN2021}, for a similar construction). This discretization is specifically designed to ensure that the discrete schemes satisfy the mass conservation property (see Theorems~\ref{theo:property:semi} and \ref{theo:property:fully} below), which is typical for CH systems. While in this work, it is possible to design a numerical scheme that meets the mass conservation requirement, achieving energy law within the same framework remains a significant challenge. 
\end{remark}

%% file: 04_discrete_formulations.tex
\section{Discrete formulations and their properties}
\label{sec:discrete:formulations}
In this section we shall to present  semi- and fully-discrete formulations for the problems~\eqref{form1} and~\eqref{form:kernel}. Additionally, we provide some properties of discrete schemes including mass conservation and energy bounds. 

\subsection{Semi-discrete formulation}
The semi-discrete version of problem \eqref{form1} reads as: 	find 
\[
\bu_h \in    L^{\infty}(0,T;\,  \bL^2(\Omega))  \cap L^{2}(0,T; \Vh) \,, 
\quad
\widehat{p}_h \in  L^{2}(0,T; \Qh)\,,  
\quad
\varphi_h \in    L^{\infty}(0,T;\,  L^2(\Omega))  \cap L^{2}(0,T; \Hh)
\]
such that for a.e. $t \in (0, T]$ and for all $(\bv_h,q_h,\phi_h) \in  \Vh \times \Qh \times \Hh$ the following holds
\begin{equation}\label{disc:form1}
	\left\{
\begin{aligned}
		m_{F,h}(\partial_t\bu_h,\bv_h) + \nu a_{\bnabla,h}(\bu_h, \bv_h) +  c_{F,h}(\bu_h;\bu_h,\bv_h) -b(\bv_h,\widehat{p}_h)  + \lambda  d_h(\varphi_h;\varphi_h, \bv_h) &= (\bg_h,\bv_h)_{\O}\,,\\
			b(\bu_h,q_h)  &= 0\,,\\
		m_{h}(\partial_t\varphi_h, \phi_h) +\gamma a_{\bD,h}(\varphi_h, \phi_h )  +c_{h}(\bu_h;\varphi_h,\phi_h)
		+ \gamma \varepsilon^{-2} r_h(\varphi_h;\varphi_h,\phi_h)  &= 0 \,.
\end{aligned}
\right.
\end{equation}
 The initial conditions are given by $\bu_h(\cdot, 0 ) = \bu_{h,0}$ and $\varphi_h(\cdot, 0) = \varphi_{h,0} $, where $(\bu_{h,0}, \varphi_{h,0})$ can be chosen as the interpolant function of the initial value $(\bu_{0}, \varphi_{0})$ (cf.~\eqref{initial:data})  in the space $\Vh \times \Hh$ (see Proposition~\ref{approx:virtual:velo:phase}).

As in the continuous setting, recalling that $\bZ_h \subset \bZ$,  problem~\eqref{disc:form1} admits the equivalent kernel formulation: find
\[
\bu_h \in    L^{\infty}(0,T;\,  \bL^2(\Omega))  \cap L^{2}(0,T; \Zh) \,, 
\quad
\varphi_h \in    L^{\infty}(0,T;\,  L^2(\Omega))  \cap L^{2}(0,T; \Hh)
\]
such that for a.e. $t \in (0, T]$ and for all $(\bv_h,\phi_h) \in  \Zh \times  \Hh$ the following holds
\begin{equation}\label{disc:form:kernel}
	\left\{
	\begin{aligned}
	m_{F,h}(\partial_t\bu_h,\bv_h) + \nu a_{\bnabla,h}(\bu_h, \bv_h) +  c_{F,h}(\bu_h;\bu_h,\bv_h)  + \lambda  d_h(\varphi_h;\varphi_h, \bv_h) &= (\bg_h,\bv_h)_{\O},\\
m_{h}(\partial_t\varphi_h, \phi_h) +\gamma a_{\bD,h}(\varphi_h, \phi_h )  +c_{h}(\bu_h;\varphi_h,\phi_h)
+ \gamma \varepsilon^{-2} r_h(\varphi_h;\varphi_h,\phi_h) & = 0,		
\end{aligned}
\right.
\end{equation}
with initial conditions given as before, i.e., $(\bu_h(\cdot, 0 ),\varphi_h(\cdot, 0))= (\bu_{h,0}, \varphi_{h,0})$. 

The subsequent analysis will be performed under the following regularity assumption on the semi-discrete phase solution $\varphi_h \in \Hh$. The regularity assumption is standard and well accepted in the analysis of the CH systems, see for instance~\cite{ABSV-SINUM2016,EF89}. 
		\begin{assumption}
		[Regularity assumption on the semi-discrete phase solution]\label{Assump:disc:phase}
			The semi-discrete problem~\eqref{disc:form:kernel} admits a unique solution $\varphi_h$ satisfying, for a.e. $t \in (0,T]$ 
			\[
			\|\varphi_h(\cdot, t)\|_{W^{1,\infty}(\O)} \leq \widetilde{C}_T,
			\]
			with a constant $\widetilde{C}_T$ independent of $h$, but depending on the final time $T$. 
	\end{assumption}
Recalling Proposition \ref{lemma:bound:ch}, from now on, we will adopt the following notation:
\begin{equation}\label{equiv:norms}
\begin{aligned}
|||\bv_h|||_{0,\O}^2 & :=
m_{F,h}(\bv_h,\bv_h) 
\approx \|\bv_h \|_{0,\O}^2
\,, &\quad 
|||\bv_h|||_{1,\O}^2 &:=
a_{\boldsymbol{\nabla},h}(\bv_h,\bv_h) 
\approx \|\bv_h \|_{1,\O}^2
\,,
&\quad &\text{$\forall \bv_h \in \Vh$,}
\\
|||\phi_h|||_{0,\O}^2 &:=
m_h(\phi_h,\phi_h) 
\approx \|\phi_h\|_{0,\O}^2  \,,
 &\quad
  |||\phi_h|||_{2,\O}^2 &:=
a_{\bD,h}(\phi_h,\phi_h) \approx	\|\phi_h \|_{2,\O}^2 \,,
&\quad &\text{$\forall \phi_h \in \Hh$.}
\end{aligned}	
\end{equation}
We now recall the continuous version of Gr\"{o}nwall’s inequality, which plays a key role in the derivation of fundamental properties and error estimates for the semi-discrete virtual scheme \eqref{disc:form:kernel}.

\begin{lemma}[Gr\"{o}nwall’s Lemma]\label{cont:gronwall}
	Let $g$, $q$, and $r$ be nonnegative integrable functions on $[0, T]$ and suppose that $g$ satisfies
	\[
	g(t) \leq q(t) + \int_0^t r(s) g(s) \, ds \quad \text{for all } t \in (0, T).
	\]
	Then,
	\[
	g(t) \leq q(t) + \int_0^t q(s) r(s) \exp\left( \int_s^t r(\sigma) \, d\sigma \right) d s \quad \text{for a.e. $t \in (0, T)$.}
	\]
\end{lemma}
Moreover, the following useful inequality holds for functions in the space $H$. Its proof follows using integration by part, the boundary conditions,  along with the Cauchy-Schwarz and Young inequalities.  For all $\kappa >0$ there exists a constant 
$ C(\kappa) >0 $ 
such that it holds
\begin{equation}\label{ineq:Hi}
	|\phi|_{1,\Omega}^2 \leq \kappa |\phi|_{2,\O}^2 + C(\kappa) \|\phi\|_{0,\Omega}^2 \qquad \forall  \phi \in H.
\end{equation}
In the next result we summarize the properties of our semi-discrete scheme introduced in~\eqref{disc:form:kernel}.

\begin{theorem}[Semi-discrete conservation law and energy stability]
\label{theo:property:semi}
Let Assumptions \ref{Assump:mesh} and \ref{Assump:disc:phase} hold.
Let $(k, \ell)$  in \eqref{eq:localVirtual-Velo} and \eqref{eq:localVirtial-phase}  be such that $k \ge \ell - 1$. 
Suppose that the initial values  $(\bu_{h,0}, \varphi_{h,0})$ of problem \eqref{disc:form:kernel} are uniformly bounded in $L^2$.
Then, the solution $(\bu_h, \varphi_h)$ of problem~\eqref{disc:form:kernel} satisfies the following:
\begin{align}
&(\varphi_h(\bx, t),1)_{\O} = (\varphi_{h,0}(\bx,t), 1)_{\O}\qquad \text{for a.e. $t \in (0,T)$,} \label{mass:disc:conserv}\\ 
&\| \bu_h\|_{L^{\infty}(0,T;\bL^2(\O))}  +  
\| \varphi_h\|_{L^{\infty}(0,T; L^2(\O))} +
\|\bu_h\|_{L^2(0,T;\bV)} +  
\|\varphi_h\|_{L^2(0,T; H)}
\leq \widetilde{C}_{{\rm stab}}, \label{energy:bound}
\end{align}		
where  $\widetilde{C}_{{\rm stab}}$ is a positive constant depending  on the data, on the physical parameters, on the constant $\widetilde{C}_T$ in Assumption \ref{Assump:disc:phase} and is independent of $h$.
\end{theorem}

\begin{proof}
In order to prove the discrete mass conservation property~\eqref{mass:disc:conserv},  we choose $\phi_h=1 \in \Hh$  as test function in \eqref{disc:form:kernel}. Exploiting the definition of the discrete forms and the polynomial consistency of $m_h(\cdot,\cdot)$ and $a_{\bD,h}(\cdot, \cdot)$ we obtain 
\begin{equation*}
m_{h}(\partial_t\varphi_h, 1) = m(\partial_t\varphi_h, 1), \quad \text{and} \quad  a_{\bD,h}(\varphi_h, 1 ) =  r_h(\varphi_h;\varphi_h,1)  = 0 \,.
\end{equation*}
The remaining trilinear form $c_h(\cdot;\cdot, \cdot)$ 
needs more analysis. 
Employing the computations in \eqref{eq:c-bound}, \eqref{key-c-bound} and \eqref{key-c-bound2} with $\bv = \bu_h$, $\varphi=\varphi_h$ and $\phi=1$,
and using the definition of $\PikK$ and $\PikKb$ (with $k \geq \ell-1$),
we infer
\begin{equation*}
\begin{split}
c_{h}(\bu_h;\varphi_h,1) &= 
 \frac{1}{2}\sum_{\E \in \O_h} \Big[	\bigl( \bu_h  , \nabla \PikK \varphi_h \bigr)_{\E} +
 \big(\bu_h ,\nabla   ({\rm I}-\PikK) \varphi_h \big)_{ \E} \Big]\\
  &= 	\sum_{\E \in \O_h} (\bu_h,  \nabla \varphi_h)_{\E}  =  -(\div \bu_h, \varphi_h)_{\O} = 0 \,,
\end{split}		
\end{equation*} 
where we in the last step we have integrated by parts and used the fact that $\bu_h \in \Zh$. The desired result follows after integrate over $(0,t)$ the second equation of \eqref{disc:form:kernel}.
	
Now, in order to establish bound \eqref{energy:bound}, we choose  $(\bv_h, \phi_h) =(\bu_h,\varphi_h)$ in  \eqref{disc:form:kernel}, using the skew-symmetric property of $c_{F,h}(\cdot;\cdot,\cdot)$
and $c_{h}(\cdot;\cdot,\cdot)$ (cf. Proposition \ref{lemma:bound:ch}) and recalling definitions \eqref{equiv:norms}, we have that 
\begin{equation}\label{energy:esti:1}
	\begin{split}
\frac{1}{2} \partial_t ||| \bu_h|||^2_{0,\O} + \nu |||\bu_h |||^2_{1,\O}	  &= (\bg_h,\bu_h)_{\O}-\lambda d_h(\varphi_h;\varphi_h,\bu_h) \,,\\		
\frac{1}{2} \partial_t ||| \varphi_h|||^2_{0,\O} + \gamma |||\varphi_h |||^2_{2,\O} &= - \gamma \varepsilon^{-2} r_h(\varphi_h;\varphi_h,\varphi_h)\,. \\		
	\end{split}	
\end{equation}
Moreover, since $f'(\eta)= 3\eta^2-1 \geq -1$, for all $\eta$, from the continuity of the $L^2$-projection (cf. Proposition \ref{approx:prop:L2}), we have that 
\begin{equation*}
	\begin{split}
r_h(\varphi_h;\varphi_h,\varphi_h)  
\geq -  \sum_{\E \in \O_h} \|\PikMoneKp  \nabla \varphi_h \|^2_{0,\E}  \geq -|\varphi_h|^2_{1,\O}.		
	\end{split}
\end{equation*}
From the above inequality and property \eqref{ineq:Hi}, it holds
\begin{equation*}
- \gamma \varepsilon^{-2} r_h(\varphi_h;\varphi_h,\varphi_h) \leq \frac{\gamma}{4} |||\varphi_h |||^2_{2,\O} +  C(\gamma, \varepsilon^{-2}) \|\varphi_h\|^2_{0,\O} \,.	
\end{equation*}
On the other hand, employing Assumption \ref{Assump:disc:phase} and the continuity of the $L^2$-projection we have 
\begin{equation*}
\begin{split}
-\lambda d_h(\varphi_h;\varphi_h,\bu_h)  &\leq \lambda  \sum_{\E \in \O_h} \| \PikMoneKp \varphi_h\|_{W^{1,\infty}(\E)} \|\PiMtwoK \Delta \varphi_h\|_{0,\E} \|\PikKb\bu_h\|_{0,\E}\\
 &\leq 	\frac{\gamma}{4} |||\varphi_h |||^2_{2,\O} +  C(\gamma^{-1}, \lambda, \widetilde{C}_T)  \|\bu_h\|^2_{0,\O} \,.	
	\end{split}	
\end{equation*}
Then, by combining the last two bounds and~\eqref{energy:esti:1}, we arrive at
\begin{equation*}\label{energy:esti:4}
\begin{aligned}
\frac{1}{2} \partial_t ||| \bu_h|||^2_{0,\O} + \frac{1}{2} \partial_t ||| \varphi_h|||^2_{0,\O} &+ \frac{\nu}{2}  |||\bu_h |||^2_{1,\O}+\frac{\gamma}{2} |||\varphi_h |||^2_{2,\O} \leq
\\
& \|\bg\|^2_{0,\O}+ C(\gamma^{-1}, \lambda, \widetilde{C}_T) |||\bu_h|||^2_{0,\O} + 
C(\gamma, \varepsilon^{-2})  |||\varphi_h |||^2_{0,\O} \,.		
\end{aligned}
\end{equation*}
Integrating over $(0,t)$, then applying the Gr\"{o}nwall inequality (cf. Lemma~\ref{cont:gronwall}) and using \eqref{equiv:norms},  we have that 
\begin{equation*}
\| \bu_h\|_{0,\O} +  \| \varphi_h\|_{0,\O} +  \|\bu_h \|_{L^2(0,t,\bV)}+\|\varphi_h \|_{L^2(0,t,H)} \leq
  \widetilde{C}_{{\rm stab}}(t) \,. 			
\end{equation*}
where the involved constant $\widetilde{C}_{{\rm stab}}(t)$ is positive and independent of $h$. The desired result follows by taking essential supremum over $[0,T]$ in the above estimate.

\end{proof}

\subsection{Fully-discrete formulation}\label{subsec:fully}
We now introduce the temporal discretization.
We consider a sequence of time steps  $t_n = n \tau$, with $n=0,1,2,\ldots,N$, with time step size $\tau=T/N$. 
The functions involved in the continuous, semi-discrete and fully-discrete schemes at time $t = t_n$ will be denoted by $w(t_n)$, $w_h(t_n)$ and $w_h^n$, respectively. Furthermore, we define $\delta_t$ as an approximation of time
derivative at time $t_n$. For any discrete function $w^n_h$, it is given by
\[
\delta_t w_h^n := \frac{\: w_h^n-w_h^{n-1} \:}{\tau}.\]

Analogous notation will be used for vectorial functions $\bw^n_h$.
Moreover,  given a Hilbert space $W$ endowed with the norm $\|\cdot\|_{W}$, for any $n=1,\dots, N$ we consider the following discrete-in-time norms
\[
\|w\|_{\ell^{2}(0,t_n;W)}
:=  \Big(\tau \sum_{m=1}^{n} \|w^{m}\|_{W}^{2} \Big)^{1/2}\,,
\qquad
\|w\|_{\ell^{\infty}(0,t_n;W)}
:= \max_{0 \le m \le n} \|w^{m}\|_{W} \,.
\]

Thus, we introduce a fully-discrete problem by coupling  the backward Euler 
method with the semi-discrete VE discretization~\eqref{disc:form1}, which reads as, given $(\bu_{h,0}, \varphi_{h,0})$, seek: 
\[
\{\bu_h^n, \widehat{p}_h^n, \varphi_h^n)\}_{n=1}^{N} \subset \Vh \times \Qh \times \Hh\,, 
\]
such that for all $(\bv_h, q_h, \phi_h) \in \Vh \times \Qh \times \Hh$
the following holds 
\begin{equation}\label{fully:disc:schm1}
	\left\{
	\begin{aligned}
m_{F,h}(\delta_t\bu^n_h,\bv_h) + \nu a_{\bnabla,h}(\bu^n_h, \bv_h) +  c_{F,h}(\bu^n_h;\bu^n_h,\bv_h)  
- b(\bv_h,\widehat{p}_h^n)
+ \lambda  d_h(\varphi^n_h;\varphi^n_h, \bv_h) &= (\bg^n_h,\bv_h)_{\O} \,,   \\ 
b(\bu^n_h,q_h)  &= 0\,,\\
m_{h}(\delta_t\varphi^n_h, \phi_h) +\gamma a_{\bD,h}(\varphi^n_h, \phi_h )  +c_{h}(\bu^n_h;\varphi^n_h,\phi_h)
+ \gamma \varepsilon^{-2} r_h(\varphi^n_h;\varphi^n_h,\phi_h)  &= 0 \,. 
	\end{aligned}
	\right.
\end{equation}
Analogously, we define the fully-discrete counterpart of VE discretization~\eqref{disc:form:kernel}, which reads as, given $(\bu_{h,0}, \varphi_{h,0})$, seek 
\[
\{(\bu_h^n, \varphi_h^n)\}_{n=1}^{N} \subset \Zh \times \Hh \,,
\]
such that for all $(\bv_h, \phi_h) \in \Vh\times \Hh$ the following holds
\begin{equation}\label{fully:disc:schm}
	\left\{
	\begin{aligned}
m_{F,h}(\delta_t\bu^n_h,\bv_h) + \nu a_{\bnabla,h}(\bu^n_h, \bv_h) +  c_{F,h}(\bu^n_h;\bu^n_h,\bv_h)  + \lambda  d_h(\varphi^n_h;\varphi^n_h, \bv_h) &= (\bg^n_h,\bv_h)_{\O}\,,\\ 
m_{h}(\delta_t\varphi^n_h, \phi_h) +\gamma a_{\bD,h}(\varphi^n_h, \phi_h )  +c_{h}(\bu^n_h;\varphi^n_h,\phi_h)
+ \gamma \varepsilon^{-2} r_h(\varphi^n_h;\varphi^n_h,\phi_h)  &= 0 \,.
	\end{aligned}
	\right.
\end{equation}
We now introduce an assumption analogous to Assumption~\ref{Assump:disc:phase}.
\begin{assumption}
[Regularity assumption on the fully-discrete phase solution]
\label{Assump:f-disc:phase}
			The fully-discrete problem~\eqref{fully:disc:schm1} admits a unique solution $\{\varphi_h^n\}_{n=1}^N$ satisfying for all $n=1,\dots, N$
			\[
			\|\varphi_h^n\|_{W^{1,\infty}(\O)} \leq \widehat{C}_T \,,
			\]
			for a constant $\widehat{C}_T$ independent of $h$, but depending on the final time $T$. 
	\end{assumption}

We now state the discrete counterpart of the continuous Gr\"{o}nwall inequality
(cf.~Lemma~\ref{cont:gronwall}), which plays a key role in the analysis of the
fully discrete scheme~\eqref{fully:disc:schm}. This discrete version was
originally proved in~\cite{HR-SINUM90}.

\begin{lemma}[Discrete Gr\"{o}nwall inequality]\label{discrete:gronwall}
	Let $D\geq 0$, $a_j$, $b_j$, $c_j$ and $\kappa_j$ be non negative numbers for any integer $j \geq 0$, such that 
	\begin{equation*}
		a_n+ \tau \sum_{j=0}^n b_j \leq \tau \sum_{j=0}^n \kappa_j a_j+ \tau \sum_{j=0}^n c_j+D,  \qquad \text{for $n \geq 0$.}
	\end{equation*}
	Suppose that $\tau \kappa_j <1$ for all $j$, and set $\sigma_j:=(1-\tau \kappa_j)^{-1}$. 
	Then, the following bound holds
	\begin{equation*}
		a_n+ \tau \sum_{j=0}^n b_j \leq \exp \Big( \tau \sum_{j=0}^n \sigma_j \kappa_j \Big) \Big( \tau \sum_{j=0}^n c_j+D\Big) \qquad \text{for $n \geq 0$.}
	\end{equation*}
\end{lemma}
In the next result we summarize the main properties of our fully-discrete scheme \eqref{fully:disc:schm}.
\begin{theorem}[Fully-discrete conservation law and energy stability]
\label{theo:property:fully}
Let Assumptions \ref{Assump:mesh} and \ref{Assump:f-disc:phase} hold.
Let $(k, \ell)$  in \eqref{eq:localVirtual-Velo} and \eqref{eq:localVirtial-phase}  be such that $k \ge \ell - 1$. 
Suppose that the initial values  $(\bu_{h,0}, \varphi_{h,0})$ of problem \eqref{disc:form:kernel} are uniformly bounded in $L^2$.
Then, the solution $\{(\bu_h^n, \varphi_h)\}_{n=1}^N$ of problem~\eqref{fully:disc:schm} 
 satisfies the discrete mass conservation property
\begin{equation}
\label{mass:f-disc:conserv}
(\varphi^n_h,1)_{\O} = ( \varphi_{h,0},1 )_{\O}
\qquad \text{for $n=1,\dots, N$.} 
 \end{equation}	
Moreover, for a sufficiently small time step $\tau$, the following stability estimate holds:
 \begin{align}
 \label{energy:f-disc}
 \| \bu_h\|_{\ell^{\infty}(0,T;\bL^2(\O))}  +
 \| \varphi_h\|_{\ell^{\infty}(0,T; L^2(\O))} +  
 \|\bu_h \|_{\ell^2(0,T;\bV)} + 
 \|\varphi_h\|_{\ell^2(0,T; H)}  \leq 
 \widehat{C}_{{\rm stab}}\,, 
 \end{align}	
 where  $\widehat{C}_{{\rm stab}}$
is a positive constant depending only on the data, on the physical parameters, on the constant $\widehat{C}_T$ in Assumption \ref{Assump:f-disc:phase}   and is independent of $h$ and $\tau$.
\end{theorem}
\begin{proof}
Discrete mass conservation follows from the same arguments as in Theorem~\ref{theo:property:semi}.
Moreover, by testing the first and second equations in~\eqref{fully:disc:schm} with $\bv_h = \bu^n_h$ and  $\phi_h = \varphi^n_h$, respectively, by exploiting the properties of the discrete forms (see Lemma~\ref{lemma:bound:ch}), 
and from the identities
\[
\begin{aligned}
m_{F,h}(\delta_t\bu^n_h,\bu^n_h)& = 
\frac{1}{2 \tau}
\big(|||\bu^n_h|||^2_{0,\O}-|||\bu^{n-1}_h|||^2_{0,\O} + |||\bu^n_h - \bu_h^{n-1}|||^2_{0,\O} \big)
\\
m_{h}(\delta_t\varphi^n_h,\varphi^n_h) &= 
\frac{1}{2 \tau}
\big(|||\varphi^n_h|||^2_{0,\O}-|||\varphi^{n-1}_h|||^2_{0,\O} + |||\varphi^n_h - \varphi^{n-1}_h|||^2_{0,\O} \big)
\end{aligned}
\]
we obtain 
\begin{equation*}
\begin{split}
&\frac{1}{2\tau}	\bigl(|||\bu^n_h|||^2_{0,\O} - |||\bu^{n-1}_h|||^2_{0,\O} \bigr)+ \nu|||\bu^n_h|||^2_{1,\O}
	\leq -\lambda d_h(\varphi^n_h;\varphi^n_h,\bu_h^n)_{0,\O}  + (\bg_h^n,\bu_h^n)_{\O} \,,
	\\
&\frac{1}{2\tau}	\bigl(|||\varphi^n_h|||^2_{0,\O} -|||\varphi^{n-1}_h|||^2_{0,\O}\bigr)+ \gamma |||\varphi^n_h|||^2_{2,\O}  
\leq -\gamma \varepsilon^{-2}r_h(\varphi^n_h;\varphi^n_h,\varphi_h^n)_{\O} \,.
\end{split}
\end{equation*}
Then, by combining the above bounds with arguments analogous to those used in Theorem~\ref{theo:property:semi}, we infer
\[
\begin{aligned}
\frac{1}{2\tau}	\bigl(|||\bu^n_h|||^2_{0,\O} - |||\bu^{n-1}_h|||^2_{0,\O} \bigr) &+ 
\frac{1}{2\tau}	\bigl(|||\varphi^n_h|||^2_{0,\O} -|||\varphi^{n-1}_h|||^2_{0,\O}\bigr)
+ \frac{\nu}{2}  |||\bu_h^n |||^2_{1,\O}+\frac{\gamma}{2} |||\varphi_h^n |||^2_{2,\O} 
\\
& \leq \|\bg^n\|^2_{0,\O}+ C(\gamma^{-1}, \lambda, \widehat{C}_T) |||\bu_h^n|||^2_{0,\O} + 
C(\gamma, \varepsilon^{-2})  |||\varphi_h^n |||^2_{0,\O} \,.		
\end{aligned}
\]
Bound \eqref{energy:f-disc} follows by summing over $n = 1,\ldots, N$,
then applying the discrete Gr\"{o}nwall inequality, provided that the time step $\tau$ is sufficiently small (cf. Lemma~\ref{discrete:gronwall}).

\end{proof}	

%% file: 05_Error_analysis.tex
\section{Error analysis}\label{sec:error:analysis}
This section presents error estimates for the semi-discrete and fully-discrete virtual element schemes (cf. \eqref{disc:form:kernel} and \eqref{fully:disc:schm}). The analysis mainly addresses the spatial discretization, which constitutes the core novelty of the method, with particular emphasis on the coupling of divergence-free and $C^1$-conforming virtual element approaches.
For completeness, some error bounds for the fully-discrete scheme are also provided, while the temporal discretization is treated by standard techniques.

The convergence analysis will be performed under the following regularity assumption on the external force $\bg$, initial data and the  solutions $(\bu, \widehat{p}, \varphi)$ of problem \eqref{form1}.
\begin{assumption}[Regularity assumptions on the data and solution -- error analysis]
\label{assump:reg:add}
	The external force, the initial data,  and the solution
	of problem~\eqref{form1} satisfy
\begin{itemize}
\item $\bg \in L^2(0, T; \bH^{k}(\O))$,  	
      $\bu_0 \in \bH^{k+1}(\O)$,
	  $\varphi_0 \in H^{\ell+1}(\O)$,
\item $\bu \in L^2(0, T; \bH^{k+1}(\O)) \cap  H^1(0, T; \bH^{k}(\O))$,
\item $\varphi \in H^1(0,T; H^{\ell+1}(\O))$,  $f(\varphi) \in 	 L^2(0,T; H^{\ell-2}(\O))$.  
\end{itemize}	
\end{assumption}
In the forthcoming analysis we set $m := \min\{k,\ell-1\}$.

\subsection{Preliminary results}\label{subsec:prelimi} 
First, we recall the   following approximation for the velocity and phase 
virtual element spaces, which can be found, for instance in~\cite{BLV-SINUM2018} and \cite{BM-CMAME2013,CM-CAMWA2016}, respectively. 
\begin{proposition}
[Interpolation errors]
\label{approx:virtual:velo:phase}
Let Assumption  \ref{Assump:mesh} hold.
Then for any  $(\bv, \phi )\in (\bV \cap  \bH^{k+1}(\Omega)) \times   (H \cap  H^{\ell+1}(\Omega))$, 
there exists $(\bv_{I},  \phi_{I})\in\Vh \times \Hh$, such that
	\begin{equation*}
		\begin{aligned}
			\|\bv-\bv_{I}\|_{0,\Omega} + h \|\bv-\bv_{I}\|_{\bV} &\lesssim h^{k+1}|\bv|_{k+1,\Omega} \,,
			\\
			\|\phi-\phi_{I}\|_{0,\Omega} + h |\phi-\phi_{I}|_{1,\Omega} + h^2 \|\phi-\phi_{I}\|_{H}   &\lesssim  h^{\ell+1}|\phi|_{\ell+1,\Omega}  \,.	
		\end{aligned}
	\end{equation*}
	Furthermore, if $\bv \in \bZ$ it holds $\bv_I \in \Zh$. 
\end{proposition}
%
Next lemmas give us the measure of the variational crime in the discretization of  the trilinear forms $c_{F}(\cdot;\cdot,\cdot)$, $c(\cdot;\cdot,\cdot)$ and  $d(\cdot;\cdot,\cdot)$.
\begin{lemma}
\label{lemma:crimen:cF}
Let  Assumption  \ref{Assump:mesh} hold.
Let $\bw\in \bV \cap H^{k+1}(\O)$, then  for all $\bv\in \bV$ it holds that
\begin{equation*}
	\begin{split}
c_F(\bw;\bw,\bv)-c_{F,h}(\bw;\bw,\bv)
&\lesssim h^{k}\big(\|\bw\|_{k,\O}+\|\bw\|_{\bV} +\|\bw\|_{k+1} \big) \|\bw\|_{k+1}\|\bv\|_{\bV}.
	\end{split}
\end{equation*}
\end{lemma}
\begin{proof}
The proof has been established in \cite[Lemma 4.3]{BLV-SINUM2018}.
\end{proof}

\begin{lemma}
\label{lemma:crimen:c}
Let  Assumption  \ref{Assump:mesh} hold.
Let $\bv \in  \bZ \cap \bH^{k+1}(\O)$ and $\varphi\in  H\cap H^{\ell+1}(\O)$, then for all $\phi\in H$ it holds that
\begin{equation*}
c(\bv;\varphi,\phi)-c_h(\bv;\varphi,\phi)
\lesssim h^{m+1}
\bigl(
\|\bv\|_{k+1,\O} \, \|\varphi\|_{H} +  
h \|\bv\|_{k+1, \O} \|\varphi\|_{\ell+1, \O} + 
h \|\bv\|_{\bV} |\varphi|_{\ell+1, \O} 
\bigr)
\|\phi\|_{H}.
\end{equation*}
\end{lemma}

\begin{proof}
Let us set
\begin{equation}
\label{eq:sigma0}
\begin{split}
	c(\bv;\varphi,\phi) - 
	c_h(\bv;\varphi,\phi) & = \frac{1}{2} \sum_{\E \in \O_h} 
	\left[  
	 \bigl( (\bv \cdot \nabla \varphi, \phi)_E - \widetilde{c}_{h}^E(\bv; \, \varphi, \phi) \bigr)  -
	\bigl ( (\bv \cdot \nabla \phi, \varphi)_E - \widetilde{c}_{h}^E(\bv; \, \phi, \varphi) \bigr) \right]
	\\
	& =:\frac{1}{2} \sum_{\E \in \O_h} (\sigma^{\E}_1-  \sigma^{\E}_2 ).
	\end{split}
\end{equation}
Thus, using the derivations in \eqref{key-c-bound} and  \eqref{key-c-bound2}, employing the orthogonality of the $L^2$-projection  and recalling that $\bv \in \bZ$, we obtain
\begin{equation}
\label{eq_sigma1}
\begin{aligned}
\sigma_1^E &=  
(\bv \cdot \nabla \varphi, ({\rm I} - \PikK) \phi)_E
+ \bigl( ({\rm I} - \PikKb) \bv \cdot \nabla \PikK \varphi, \PikK \phi   \bigr)_{\E} + 
\bigl( \bv (\PikK - {\rm I}) \varphi , \nabla  \PikK \phi \bigr)_{E}
\\
&= (({\rm I} - \PikK) (\bv \cdot \nabla \varphi), ({\rm I} - \PikK) \phi)_E
+ \bigl( ({\rm I} - \PikKb) \bv \cdot \nabla \PikK \varphi, \PikK \phi   \bigr)_{E} + 
\bigl( \bv (\PikK - {\rm I}) \varphi , \nabla  \PikK \phi \bigr)_{E}
\\
& =: \sigma^E_{1,a} +  \sigma^E_{1,b} +  \sigma^E_{1,c} 
\end{aligned}
\end{equation}
and
\begin{equation}
\label{eq_sigma2}
\begin{aligned}
\sigma_2^E & =
(\bv \cdot \nabla \phi, ({\rm I} - \PikK) \varphi)_E
+ \bigl( ({\rm I} - \PikKb) \bv \cdot \nabla \PikK \phi, \PikK \varphi   \bigr)_{\E} + 
\bigl( \bv (\PikK - {\rm I}) \phi , \nabla  \PikK \varphi \bigr)_{E}
\\
&= \bigl(({\rm I}-\PikK )(\bv \cdot\nabla  \PikK \varphi \bigr), (\PikK - {\rm I} ) \phi )_{E} 
+ \bigl( ({\rm I} - \PikKb) \bv \cdot  \nabla \PikK \phi, \PikK \varphi  \bigr)_{\E} 
+
 (\bv \cdot \nabla \phi, ({\rm I} - \PikK) \varphi)_E
\\
& =: \sigma^E_{2,a} +  \sigma^E_{2,b} +  \sigma^E_{2,c} \,.
\end{aligned}
\end{equation}
In what follows, we will bound the six terms above.
Applying adequately the H\"{o}lder inequality and Proposition \ref{approx:prop:L2} we infer 
\begin{equation}
\label{eq:sigma1a}
\begin{aligned}
\sigma^E_{1,a} - \sigma^E_{2,a}
&\lesssim 
 \|({\rm I} - \PikK) (\bv \cdot \nabla ({\rm I} + \PikK) \varphi) \|_{0,E} \, \|({\rm I} - \PikK) \phi \|_{0,E} 
\\
&
\lesssim h_E^{m} \, |\bv \cdot \nabla (({\rm I} + \PikK) \varphi)|_{m,E} \, h_E^2 |\phi|_{2,E}
\\
& \lesssim
h_E^{m+2} \, \|\bv \|_{W^{m,4}(E)} \, \|\nabla (({\rm I} + \PikK) \varphi) \|_{W^{m,4}(E)} \, |\phi|_{2,E} 
\\
& \lesssim
h_E^{m+2} \, \|\bv \|_{W^{m,4}(E)} \, \|\varphi \|_{W^{m+1,4}(E)}|\phi|_{2,E}
\\
& \lesssim
h_E^{m+2} \, \|\bv \|_{W^{k,4}(E)} \, \|\varphi \|_{W^{\ell,4}(E)} \,|\phi|_{2,E} \,;
\\
\\
\sigma^E_{1,b} - \sigma^E_{2,b} &
\lesssim \| ({\rm I} - \PikKb) \bv \|_{0,E} 
\left( \|\nabla \PikK \varphi \|_{L^4(E)}  \| \PikK \phi \|_{L^4(E)} +
\|\nabla \PikK \phi \|_{L^4(E)}  \| \PikK \varphi \|_{L^4(E)}
\right)
\\
&\lesssim h_E^{k+1} \, |\bv |_{k+1,E} \, \|\varphi \|_{W^{1,4}(E)} \, \|  \phi \|_{W^{1,4}(E)}
\\
&\lesssim h_E^{m+1} \, |\bv |_{k+1,E} \, \|\varphi \|_{W^{1,4}(E)}  \|  \phi \|_{W^{1,4}(E)} \,;
\\
\\
\sigma^E_{1,c} - \sigma^E_{2,c} 
&\lesssim \| \bv \|_{L^4(E)} 
 \| ({\rm I} - \PikK) \varphi \|_{L^2(E)}
\|\nabla (({\rm I } + \PikK) \phi )\|_{L^4(E)} 
\\
& \lesssim h_E^{\ell+1} \, \| \bv \|_{L^4(E)} \,  |\varphi|_{\ell+1,E}  
\, \|  \phi \|_{W^{1,4}(E)}
\\
& \lesssim h_E^{m+2} \, \| \bv \|_{L^4(E)} \,  |\varphi|_{\ell+1,E}  
\|  \phi \|_{W^{1,4}(E)} \,.
\end{aligned}
\end{equation} 
Thus, employing bounds \eqref{eq:sigma1a},  equations \eqref{eq_sigma1}, \eqref{eq_sigma2} and \eqref{eq:sigma0}, we infer
\[
\begin{aligned}
c(\bv;\varphi,\phi) - c_h(\bv;\varphi,\phi) &\lesssim
\sum_{\E \in \O_h} h_E^{m+1} \,
|\bv |_{k+1,E}   \|\varphi \|_{W^{1,4}(E)}    \|  \phi \|_{W^{1,4}(E)} +
\\
& + \sum_{\E \in \O_h} h_E^{m+2}  \bigl( 
\|\bv \|_{W^{k,4}(E)}  \|\varphi \|_{W^{\ell,4}(E)},|\phi|_{2,E} +
\| \bv \|_{L^4(E)} \,  |\varphi|_{\ell+1,E}   
\|  \phi \|_{W^{1,4}(E)}
 \bigr) \,.
 \end{aligned}
\]
Therefore, from the H\"{o}lder inequality (for sequences) and the Sobolev embedding theorem, the thesis follows.
\end{proof}

Now, we have the following result for the forms $d(\cdot;\cdot,\cdot)$ and $d_h(\cdot;\cdot,\cdot)$.
\begin{lemma}\label{lemma:crimen:d}
Let  Assumption  \ref{Assump:mesh} hold.
Let  $\varphi\in  H\cap H^{\ell+1}(\O)$, then for all $\bv \in \bV$ it holds that
\begin{equation*}
d(\phi;\phi,\bv)-d_h(\phi;\phi,\bv)
\lesssim h^{m}\big(\|\phi\|_{\ell,\O}+\|\phi\|_{H} \big)
		\|\phi\|_{\ell+1,\O} \, \|\bv\|_{\bV} \,.
	\end{equation*}	
\end{lemma}
\begin{proof}
By using the definition of the forms $d(\cdot;\cdot,\cdot)$ and $d_h(\cdot;\cdot,\cdot)$ we infer
\begin{equation*}
	\begin{split}
		d(\phi;\phi,\bv)&-d_h(\phi;\phi,\bv) 
		= \sum_{\E \in \O_h}
		\left( ( \Delta \phi\nabla \phi, \bv)_{\E}-  \big(\PiMtwoK \Delta \phi  \, \PikMoneKp \nabla\phi , \PikKb \bv \bigr)_{\E} 
		\right)
		\\
		&=\sum_{\E \in \O_h}  ( \Delta \phi\nabla \phi, \bv-\PikKb \bv)_{\E} 
		+   \sum_{\E \in \O_h}\big(\Delta \phi\nabla \phi- \PiMtwoK \Delta \phi  \PikMoneKp \nabla\phi, \PikKb \bv \bigr)_{\E}\\
	&	=\sum_{\E \in \O_h} ( ({\rm I}-\PikKb)(\Delta \phi\nabla \phi), ({\rm I}- \PikKb) \bv )_{\E}
	+  \sum_{\E \in \O_h} \big(\Delta \phi  \: ({\rm I}-\PikMoneKp)\nabla\phi, \: \PikKb \bv \bigr)_{\E} +
\\		&\quad
		 +  \sum_{\E \in \O_h} \big(({\rm I}- \PiMtwoK) \Delta \phi \: \PikMoneKp \nabla\phi, \: \PikKb \bv \bigr)_{\E}
\end{split}
\end{equation*}
thus, employing H\"{o}lder inequality and Proposition \ref{approx:prop:L2},  we have
\begin{equation*}
\begin{split}
&d(\phi;\phi,\bv)-d_h(\phi;\phi,\bv) 		 
		 \leq \sum_{\E \in \O_h} \|({\rm I}-\PikKb)(\Delta \phi\nabla \phi)\|_{0,\E} \|({\rm I}- \PikKb) \bv\|_{\E}+
		 \\
		&  + \sum_{\E \in \O_h} \left( \|\Delta \phi\|_{\E} \|({\rm I}-\PikMoneKp)\nabla\phi\|_{L^4(\E)}  
		+  \|({\rm I}- \PiMtwoK) \Delta \phi\|_{0,\E} \|\PikMoneKp \nabla\phi\|_{L^4(\E)} \right) \|\PikKb \bv\|_{L^4(\E)}
		\\
		& \lesssim \sum_{\E \in \O_h}  h_{\E}^{m-1}|\Delta \phi\nabla \phi|_{m-1,\E} \, h_{\E}| \bv|_{1,\E} + 
		\\
		& \quad  +  \sum_{\E \in \O_h}
		  h_{\E}^{\ell-1} \left( \|\Delta \phi\|_{0,\E} |\nabla\phi|_{W^{\ell-1,4}(\E)} +  |\Delta \phi|_{\ell-1,\E} \|\nabla\phi\|_{L^4(\E)} \right) \| \bv\|_{L^4(\E)} 
		\\
		  & \lesssim \sum_{\E \in \O_h} h_E^{m}\bigl(\|\phi \|_{W^{m+1,4}(E)} \| \phi\|_{W^{m,4}(E)} | \bv|_{1,E}
		 + \left( |\phi|_{2,E} |\phi|_{W^{\ell,4}(E)}  + |\phi|_{\ell+1,E} |\phi|_{W^{1,4}(E)} \right) \| \bv\|_{L^4(E)} \bigr) \,.
	\end{split}
\end{equation*}

Then, the desired result follows from  the H\"{o}lder inequality (for sequences) and Sobolev inclusions.

\end{proof}

We conclude the subsection with the following bound involving the $W^{s,\infty}$-norms.  
\begin{lemma}\label{lemma:crimen:d-B}
Let  Assumption  \ref{Assump:mesh} hold.
Let   $\varphi\in  H \cap W^{2,\infty}(\O)$ and $\phi\in H \cap  W^{1, \infty}(\O)$, then  for all $\bv\in \bV$ it holds that
\begin{equation*}
d_h(\varphi;\varphi,\bv)-d_h(\phi;\phi,\bv)
\lesssim \big(\|\varphi\|_{W^{2,\infty}(\O)}+\|\phi\|_{W^{1,\infty}(\O)} \big) \|\varphi-\phi\|_{2, \O}\|\bv\|_{0,\O}\,. 
	\end{equation*}
\end{lemma}
\begin{proof}
Simple computations yield
\begin{equation*}
\begin{split}
d_h(\varphi;\varphi,\bv)-d_h(\phi;\phi,\bv)	
&= d_h(\varphi;\varphi-\phi,\bv)+d_h(\varphi-\phi;\phi,\bv)
\\
& \lesssim \|\varphi\|_{W^{2,\infty}(\O)}\|\varphi-\phi\|_{1,\O} \|\bv\|_{0,\O} +  \|\phi\|_{W^{1,\infty}(\O)}\|\varphi-\phi\|_{2,\O} \|\bv\|_{0,\O}\\
& \lesssim \big(\|\varphi\|_{W^{2,\infty}(\O)}+\|\phi\|_{W^{1,\infty}(\O)} \big) \|\varphi-\phi\|_{2,\O}\|\bv\|_{0,\O} \,,
\end{split}	
\end{equation*}
where we have used the definition of the discrete form $d_h(\cdot;\cdot,\cdot)$ and the continuity of the $L^2$-projectors with respect to the $L^2$- and $L^{\infty}$-norms.
 \end{proof}

\subsection{Ritz-projections and their approximation properties}\label{subsec:Ritz}

In this part we will introduce Ritz-type projections needed for our error analysis, together with their approximation properties. 
Let us introduce the following operators:

\begin{itemize}

\item the \textbf{Stokes projection} $\mathcal{S}_h \colon \bZ  \to \Zh$ defined for all $\bv \in \bZ$ by
\begin{equation}\label{Stokes:projector}
	a_{\bnabla,h} (\mS_h\bv, \bz_h) = a_{\bnabla}(\bv, \bz_h ) \qquad  \text{for all  $\bz_h \in \Zh$.} 
\end{equation}

\item the \textbf{Cahn--Hilliard elliptic projection} 
$\mP_h \colon H \to \Hh$ 
defined for all $\chi \in H$ by 
\begin{equation}\label{Cahn:projector}
		B_{h} (\mP_h\chi,\phi_h) = 	B (\chi,\phi_h)      \qquad  \text{for all  $\phi_h \in \Hh$,}
\end{equation} 
where, recalling that $\varphi$ denotes the continuous phase solution,  $B (\cdot,\cdot)$  and  $B_{h} (\cdot,\cdot)$ are the bilinear forms given by 
\begin{align}
	B(\psi,\phi) &:=   a_{\bD}(\psi,\phi) +   \varepsilon^{-2} r(\varphi; \psi, \phi)  
	+ \alpha  m(\psi,\phi), &\qquad \forall \psi, \phi \in H \,, \label{def:B}\\
	B_{h} (\psi_h,\phi_h) &:=   a_{\bD,h}(\psi_h,\phi_h )+   \varepsilon^{-2}\bar{r}_h(\varphi; \psi_h, \phi_h)  
	+ \alpha  m(\psi_h,\phi_h)       &\quad  \forall  \psi_h, \phi_h \in \Hh\,, \label{def:B_h}
\end{align}
with
\begin{equation}\label{form:rbar}
	\overline{r}_h(\eta; \psi, \phi) := \sum_{\E \in \O_h}  \big( f'(\eta) \PikMoneKp  \nabla \psi,  \PikMoneKp  \nabla \phi \big)_{E} \qquad \forall \eta, \psi, \phi \in H.   
\end{equation}	 
In definitions \eqref{def:B} and \eqref{def:B_h} $\alpha$ is a sufficiently large positive parameter such the the form $B_h(\cdot, \cdot)$ will be elliptic, i.e.
\begin{equation}
\label{eq:Bcoer}
\| \phi_h\|_{2, \O} \lesssim
B_{h} (\phi_h,\phi_h)
\quad \text{for all $\phi_h \in \Hh$.}
\end{equation}
\end{itemize}
We observe that both projection operators are well defined as a consequence of item 
$i)$ in Proposition~\ref{lemma:bound:ch} and \eqref{eq:Bcoer}.
We now state the following result.
\begin{lemma}\label{lemma:error:S_h}
Let Assumptions \ref{Assump:mesh} and  \ref{assump:reg:add} hold.
Let $ \mS_h \bu \in \bZ^k_h$ be the Stokes projection of the continuous velocity solution $\bu$.
Then for a.e. $t\in [0,T]$ the following holds
\begin{equation*}
	\|\bu - \mS_h \bu\|_{0,\Omega} + h 	|\bu - \mS_h \bu|_{1,\Omega} 
	\lesssim h^{k+1}  |\bu|_{k+1,\Omega} \,.	
\end{equation*}
\end{lemma}
\begin{proof}
The proof can be obtained by following similar arguments as in \cite{Vacca-parabolico,VK2023}.
\end{proof}

On the other hand, error estimates have been established for analogous Cahn–Hilliard-type formulations based on nonconforming Morley finite element and virtual element methods \cite{EF89,DH2024}, as well as for the lowest-order $C^1$-conforming virtual element method \cite{ABSV-SINUM2016}. For the sake of completeness, we provide in the forthcoming analysis error bounds for the Cahn–Hilliard elliptic projection defined in~\eqref{Cahn:projector} within our high-order $C^1$-conforming framework. We emphasize that, although the overall strategy is related, the proofs differ due to the intrinsic features of the respective numerical schemes.

We also rely on the following regularity assumption (see for instance \cite{ABSV-SINUM2016}).

\begin{assumption}[Assumption on the regularity of the auxiliary problem]
\label{Assump:reg:aux}
	Let $\mG \in H^{1}(\Omega)$ and let $\chi \in H$ be the solution of auxiliary the problem
	\begin{equation}
	\label{aux:problem}
		B(\chi, w) = (\mG,w )_{\O} + (\nabla \mG,\nabla w )_{\O}
		\qquad \forall w \in H.
	\end{equation}
	Then, there exists a positive constant $C_{\Omega}$, depending only on $\Omega$,
	such that
	\begin{equation*}
		\| \chi \|_{3,\Omega} \le C_{\Omega} \, \| \mG \|_{1,\O}.
	\end{equation*}
\end{assumption}

We note that, in the particular case where $\Omega$ is a rectangular domain,
Assumption~\ref{Assump:reg:aux} holds; see, for example, \cite[Theorem~A1]{EF89}.

We have the following approximation result for the elliptic projection $\mP_h$ defined in \eqref{Cahn:projector}.

\begin{lemma}\label{lemma:error:P_h}
Let Assumptions \ref{Assump:cont:phase}, \ref{Assump:mesh}, \ref{assump:reg:add}, and \ref{Assump:reg:aux} hold. 
Let $\mP_h \varphi \in \Hh$ be the Cahn-Hilliard elliptic projection of the continuous phase solution $\varphi$.
Then for a.e. $t\in [0,T]$ the following hold 
\begin{align*}
\|\varphi - \mP_h \varphi\|_{1,\O} +	h\|\varphi - \mP_h \varphi\|_{2,\O} &\leq  C(\varphi, f(\varphi), \varepsilon^{-2}, C_T, C_{\O}) h^{\ell} \,,
\\ 
\|\partial_t(\varphi - \mP_h \varphi)\|_{1,\O} +	h\|\partial_t(\varphi - \mP_h \varphi)\|_{2,\O} & \leq C(\varphi, f(\varphi), \varepsilon^{-2}, C_T, C_{\O}) h^{\ell}
\end{align*}	
\end{lemma}
\begin{proof}
We split the proof into two steps. In the first step, we establish the first
estimate. In the second step, we provide a sketch of the proof for the estimates
of $\partial_t(\varphi - \mP_h \varphi)$.

\paragraph{Step 1. Estimate in $H^2$- and $H^1$-norms of $\varphi - \mP_h \varphi$.}	
Let $\varphi_I \in \Hh$ be the interpolant of $\varphi$ (cf. Proposition \ref{approx:virtual:velo:phase}) and let $\delta_h:= \mP_h \varphi - \varphi_I \in \Hh$.
Exploiting the coericivity property in \eqref{eq:Bcoer} and the definition of Cahn-Hilliard projection we infer 
\begin{equation}\label{eq:Ph-B0}
\begin{split}
\| \delta_h \|_{2,\O}^ 2 &\lesssim 
B_h(\delta_h, \delta_h) = 
B_h(\mP_h \varphi, \delta_h) -
B_h(\varphi_I, \delta_h) =
B(\varphi, \delta_h) -
B_h(\varphi_I, \delta_h) 
\\
&= 
\bigl(a_{\bD}(\varphi,\delta_h) - a_{\bD, h}(\varphi_I,\delta_h) \bigr) +
 \varepsilon^{-2} \bigl( r(\varphi; \varphi, \delta_h) - \overline{r}_h(\varphi; \varphi_I, \delta_h) \bigr)
\\
& =: B_1 + B_2\,.
\end{split}
\end{equation} 
Using the definition of $L^2$-projection, simple computations yield:
\[
\begin{aligned}
B_1 &= \sum_{E \in \Oh} \bigl(  (\bD^2 \varphi, \bD^2 \delta_h)_E  -\big(\boldsymbol{\Pi}^{0,\ell-2}_{\E}\bD^2 \varphi_I, 	\boldsymbol{\Pi}^{0,\ell-2}_{\E}\bD^2 \delta_h\big)_{\E} -
s_{\bD}^{\E}\big(({\rm I}-\PiK) \varphi_I,({\rm I}-\PiK) \delta_h\big) \bigr)
\\
&= \sum_{E \in \Oh} \bigl(  \big(({\rm I} - \boldsymbol{\Pi}^{0,\ell-2}_{\E})\bD^2 \varphi  +
\boldsymbol{\Pi}^{0,\ell-2}_{\E}\bD^2 (\varphi -\varphi_I), 	\bD^2 \delta_h\big)_{\E} -
s_{\bD}^{\E}\big(({\rm I}-\PiK) \varphi_I,({\rm I}-\PiK) \delta_h\big) \bigr) \,.
\end{aligned}
\]
Thus, employing the Cauchy-Schwarz inequality, the equivalences \eqref{term:stab:SK}, Propositions \ref{approx:prop:L2} and \ref{approx:virtual:velo:phase},  along with  the triangular inequality, we infer
\begin{equation}
\label{eq:Ph-B1}
\begin{aligned}
B_1  &\lesssim 
\sum_{E \in \Oh} \biggl( \big(\|({\rm I} - \boldsymbol{\Pi}^{0,\ell-2}_{\E})\bD^2 \varphi\|_{0,E} +  |\varphi - \varphi_I|_{2,E}\bigr) 
| \delta_h|_{2,E} + |({\rm I} - \PiK) \varphi_I |_{2,E} | ({\rm I} - \PiK) \delta_h |_{2,E} \biggr)
\\
& \lesssim 
\sum_{E \in \Oh} \bigl( \|({\rm I} - \boldsymbol{\Pi}^{0,\ell-2}_{\E})\bD^2 \varphi\|_{0,E} +  |\varphi - \varphi_I|_{2,E} +
|({\rm I} - \PiK) \varphi |_{2,E} \bigr) 
| \delta_h|_{2,E} 
\\
& \lesssim 
\sum_{E \in \Oh} h_E^{\ell-1} |\varphi|_{\ell+1,E} 
|\delta_h|_{2,E} 
 \lesssim
h^{\ell-1} |\varphi|_{\ell+1,\O}  \|\delta_h\|_{2,\O} \,. 
\end{aligned}
\end{equation}
We now focus on the term $B_2$. Employing the definition of $L^2$-projection and the identity $\nabla f(\varphi) = f'(\varphi)\nabla \varphi$, we obtain
\[
\begin{aligned}
& B_2  =
\varepsilon^{-2}\sum_{\E \in \O_h}  \biggl( 
\big(f'(\varphi) \nabla \varphi,   \nabla \delta_h\big)_{E}  - 
\big(f'(\varphi) \PikMoneKp  \nabla \varphi_I,  \PikMoneKp  \nabla \delta_h\big)_{E} 
\biggr)
\\
& 
= 
\varepsilon^{-2}\sum_{\E \in \O_h}  \biggl( 
\big(  f'(\varphi)\nabla \varphi,   ({\rm I} - \PikMoneKp)\nabla \delta_h\big)_{E}  +
\big(f'(\varphi) (\nabla \varphi - \PikMoneKp  \nabla \varphi_I),  \PikMoneKp  \nabla \delta_h\big)_{E} 
\biggr)
\\
& = 
\varepsilon^{-2}\sum_{\E \in \O_h}  \biggl( 
\big(  ({\rm I} - \PikMoneKp) \nabla f(\varphi), ({\rm I} - \PikMoneKp)\nabla \delta_h\big)_{E}  +
\big(f'(\varphi) (\nabla \varphi - \PikMoneKp  \nabla \varphi_I),  \PikMoneKp  \nabla \delta_h\big)_{E} 
\biggr) \,.
\end{aligned}
\]
Thus, employing the Cauchy-Schwarz inequality, Propositions \ref{approx:prop:L2} and  \ref{approx:virtual:velo:phase}, the term $B_2$ can be bounded as follows:
\begin{equation}
\label{eq:Ph-B2} 
\begin{aligned}
B_2 &\lesssim \varepsilon^{-2} \sum_{\E \in \O_h} 
\bigl(
h_E^{\ell-2}|f(\varphi)|_{\ell-2,E} \, h_E |\delta_h|_{2,E} + 
\|f'(\varphi)\|_{L^\infty(E)} \big( \|({\rm I} - \PikMoneKp)  \nabla \varphi)\|_{0,E}  + |\varphi - \varphi_I|_{1,E}\big) |\delta_h|_{1,E}
\bigr)
\\
& \lesssim \varepsilon^{-2} \sum_{\E \in \O_h} 
\bigl(
h_E^{\ell-1}|f(\varphi)|_{\ell-2,E}  + 
h_E^{\ell}  \|\varphi^2  + 1\|_{L^\infty(E)} |\varphi|_{\ell+1,E} \bigr) \|\delta_h\|_{2,E}
\\
& \lesssim \varepsilon^{-2} h^{\ell-1}
\bigl( |f(\varphi)|_{\ell-2,\O}  + 
h \, C(C_T)   |\varphi|_{\ell+1, \O} \bigr) \|\delta_h\|_{2, \O} \,.
\end{aligned}
\end{equation}
Therefore, collecting bounds \eqref{eq:Ph-B1} and \eqref{eq:Ph-B2} in \eqref{eq:Ph-B0}, together with Proposition \ref{approx:virtual:velo:phase} and the triangular inequality we arrive at
\begin{equation}
\label{eq:Ph-norma2}
\|\varphi - \mP_h \varphi\|_{2, \O} \leq 
C(\varphi, f(\varphi), \varepsilon^{-2}, C_T) h^{\ell-1} \,.
\end{equation}

We now proceed to derive the $H^{1}$-error estimate.  
To this end, we apply Assumption~\ref{Assump:reg:aux} with $\mG = (\varphi - \mP_h \varphi) \in H^2(\O) \subset H^1(\O)$, which gives
\begin{equation}\label{add:reg}
	\|\chi\|_{3,\Omega} \le C_{\Omega} \, \|\varphi - \mP_h \varphi\|_{1,\O}\,.
\end{equation}
Let $\chi_I \in \Hh$ be the $\Hh$-interpolant of $\chi$ (cf. Proposition \ref{approx:virtual:velo:phase}), then taking $w = (\varphi - \mP_h \varphi) \in  H$ in \eqref{aux:problem} we have
\begin{equation}\label{split:H1:norm}
	\begin{split}
		&\|\varphi - \mP_h \varphi\|_{1,\O}^2 
		= B(\chi, \varphi - \mP_h \varphi) \\
		&= B(\chi - \chi_I, \varphi - \mP_h \varphi) + B(\chi_I, \varphi - \mP_h \varphi) \\
		&= B(\chi - \chi_I, \varphi - \mP_h \varphi) + B(\chi_I, \varphi) - B(\chi_I, \mP_h \varphi) \\
		&= B(\chi - \chi_I, \varphi - \mP_h \varphi) + \bigl(B_h(\chi_I, \mP_h \varphi) - B(\chi_I, \mP_h \varphi)\bigr) \\
		&= B(\chi - \chi_I, \varphi - \mP_h \varphi) + \bigl(a_{\bD,h}(\chi_I, \mP_h \varphi) - a_{\bD}(\chi_I, \mP_h \varphi)\bigr) +
		\varepsilon^{-2} \big(\overline{r}_h(\varphi; \chi_I, \mP_h \varphi) - r(\varphi; \chi_I, \mP_h \varphi)\big)
		 \\
		&=: B_3 + B_4 + B_5 \,,
	\end{split}
\end{equation}
where we used that $B_h(\chi_I, \mP_h \varphi) = B(\chi_I, \varphi)$, which follows directly from the definition of the elliptic projection $\mP_h$.
For $B_3$, employing definition \eqref{def:B}, Proposition~\ref{approx:virtual:velo:phase}, bounds \eqref{add:reg} 
and \eqref{eq:Ph-norma2},
 we immediately obtain
\begin{equation}\label{eq:Ph-B3}
\begin{aligned}
	B_3 & \lesssim 
	\bigl( |\chi - \chi_I|_{2,\O}  + 
	\varepsilon^{-2} \|f'(\varphi)\|_{L^\infty(\O)} \,
	|\chi - \chi_I|_{1,\O}  +
	\alpha \|\chi - \chi_I\|_{0,\O} \bigr) \|\varphi - \mP_h\varphi\|_{2,\O} 
	\\
	&  \lesssim
	h  \bigl( 1 + \varepsilon^{-2} C(C_T) h + \alpha h^2\bigr)
	|\chi|_{3, \O} \, \|\varphi - \mP_h \varphi\|_{2,\O}
	\\
	&  \lesssim
	C(\varphi, f(\varphi), \varepsilon^{-2}, C_T, C_{\O}) h^\ell \,
	\|\varphi - \mP_h \varphi\|_{1, \O} \,.
\end{aligned}
\end{equation}
We next estimate $B_4$. We preliminary observe that 
\[
\begin{aligned}
B_4 &= \sum_{E \in \Oh} \bigl( \big((\boldsymbol{\Pi}^{0,\ell-2}_{\E}- {\rm I})\bD^2 \chi_I, 	\bD^2 \mP_h \varphi\big)_{\E} +
s_{\bD}^{\E}\big(({\rm I}-\PiK) \chi_I,({\rm I}-\PiK) \mP_h \varphi\big) 
 \bigr)
\\
&= \sum_{E \in \Oh} \bigl( \big(({\rm I} - \boldsymbol{\Pi}^{0,\ell-2}_{\E})\bD^2 \chi_I, \boldsymbol{\Pi}^{0,\ell-2}_{\E}\bD^2 \varphi -	\bD^2 \mP_h \varphi\big)_{\E} +
s_{\bD}^{\E}\big(({\rm I}-\PiK) \chi_I,({\rm I}-\PiK) \mP_h \varphi\big) 
 \bigr) \,.
\end{aligned}
\]
Thus, employing Propositions \ref{approx:prop:L2} and \ref{approx:virtual:velo:phase}, the triangular inequality, and bounds \eqref{eq:Ph-norma2} and \eqref{add:reg}, we obtain
\begin{equation}
\label{eq:Ph_B4}
\begin{aligned}
B_4& \lesssim 
\sum_{E \in \Oh} h_E |\chi|_{3,E} 
\big( h_E^{\ell-1}|\varphi|_{\ell+1,E} + |\varphi - \mP_h \varphi|_{2,E} \big)
\\ 
&  \lesssim
h \big( h^{\ell-1}|\varphi|_{\ell+1,E} + \|\varphi - \mP_h \varphi\|_{2,\O} \big) \|\chi \|_{3, \O}
\\
& \lesssim
C(\varphi, f(\varphi), \varepsilon^{-2}, C_T, C_\O) h^\ell \, \|\varphi - \mP_h \varphi\|_{1, \O} \,.
\end{aligned}
\end{equation}
Finally for $B_5$, employing the definition of $L^2$-projection and the identity $\nabla f(\varphi) = f'(\varphi)\nabla \varphi$, we obtain
\[
\begin{aligned}
& B_5  =
\varepsilon^{-2}\sum_{\E \in \O_h}  \biggl( 
\big(f'(\varphi) \PikMoneKp  \nabla \chi_I,  \PikMoneKp  \nabla \mP_h \varphi\big)_{E} 
-\big(f'(\varphi) \nabla \chi_I,   \nabla \mP_h \varphi\big)_{E}  
\biggr)
\\
&= \varepsilon^{-2}\sum_{\E \in \O_h}  \biggl( 
\big(f'(\varphi) \PikMoneKp  \nabla \chi_I,  (\PikMoneKp  - {\rm I}) \nabla \mP_h \varphi\big)_{E} 
+ \big(f'(\varphi) (\PikMoneKp - {\rm I})\nabla \chi_I,   \nabla \mP_h \varphi\big)_{E}  
\biggr)
\\
&= \varepsilon^{-2}\sum_{\E \in \O_h}  \biggl( 
\big(({\rm I}- \Pi_E^{0,0})f'(\varphi)   \PikMoneKp  \nabla \chi_I,  (\PikMoneKp  - {\rm I}) \nabla \mP_h \varphi\big)_{E} 
+ 
\\
& \quad  + \big(f'(\varphi) (\PikMoneKp - {\rm I} )\nabla \chi_I,   \nabla \mP_h \varphi - \nabla \varphi\big)_{E}  +
\big( (\PikMoneKp - {\rm I}) \nabla \chi_I,  ({\rm I} - \PikMoneKp)\nabla f(\varphi) \big)
\biggr) \,.
\end{aligned}
\]
Thus, arguing as in \eqref{eq:Ph-B2}, we obtain 
\begin{equation}
\label{eq:Ph-B5}
\begin{aligned}
B_5 &\lesssim
\varepsilon^{-2}
\sum_{E \in \Oh} \biggl( 
h_E \, |f'(\varphi)|_{W^{1,\infty}(E)} \, 
(|\chi|_{1,E} + h_E^2\, |\chi|_{3,E}) \, (h_E^{\ell} |\varphi|_{\ell+1,E} + |\varphi - \mP_h \varphi|_{1,E}) +
\\
& \quad +  
h_E^2\, |\chi|_{3,E} \, (\|f'(\varphi)\|_{L^{\infty}(E)}  \,|\varphi - \mP_h \varphi|_{1,E} + h_E^{\ell-2}|f(\varphi)|_{\ell-2,E} ) \biggr) 
\\
&\lesssim
\varepsilon^{-2}
\sum_{E \in \Oh} h_E \, \|\chi\|_{3,E}\biggl( 
\|\varphi^2+1\|_{W^{1,\infty}(E)} \, 
(h_E^{\ell} |\varphi|_{\ell+1,E} + \|\varphi - \mP_h \varphi\|_{2,E}) +
h_E^{\ell-1}|f(\varphi)|_{\ell-2,E} ) \biggr) 
\\
& \lesssim
C(\varphi,  f(\varphi), \varepsilon^{-2}, C_T, C_\O) h^{\ell} \|\varphi - \mP_h \varphi\|_{1,\O} \,.
\end{aligned}
\end{equation}
Thus, combining bounds in \eqref{eq:Ph-B3}, \eqref{eq:Ph_B4}, \eqref{eq:Ph-B5} in equation \eqref{split:H1:norm} we obtain
\begin{equation}
\label{eq:Ph-norma1}
\|\varphi - \mP_h \varphi\|_{1,\O}  \leq C(\varphi, f(\varphi), \varepsilon^{-2}, C_T, C_\O) h^{\ell}  \,.
\end{equation}
The thesis now follows from \eqref{eq:Ph-norma2} and \eqref{eq:Ph-norma1}.

\paragraph{Step 2. Estimate in $H^2$- and $H^1$-norms of $\partial_t(\varphi - \mP_h \varphi)$.}	
We provide a sketch of the proof for the estimates of 
$\partial_t(\varphi - \mP_h \varphi)$. 
We start by proving the $H^2$-estimate. Let $\eta_h = \partial_t(\mP_h \varphi) - \mP_h (\partial_t \varphi)  \in \Hh$. 
Thus, employing \eqref{eq:Bcoer} and the definition of $\mP_h$, we deduce
\begin{equation}
\label{eq:eta}
\begin{split}	
\|\eta_h\|^2_{2,\O} &\lesssim	
B_h(\eta_h, \eta_h) =
B_h(\partial_t(\mP_h \varphi), \eta_h) - B_h(\mP_h(\partial_t\varphi), \eta_h) = 
B_h(\partial_t(\mP_h \varphi), \eta_h) - B(\partial_t\varphi, \eta_h)
\\
& = \varepsilon^{-2}\sum_{E \in \Omega_h}\biggl( \bigl(\partial_t (f^{\prime}(\varphi)) \, \nabla \varphi, \nabla \eta_h\bigr)_{\E}
- \bigl(\partial_t (f^{\prime}(\varphi)) \, \PikMoneKp \nabla\mP_h \varphi, \PikMoneKp \nabla \eta_h\bigr)_{\E} \biggr) 
=:B_6 \,.
\end{split}
\end{equation}
We now bound the term $B_6$. 
We first observe that 
\[
\begin{aligned}
 B_6  & = 
 \varepsilon^{-2}\sum_{\E \in \O_h}  \biggl( 
\big(  ({\rm I} - \PikMoneKp)( \partial_t (f^{\prime}(\varphi)) \, \nabla \varphi),   ({\rm I} - \PikMoneKp)\nabla \eta_h\big)_{E}  +
\\
& \quad + 
\big(\partial_t (f^{\prime}(\varphi))  (\nabla \varphi - \PikMoneKp  \nabla \mP_h \varphi),  \PikMoneKp  \nabla \eta_h\big)_{E} 
\biggr)\,.
\end{aligned}
\]
Therefore, using the identity $\partial_t (f'(\varphi)) = 6 \varphi \, \partial_t \varphi$, the H\"{o}lder inequality, the Sobolev inclusions and bound \eqref{eq:Ph-norma1},  we obtain
\begin{equation}
\label{eq:Ph-B6}
\begin{aligned}
B_6 &\lesssim 
\varepsilon^{-2} \sum_{\E \in \O_h} 
\bigl(
h_E^{\ell-2} \, |\partial_t (f'(\varphi)) \, \nabla \varphi|_{\ell-2,E} \, h_E \, |\eta_h|_{2,E} + 
\\
& \quad + 
\|\varphi\|_{L^\infty(E)} \|\partial_t \varphi\|_{L^4(E)} \big( \|({\rm I} - \PikMoneKp)  \nabla \varphi)\|_{0,E}  + |\varphi - \mP_h \varphi|_{1,E}\big) |\eta_h|_{W^{1,4}(E)}
\bigr)
\\
& \lesssim \varepsilon^{-2} \sum_{\E \in \O_h} 
\bigl(
h_E^{\ell-1} \, |\varphi|_{W^{\ell-2, 6}(E)}  \,  
|\partial_t \varphi|_{W^{\ell-2, 6}(E)} \,
|\varphi|_{W^{\ell-1, 6}(E)} \, \|\eta_h\|_{2,E} +
\\
& \quad + 
C(C_T) \|\partial_t \varphi\|_{L^4(E)} \big( h^\ell \, |\varphi|_{\ell+1,E}  + |\varphi - \mP_h \varphi|_{1,E}\big) |\eta_h|_{W^{1,4}(E)}
\\
& \lesssim \varepsilon^{-2} 
\bigl(
h^{\ell-1}\|\varphi\|_{H^{\ell-1}(\O)}  \,  
\|\partial_t \varphi\|_{H^{\ell-1}(\O)} \,
\|\varphi\|_{H^{\ell}(\O)}  +
C(C_T) \|\partial_t \varphi\|_{H^1(\O)}   \, |\varphi - \mP_h \varphi|_{1,\O}
\bigr)
\, \|\eta_h\|_{2,\O} 
\\
& \lesssim
C(\varphi, f(\varphi), \varepsilon^{-2}, C_T) h^{\ell-1} \, \|\eta_h\|_{2, \O} \,.
\end{aligned}
\end{equation}
Thus, combining bounds \eqref{eq:eta} and \eqref{eq:Ph-B6} with \eqref{eq:Ph-B2} (where we replace $\varphi$ with $\partial_t \varphi$) along with the triangular inequality, we obtain
\[
\|\partial_t (\varphi - \mP_h \varphi)\|_{2, \O} 
\lesssim 
\|\partial_t \varphi - \mP_h(\partial_t \varphi)\|_{2, \O}  +
\|\eta_h\|_{2, \O} 
\leq 
C(\varphi, f(\varphi), \varepsilon^{-2}, C_T, C_\Omega) h^{\ell-1} \,.
\]
To establish the remaining bound, 
we consider the auxiliary problem \eqref{aux:problem} with $\mathcal{G}= \partial_t (\varphi - \mP_h \varphi)$ and 
we proceed analogously to the proof of the $H^1$-estimate for $\varphi - \mP_h \varphi$.

\end{proof}

\begin{remark}\label{remark:projection}
We recall that as in \cite{ABSV-SINUM2016} and~\cite{EF89} (see also \cite{DH2024}) the elliptic operator $\mP_h \phi  \in W^{1, \infty}(\O)$, with norm bounded uniformly in time, for all $\phi \in H^2(\O)$,
\begin{equation*}
\| \mP_h \phi \|_{W^{1, \infty}(\O)} \leq C(\phi).
\end{equation*}
We will use this property in the proof of some technical lemmas described in the next subsection.	
\end{remark}

\subsection{Error bounds for the semi-discrete problem}
In this subsection we will provide error bounds for the semi-discrete scheme. We start by establishing error identities for the the momentum and  energy equations.

\paragraph{Error identity for the momentum equation.} By using the Stokes operator defined in \eqref{Stokes:projector}, we split the error between the continuous and discrete velocity solutions as:
\begin{equation}\label{eq:split:velo}
\bu -\bu_h = (\bu-\mS_h\bu)- (\bu_h-\mS_h\bu) =: \be^{\bu}_I - \be_h^{\bu}. 
\end{equation}
Thus, from the momentum equation  in \eqref{form:kernel},  for all $\bv_h \in \Zh$, we have the following error identity:
\begin{equation}\label{NS:error:eq}
	\begin{split}
&m_{F,h}(\partial_t \be_h^{\bu}, \bv_h) + \nu a_{\bnabla,h}(\be_h^{\bu}, \bv_h)  =\big((\bg_h,\bv_h)_{\O}- (\bg, \bv_h)_{\O}\big) +\big(  c_{F}(\bu; \bu, \bv_h)- c_{F,h}(\bu_h; \bu_h, \bv_h)\big)\\
& \quad +   \big( m_{F}(\partial_t \bu, \bv_h)- m_{F,h}(\partial_t \mS_h \bu, \bv_h)\big)+  \lambda  \big( d(\varphi ;\varphi, \bv_h)- d_{h}(\varphi_h ;\varphi_h, \bv_h)\big).
	\end{split}
\end{equation}
\paragraph{Error identity for the phase equation.} For the phase equation we use the Cahn--Hilliard operator defined in \eqref{Cahn:projector}, and we split the phase error as:
\begin{equation}\label{eq:split:phase}
\varphi -\varphi_h = (\varphi-\mP_h\varphi ) - ( \varphi_h-\mP_h\varphi) =: e^{\varphi}_I - e_h^{\varphi}. 
\end{equation}
Then, from the definition of $\mP_h \varphi$ and the second equation in \eqref{form:kernel}, for all $\phi_h \in \Hh$, we get
\begin{equation}\label{Cahn:error:eq}
	\begin{split}
m_{h}(\partial_t e^{\varphi}_h, \phi_h) &+ \gamma a_{\bD,h}(e^{\varphi}_h, \phi_h) 
= \big(m(\partial_t \varphi, \phi_h)-m_{h}(\partial_t \mP_h\varphi, \phi_h) \big)  -\gamma \alpha m( e^{\varphi}_I, \phi_h)_{\O} \\
& +   \gamma \varepsilon^{-2}\big(\bar{r}_h(\varphi; \mP_h\varphi, \phi_h)  -r_h(\varphi_h; \varphi_h, \phi_h)\big) +
\big(c(\bu;\varphi,\phi_h)-c_h(\bu_h;\varphi_h,\phi_h) \big),		
	\end{split}
\end{equation}
where $\bar{r}_h(\cdot;\cdot, \cdot)$ is defined in \eqref{form:rbar}. 

The next step is to establish error estimates for the momentum and phase equations (cf. \eqref{NS:error:eq} and \eqref{Cahn:error:eq}). The following two lemmas provide such bounds and will be useful to obtain the convergence result for the semi-discrete problem.

\begin{lemma}[Error estimate for the momentum equation]\label{lemma:momentum:error}
Let Assumptions \ref{Assump:cont:phase}, \ref{Assump:mesh},  \ref{Assump:disc:phase}, \ref{assump:reg:add} and \ref{Assump:reg:aux} hold. 
Then for a.e. $t \in (0,T]$,  the following error estimate holds	
\begin{equation*}
\begin{split}
\frac{1}{2} \partial_t ||| \be_h^{\bu}|||^2_{0,\Omega} + \frac{\nu}{2}  ||| \be_h^{\bu}|||^2_{1,\O} 
&\le
C_{{\rm M},1} \, h^{2m} +
\frac{\gamma}{4} |||e_h^{\varphi}||||^2_{2,\O} +  
C_{{\rm M},2} \,|||\be_h^{\bu}|||^2_{0,\Omega}  \,,
\end{split}
\end{equation*}
where $C_{{\rm M},1} = C(\bu, \varphi, f(\varphi), \bg, \nu^{-1}, \lambda, \varepsilon^{-2}, C_T, C_{\O})$ and 
$C_{{\rm M},2} = C(\bu, \nu^{-3}, \gamma^{-1}, \lambda, C_T, \widetilde{C}_T)$ are positive constants independent of $h$.
\end{lemma}

\begin{proof} Taking $\bv_h := \be_h^{\bu} \in \Zh$ in \eqref{NS:error:eq}, using the coercivity properties in Proposition \ref{lemma:bound:ch} and the equivalences \eqref{equiv:norms},   we have 
\begin{equation}\label{error:estima:NS-1}
\begin{split}
\frac{1}{2} \partial_t ||| \be_h^{\bu}|||^2_{0,\Omega}& + \nu ||| \be_h^{\bu}|||^2_{1,\O} \lesssim  
\big((\bg_h,\be_h^{\bu})_{\O}- (\bg, \be_h^{\bu})_{\O}\big) + \big(  c_{F}(\bu; \bu, \be_h^{\bu})- c_{F,h}(\bu_h; \bu_h, \be_h^{\bu})\big)\\
& \quad +   \big( m_{F}(\partial_t \bu, \be_h^{\bu})- m_{F,h}(\partial_t \mS_h \bu, \be_h^{\bu})\big)
+  \lambda  \big( d(\varphi ;\varphi, \be_h^{\bu})- d_{h}(\varphi_h ;\varphi_h, \be_h^{\bu})\big)\\
&=: T_{\bg} + T_{c} + T_{m} + \lambda T_{d}.
\end{split}
\end{equation}
In what follows, we will provide bounds for each terms above. In fact, the terms $T_{\bg}$ and $T_m$ are easy to bound by using classical arguments~\cite{BLV-SINUM2018,Vacca-parabolico}:
\begin{equation}\label{Tf:Tm}
T_{\bg} +T_m \leq C h^{2k} (\|\partial_t \bu\|^2_{k,\O}+ \|\bg\|^2_{k,\O}) +C|||\be_h^{\bu}|||^2_{0,\Omega}.
\end{equation}
Thus, we will focus our attention on the terms $\lambda T_{d}$ and $T_{c}$. From the definition we have
\begin{equation*}
\begin{split}
\lambda T_d =  \lambda \big( d(\varphi ;\varphi, \be_h^{\bu})- d_{h}(\varphi ;\varphi, \be_h^{\bu}) \big) + \lambda \big( d_{h}(\varphi ;\varphi, \be_h^{\bu})- d_{h}(\varphi_h ;\varphi_h, \be_h^{\bu}) \big) =: \lambda T_{d1}+ \lambda T_{d2}.
\end{split}	
\end{equation*} 
We use Lemma~\ref{lemma:crimen:d}, equivalences \eqref{equiv:norms} 
and the Young inequality to obtain
\begin{equation*}
	\begin{split}
\lambda T_{d1} &\leq C \lambda h^{m}(\|\varphi\|_{\ell,\O} + \|\varphi\|_{H} ) \|\varphi\|_{\ell+1,\O}  \| \be_h^{\bu}\|_{\bV} \\
& \leq C \lambda^2 \nu^{-1} h^{2m}(\|\varphi\|_{\ell,\O} + \|\varphi\|_{H} )^2 \|\varphi\|^2_{\ell+1,\O} + \frac{\nu}{4}  ||| \be_h^{\bu}|||^2_{1,\O}.	
	\end{split}
\end{equation*}
Recalling Assumption \ref{Assump:cont:phase} and \ref{Assump:disc:phase}, from Lemma~\ref{lemma:crimen:d-B} and Lemma \ref{lemma:error:P_h}, we derive
\begin{equation*}
	\begin{split}
\lambda T_{d2} & \leq C \lambda \big(\|\varphi\|_{W^{2,\infty}(\O)} +\|\varphi_h\|_{W^{1,\infty}(\O)} \big)  \|\varphi-\varphi_h\|_{2,\O} \|\be_h^{\bu}\|_{0,\O}\\  
  &\leq \lambda (C_T + \widetilde{C}_T) \big( \|e_h^{\varphi}\|_{2,\O}  +\|e_I^{\varphi}\|_{2,\O}  \big) \|\be_h^{\bu}\|_{0,\O}\\  
  &\leq \lambda (C_T + \widetilde{C}_T) \big( |||e_h^{\varphi}|||_{2}  + C(\varphi, f(\varphi), \varepsilon^{-2}, C_T, C_{\O}) h^{\ell -1}  \big) \|\be_h^{\bu}\|_{0,\O}\\  
   & \leq    C(\varphi, f(\varphi), \varepsilon^{-2}, C_T, C_{\O})h^{2(\ell -1)}   +\frac{\gamma}{4} |||e_h^{\varphi}||||^2_{2,\O} + C\lambda^2 (C_T + \widetilde{C}_T) ^2( 1 + \gamma^{-1})  |||\be_h^{\bu}|||^2_{0,\O} \,.
   	\end{split}
\end{equation*}
Thus, by combining  the there terms above we have
\begin{equation}\label{Td}
	\begin{split}
		\lambda T_d  \leq C(\varphi, f(\varphi), \nu^{-1}, \lambda, \varepsilon^{-2}, C_T, C_{\O}) h^{2m} +  C(\gamma^{-1}, \lambda, C_T, \widetilde{C}_T) 
		|||\be_h^{\bu}|||^2_{0,\O} + \frac{\gamma}{4} |||e_h^{\varphi}|||^2_{2,\O} + \frac{\nu}{4} ||| \be_h^{\bu}|||^2_{1,\O}\,.
	\end{split}	
\end{equation} 
Now, we will analyze term $T_c$. We have that 
\begin{equation}\label{Tc}
\begin{split}
T_{c} 
&=   \big(c_{F}(\bu ;\bu, \be_h^{\bu}) - c_{F,h}(\bu ;\bu, \be_h^{\bu}) \big)  
+  \big(c_{F,h}(\bu ;\bu, \be_h^{\bu}) - c_{F,h}(\bu_h ;\bu_h, \be_h^{\bu}) \big) =: T_{c1} + T_{c2} \,.	
\end{split}
\end{equation} 
For the term $ T_{c1}$, we apply Lemma~\ref{lemma:crimen:cF} and the Young inequality to obtain
\begin{equation}\label{Tc1}
\begin{aligned}
T_{c1} &\leq Ch^{k}\big(\|\bu\|_{k,\O}+\|\bu\|_{\bV} +\|\bu\|_{k+1,\O} \big) \|\bu\|_{k+1,\O}\|\be_h^{\bu}\|_{\bV}
\\
&\leq C\nu^{-1}h^{2k}\big(\|\bu\|_{k,\O}+\|\bu\|_{\bV} +\|\bu\|_{k+1,\O} \big)^2 \|\bu\|_{k+1,\O}^2 + \frac{\nu}{8}|||\be_h^{\bu}|||_{1,\O}^2.
\end{aligned}
\end{equation}
The term $T_{c2}$ need additional analysis. To this end, we add and subtract suitable terms, and employ the identity~\eqref{eq:split:velo} and the skew-symmetry  of $c_{F,h}(\cdot; \cdot,\cdot)$ to obtain 
\begin{equation*}\label{Tc2a}
\begin{split}
 T_{c2} 
 &=  	c_{F,h}(\bu ;\bu-\bu_h, \be_h^{\bu}) + c_{F,h}(\bu-\bu_h ;\bu_h, \be_h^{\bu})\\
   &=  	c_{F,h}(\bu ;\be_I^{\bu}, \be_h^{\bu}) +	c_{F,h}(\be_I^{\bu} ;\bu_h, \be_h^{\bu}) -c_{F,h}(\be_h^{\bu},\bu_h, \be_h^{\bu})\\
     & =	c_{F,h}(\bu ;\be_I^{\bu}, \be_h^{\bu}) +  
     \bigl( c_{F,h}(\be_I^{\bu};\bu, \be_h^{\bu}) -	c_{F,h}(\be_I^{\bu} ;\be_I^{\bu}, \be_h^{\bu}) \bigr) 
     + \bigl( c_{F,h}(\be_h^{\bu}; \be_I^{\bu}, \be_h^{\bu})
     -c_{F,h}(\be_h^{\bu}; \bu, \be_h^{\bu}) \bigr)
 \,.
\end{split}
\end{equation*}  
We now use the  boundedness properties of $c_{F,h}(\cdot; \cdot,\cdot)$ (cf. Proposition \ref{lemma:bound:ch}, item $iii)$), Lemma \ref{lemma:error:S_h}, the Young inequality twice,  and derive the following bound
\begin{equation}
\label{eq:Tc2}
\begin{split} 
 T_{c2} &  \leq  
C  \|\be_I^{\bu}\|_{\bV} (\|\bu\|_{\bV} + \|\be_I^{\bu}\|_{\bV})  \|\be_h^{\bu}\|_{\bV} + 
C \|\be_h^{\bu}\|_{0, \O}^{1/2} \|\be_h^{\bu}\|_{\bV}^{1/2} (\|\bu\|_{\bV} + \|\be_I^{\bu}\|_{\bV})  \|\be_h^{\bu}\|_{\bV}
\\
& \leq \frac{\nu}{16} |||\be_h^{\bu}|||^2_{1,\O}  + C \nu^{-1} (\|\bu\|^2_{\bV} + \|\be_I^{\bu}\|^2_{\bV}) \|\be_I^{\bu}\|^2_{\bV}+
\\
 &  \quad \: + C \nu^{-1} \|\be_h^{\bu} \|_{0,\O}  \|\be_h^{\bu} \|_{\bV} \big(\|\bu\|^2_{\bV} + \|\be_I^{\bu} \|^2_{\bV} \big)  +  \frac{\nu}{32} |||\be_h^{\bu}|||^2_{1,\O}\\
     & \leq \frac{\nu}{16} |||\be_h^{\bu} |||_{1,\O}^2 + C \nu^{-1} \big(\|\bu\|^2_{\bV} + \|\be_I^{\bu} \|^2_{\bV} \big) \|\be_I^{\bu} \|^2_{\bV} + \\
 &  \quad  + \: C \nu^{-1}(\nu^{-1} \big(\|\bu\|^2_{\bV} + \|\be_I^{\bu} \|^2_{\bV} \big) \|\be_h^{\bu} \|_{0,\O})^2 +  \frac{\nu}{32} |||\be_h^{\bu}|||^2_{1,\O} +  \frac{\nu}{32} |||\be_h^{\bu}|||^2_{1,\O}\\
    & \leq \frac{\nu}{8} |||\be_h^{\bu} |||_{1,\O}^2 + C \nu^{-1} \big(\|\bu\|^2_{\bV} + \|\be_I^{\bu} \|^2_{\bV} \big) \|\be_I^{\bu} \|^2_{\bV} +  C \nu^{-3} \big(\|\bu\|^4_{\bV} + \|\be_I^{\bu}\|^4_{\bV} \big) \|\be_h^{\bu} \|^2_{0,\O}\\
   &  \leq \frac{\nu}{8} |||\be_h^{\bu} |||_{1,\O}^2 + C(\bu, \nu^{-1}) h^{2k} +C (\bu, \nu^{-3})|||\be_h^{\bu} |||^2_{0,\O}\,. 
\end{split}
\end{equation} 
Thus, from \eqref{Tc} and bounds \eqref{Tc1} and \eqref{eq:Tc2} we infer
\begin{equation}
\label{Tcfinal}
T_c \leq C(\bu, \nu^{-1}) h^{2k} +
C (\bu, \nu^{-3})|||\be_h^{\bu} |||^2_{0,\O} + 
\frac{\nu}{4} |||\be_h^{\bu} |||_{1,\O}^2 \,.
\end{equation}
The desired result follows by combining the estimates \eqref{Tf:Tm}, \eqref{Tcfinal} and \eqref{Td},  and inserting the resulting inequalities in \eqref{error:estima:NS-1}.
 
\end{proof}

We have the following key result dealing with the  discrete semilinear forms $r_h(\cdot;\cdot,\cdot)$ and $\bar{r}_h(\cdot;\cdot,\cdot)$. 
\begin{lemma}\label{lemma:rh}
Let Assumptions \ref{Assump:cont:phase}, \ref{Assump:mesh},  \ref{Assump:disc:phase}, \ref{assump:reg:add} and \ref{Assump:reg:aux} hold. 
Then for a.e. $t \in (0,T]$,  the following error estimate holds 
	\begin{equation*}
		\begin{split}
\bar{r}_h(\varphi; \mP_h\varphi, e_h^{\varphi}) -
r_h(\varphi_h; \varphi_h, e_h^{\varphi})   
\leq 
C(\varphi, C_T, \widetilde{C}_T) \| e_h^{\varphi}\|_{1,\O}^2 +
C(\varphi, f(\varphi), \varepsilon^{-2}, C_T, \widetilde{C}_T, C_\O) h^{\ell} \,| e_h^{\varphi}|_{1,\O}
 \,.
		\end{split}
	\end{equation*}	
\end{lemma}
\begin{proof}
By using the definitions of $\overline{r}_h(\cdot;\cdot,\cdot)$ and $r_h(\cdot;\cdot,\cdot)$, we have that
\begin{equation}
\label{eq:R1+R2}
\begin{aligned}
&
\overline{r}_h(\varphi; \mP_h \varphi, e_h^{\varphi})-
r_h( \varphi_h; \varphi_h, e_h^{\varphi})  =
\\
& = \sum_{\E \in \O_h}	
	\Bigl(
	(f'(\varphi) -f'(\PikK \varphi_h))\PikMoneKp \nabla \mP_h \varphi,
  \PikMoneKp  \nabla e_h^{\varphi} \Bigr)_E
- r_h( \varphi_h; e_h^{\varphi}, e_h^{\varphi}) 
\\
& = \sum_{\E \in \O_h}	
	\Bigl(
	3(\varphi^2 - (\PikK \varphi_h)^2)\PikMoneKp \nabla \mP_h \varphi,
  \PikMoneKp  \nabla e_h^{\varphi} \Bigr)_E
- r_h( \varphi_h; e_h^{\varphi}, e_h^{\varphi}) 
\\
& = \sum_{\E \in \O_h}	
	\Bigl(
	3(\varphi - \PikK \varphi_h)(\varphi + \PikK \varphi_h)\PikMoneKp \nabla \mP_h \varphi,
  \PikMoneKp  \nabla e_h^{\varphi} \Bigr)_E
- r_h( \varphi_h; e_h^{\varphi}, e_h^{\varphi}) 
\\
&=: R_1 + R_2 \,.
\end{aligned}
\end{equation}
We first analyse the term $R_1$. 
Employing the Cauchy-Schwarz inequality, the continuity of the $L^2$-projection, recalling Remark \ref{remark:projection}, Assumption \ref{Assump:cont:phase} and Assumption \ref{Assump:disc:phase}, 
using Proposition \ref{approx:prop:L2} and Lemma \ref{lemma:error:P_h} we have:
\begin{equation}
\label{eq:R1}
\begin{aligned}
R_1 & \lesssim 
\sum_{\E \in \O_h}	
 \|\varphi-\PikK \varphi_h \|_{0,E} \,(\|\varphi\|_{L^\infty(E)} + \|\PikK  \varphi\|_{L^\infty(E)}) \|\PikMoneKp \nabla \mP_h \varphi\|_{L^\infty(E)}
   \,  \|\PikMoneKp \nabla e_h^{\varphi}\|_{0,E}
\\
& \lesssim  
(\|\varphi\|_{L^\infty(\Omega)} + \|\varphi_h\|_{L^\infty(\Omega)}) 
\| \nabla \mP_h \varphi\|_{L^\infty(\Omega)}
\sum_{\E \in \O_h}	
 (\|\varphi- \varphi_h \|_{0,E} + \|({\rm I}-\PikK) \varphi \|_{0,E})  
 \,  | e_h^{\varphi}|_{1,E}
 \\
& \lesssim  
C(\varphi, C_T, \widetilde{C}_T)
\sum_{\E \in \O_h}	
 (\|e_h^{\varphi} \|_{0,E} + \|e_I^{\varphi} \|_{0,E} +
 h_E^{\ell+1} |\varphi|_{\ell+1, E})  
 \,  | e_h^{\varphi}|_{1,E}
 \\
& \lesssim  
C(\varphi, C_T, \widetilde{C}_T)	
 (\|e_h^{\varphi} \|_{0,\O} + \|e_I^{\varphi} \|_{1,\O} +
 h^{\ell+1}|\varphi |_{\ell+1,\O})  
 \,  | e_h^{\varphi}|_{1,\O}
 \\
& \lesssim  
C(\varphi, C_T, \widetilde{C}_T)	
 (\|e_h^{\varphi} \|_{0,\O} +
C(\varphi, f(\varphi), \varepsilon^{-2}, C_T, C_\O) h^{\ell})  
 \,  | e_h^{\varphi}|_{1,\O}
  \,.
\end{aligned}
\end{equation}
Using analogous arguments, the term $R_2$ can be bounded as follows:
\begin{equation}
\label{eq:R2}
R_2 \lesssim 
\sum_{\E \in \O_h}	
(\|\PikK \varphi_h\|_{L^\infty(E)}^2 + 1)
   \,  \|\PikMoneKp \nabla e_h^{\varphi}\|_{0,E}^2
   \lesssim  
 C(\widetilde{C}_T)\,  | e_h^{\varphi}|_{1,\Omega}^2 \,.
\end{equation}
The thesis follows collecting \eqref{eq:R1} and \eqref{eq:R2} in \eqref{eq:R1+R2}.
\end{proof}	

For the phase variable estimate, we will also use preliminary results in Subsections~\ref{subsec:prelimi} and \ref{subsec:Ritz}, along with Lemma~\ref{lemma:rh}.
\begin{lemma}[Error estimate for the phase equation]\label{lemma:phase:error}
Let Assumptions \ref{Assump:cont:phase}, \ref{Assump:mesh},  \ref{Assump:disc:phase}, \ref{assump:reg:add} and \ref{Assump:reg:aux} hold. 
Then for a.e. $t \in (0,T]$,  the following error estimate holds 
\begin{equation*}
	\begin{split}
\frac{1}{2} \partial_t ||| e_h^{\varphi}|||^2_{0,\Omega} +  \frac{3\gamma}{4}  |||e_h^{\varphi} |||^2_{2,\O} 
&	\leq 	 
C_{{\rm P}, 1} h^{2m} + 
C_{{\rm P}, 2} ||| \be_h^{\bu} |||^2_{0,\Omega} +
C_{{\rm P}, 3} |||e_h^{\varphi}|||_{0,\O}^2 \,,
	\end{split}
\end{equation*}	
where
$C_{{\rm P}, 1}= C(\bu, \varphi, f(\varphi), \gamma, \varepsilon^{-2}, C_T, \widetilde{C}_T, C_\O)$,
$C_{{\rm P}, 2}= C(\gamma^{-1},  \widetilde{C}_T)$,
$C_{{\rm P}, 3}= C(\varphi, \gamma, \varepsilon^{-2}, C_T, \widetilde{C}_T)$
are positive constants independent of $h$.
\end{lemma}

\begin{proof}	
Choosing  $\phi_h = e_h^{\varphi}\in \Hh$ in the error equation~\eqref{Cahn:error:eq}, we have
	\begin{equation}\label{Cahn:error:eq2}
		\begin{aligned}
		\frac{1}{2}\partial_t||| e^{\varphi}_h|||^2_{0,\O} &+ \gamma |||e^{\varphi}_h|||^2_{2,\O} 
		= \bigl(m(\partial_t \varphi, e_h^{\varphi})-m_{h}(\partial_t \mP_h \varphi, e_h^{\varphi}) \bigr) - \gamma \alpha m( e^{\varphi}_I, e_h^{\varphi})  \\
		&+   \gamma \varepsilon^{-2}\big(\bar{r}_h(\varphi; \mP_h\varphi, e_h^{\varphi}) -r_h(\varphi_h; \varphi_h, e_h^{\varphi})  \big) +
		\big(c(\bu;\varphi,e_h^{\varphi})-  c_h(\bu_h;\varphi_h,e_h^{\varphi}) \big)\\
		& =: S_{m1} + \gamma \alpha S_{m2}  +\gamma \varepsilon^{-2} S_r + S_c \,.  	
		\end{aligned}
	\end{equation}
For the term $S_{m1}$ we use 
the consistency of forms $m^{\E}_h(\cdot,\cdot)$,
item $ii)$ in  Proposition \ref{lemma:bound:ch}, and  \eqref{equiv:norms} along with Proposition \ref{approx:prop:L2} and Lemma \ref{lemma:error:P_h},
to derive
\begin{equation}\label{Im}
\begin{split}
S_{m1}
&= \sum_{\E\in\Oh} \big( m^{\E}( \partial_t \varphi- \Pi_{\E}^{0,\ell}(\partial_t  \ \varphi), e_h^{\varphi})- m^{\E}_h(\Pi_{\E}^{0,\ell}(\partial_t   \varphi)-\partial_t  \mP_h \varphi , e_h^{\varphi})
\big) \\
& \leq C
\sum_{\E\in\Oh} \bigl( \|({\rm I} - \Pi_{\E}^{0,\ell})\partial_t   \varphi \| + \| \partial_t e_I^{\varphi}\|_{0,E}  \bigr) \, \|e_h^{\varphi}\|_{0,E}
\\
& 
\leq C (h^{\ell+1}   |\partial_t  \varphi|_{\ell+1,\O} + \|\partial_t e_I^{\varphi}\|_{1,\O})
||| e_h^{\varphi}|||_{0,\O} 
\\
& \leq C(\varphi, f(\varphi), \varepsilon^{-2}, C_T, C_\O) h^{2\ell} + 
C ||| e_h^{\varphi}|||^2_{0,\O}\,.
\end{split}
\end{equation}
For the second term, employing similar arguments we obtain
\begin{equation}\label{Ivarphi}
	\begin{split}
	\gamma \alpha S_{m2} & = - \gamma \alpha m(e^{\varphi}_I, e_h^{\varphi}) 
	\leq \gamma \alpha\|e^{\varphi}_I\|_{0,\O} \|e^{\varphi}_h\|_{0,\O} 
	 \leq  
	C(\gamma) \|e^{\varphi}_I\|^2_{1,\O} +  C|||e^{\varphi}_h|||^2_{0,\O}
	\\		
	 & \leq C(\varphi, f(\varphi), \gamma, \varepsilon^{-2}, C_T, C_\O) h^{2 \ell} +  C|||e^{\varphi}_h|||^2_{0,\O} \,.
	\end{split}
\end{equation}	
For the third term we use Lemma~\ref{lemma:rh}, the Young inequality,
inequality \eqref{ineq:Hi} and the equivalences \eqref{equiv:norms},
 to obtain 
\begin{equation}\label{Ir}
	\begin{split}
	\gamma \varepsilon^{-2}	I_{r}
	  & \leq 
	 \gamma \varepsilon^{-2} C(\varphi, C_T, \widetilde{C}_T) (\| e_h^{\varphi}\|_{0,\O}^2 +	| e_h^{\varphi}|_{1,\O}^2) 
+  C(\varphi, f(\varphi), \gamma, \varepsilon^{-2}, C_T, \widetilde{C}_T, C_\O) h^{2\ell} 
\\
	  & \leq 
	  C(\varphi, \gamma, \varepsilon^{-2}, C_T, \widetilde{C}_T)
	  |||e_h^{\varphi}|||_{0,\O}^2 + 
	  C(\varphi, f(\varphi), \gamma, \varepsilon^{-2}, C_T, \widetilde{C}_T, C_\O) h^{2\ell}  + \frac{\gamma}{8} |||e_h^{\varphi}|||_{2,\O}^2 \,.
	\end{split}
\end{equation}	
The term $I_c$ requires further analysis. In order to control this term, we add and subtract $c_h(\bu;\varphi,e_h^{\varphi})$, which yields the following identity:
\begin{equation}
\label{eq:Ic-5.9}
\begin{split}
I_{c} 
& = 
 \big(c(\bu;\varphi,e_h^{\varphi})-c_h(\bu;\varphi,e_h^{\varphi})\big)  + 
\big( c_h(\bu;\varphi,e_h^{\varphi})-c_h(\bu_h;\varphi_h,e_h^{\varphi})\big)  =: I_{c1} + I_{c2} \,. 
	\end{split}
\end{equation}
For the first term from Lemma~\ref{lemma:crimen:c} and the Young inequality we have 
\begin{equation*}
\label{eq:Ic1-5.9}
	\begin{split}
I_{c1} &\leq 
C(\bu, \varphi) h^{m+1} \|e_h^{\varphi}\|_H
\leq C(\bu, \varphi, \gamma^{-1}) h^{2(m+1)} 
+\frac{\gamma}{16} |||e_h^{\varphi}|||^2_{2,\O}\,.
	\end{split} 
\end{equation*}	
On the other hand, for the second term we make use of the identities~\eqref{eq:split:velo} and~\eqref{eq:split:phase}, together with the skew-symmetry of $c_h(\cdot;\cdot,\cdot)$ and obtain
\[
I_{c2}
	= c_h(\bu; \varphi - \varphi_h, e_h^{\varphi}) + c_h(\bu - \bu_h; \varphi_h, e_h^{\varphi})  
	= c_h(\bu; e_I^{\varphi}, e_h^{\varphi}) - c_h(\be_h^{\bu}; \varphi_h, e_h^{\varphi}) +  c_h(\be_I^{\bu}; \varphi_h, e_h^{\varphi})\,.
\]
Thus, employing bounds \eqref{c-bound-1}, \eqref{c-bound-2}, \eqref{ineq:Hi}, Assumption \ref{Assump:disc:phase}, Lemma \ref{lemma:error:S_h} and Lemma \ref{lemma:error:P_h}, we derive
\begin{equation*}
\label{eq:Ic2-5.9}
\begin{aligned}
	I_{c2}
	&\leq C \| \bu\|_{\bV}  \| e_I^{\varphi}\|_{H}  \| e_h^{\varphi}\|_{H} + C\| \varphi_h\|_{W^{1, \infty}(\Omega)} 
	( \|\be_I^{\bu}\|_{0, \Omega} + \|\be_h^{\bu}\|_{0, \Omega} ) \|e_h^{\varphi}\|_{1, \Omega}
	\\
& \leq 	C(\bu, \varphi, f(\varphi), \varepsilon^{-2}, C_T, C_\O)  h^{\ell-1} \,  \| e_h^{\varphi}\|_{H} +
C\| \varphi_h\|_{W^{1, \infty}(\Omega)} ( h^{k+1}\|\bu\|_{k+1, \Omega} + \|\be_h^{\bu}\|_{0, \Omega} ) \|e_h^{\varphi}\|_{1, \Omega}
	\\
& \leq 
	C(\bu, \varphi, f(\varphi), \gamma^{-1}, \varepsilon^{-2}, C_T, C_\O)  h^{2(\ell-1)}  +   
C(\bu, \gamma^{-1}, \widetilde{C}_T) h^{2(k+1)} +
\\
& \quad + 
C(\gamma^{-1}, \widetilde{C}_T)
|||\be_h^{\bu}|||_{0, \Omega}^2 + 
C(\gamma) |||e_h^{\varphi}|||_{0, \Omega}^2  + 
\frac{\gamma}{16}|||e_h^{\varphi}|||_{2, \Omega}^2.
\end{aligned}	
\end{equation*}
By combining the above bounds in \eqref{eq:Ic-5.9} we have that 
\begin{equation}\label{Ic}
\begin{aligned}
I_c &\leq 
C(\bu, \varphi, f(\varphi), \gamma^{-1}, \varepsilon^{-2}, C_T, \widetilde{C}_T, C_\O)   h^{2m} +
\\
& \quad  + 
 C(\gamma^{-1}, \widetilde{C}_T) ||| \be_h^{\bu} |||^2_{0,\Omega} +
 C(\gamma) |||e_h^{\varphi}|||_{0, \Omega}^2  + 
\frac{\gamma}{8} |||e_h^{\varphi}|||^2_{2,\O}
\,.
\end{aligned}
\end{equation}	
The desired result follows by inserting the  estimates \eqref{Im}, \eqref{Ic}, \eqref{Ir}, \eqref{Ivarphi} into \eqref{Cahn:error:eq2}.
\end{proof}

The following results establish error estimates for the semi-discrete and fully discrete virtual schemes~\eqref{disc:form1} and~\eqref{fully:disc:schm}, respectively.
\begin{theorem}[Error estimate for the semi-discrete scheme]\label{converg:semidisc}
Let Assumptions \ref{Assump:cont:phase}, \ref{Assump:mesh},  \ref{Assump:disc:phase}, \ref{assump:reg:add} and \ref{Assump:reg:aux} hold. 
Let $(\bu,\varphi)$ be the solution  of problem~\eqref{form:kernel}, 
let $(\bu_h,\varphi_h)$ be the solution  of problem~\eqref{disc:form:kernel} and let $m= \min\{k, \ell-1\}$. Then the for a.e. $t\in (0,T]$ following error estimate holds	
\begin{equation*}
	\begin{split}
	&\|\bu- \bu_h\|_{L^{\infty}(0,t;\bL^2(\O))} +   \|\varphi- \varphi_h\|_{L^{\infty}(0,t;L^2(\O))} 	+\|\bu- \bu_h\|_{L^{2}(0,t;\bV)} +  \|\varphi- \varphi_h\|_{L^{2}(0,t;H)}  \leq 
\widetilde{C}_{{\rm est}}	
	h^{m} \,,
	\end{split}
\end{equation*}
where  $\widetilde{C}_{{\rm est}}$ is a positive constant depending only on the data, on the solution $(\bu,\varphi)$, on the physical parameters, on the constants $C_T$, $\widetilde{C}_T$, $C_\O$ in Assumptions
\ref{Assump:cont:phase}, \ref{Assump:disc:phase}, \ref{Assump:reg:aux}
and is independent of $h$.
\end{theorem}
\begin{proof}
We combine Lemmas~\ref{lemma:momentum:error} and~\ref{lemma:phase:error},  to obtain for a.e. $t \in (0, T]$
\begin{equation*}
	\begin{split}	
\frac{1}{2} \partial_t ||| \be_h^{\bu}|||^2_{0,\Omega} + \frac{1}{2} \partial_t ||| e_h^{\varphi}|||^2_{0,\Omega} &+ \frac{\nu}{2} ||| \be_h^{\bu}|||^2_{1,\O} +\frac{\gamma}{2}  |||e_h^{\varphi}||||^2_{2,\O}\\\
& \leq  
(C_{{\rm M}, 1} + C_{{\rm P}, 1})  h^{2m} + 
  (C_{{\rm M}, 2} + C_{{\rm P}, 2}) |||\be_h^{\bu}|||^2_{0,\Omega} +  C_{{\rm P}, 3} |||e_h^{\varphi}|||^2_{0,\Omega} .
	\end{split}
\end{equation*}
Then, the estimate for the velocity and phase fields follows from  the Gronw\"{a}ll inequality (cf. Lemma \ref{cont:gronwall}), the identities \eqref{eq:split:velo}, \eqref{eq:split:phase}, triangular inequality, along with the approximation properties of the operators $\mS_h$ and $\mP_h$ (cf. Lemmas~\ref{lemma:error:S_h} and~\ref{lemma:error:P_h}).

\end{proof}	

We now state a convergence result for the pressure variable. The proof is only sketched, as it relies on classical Navier–Stokes arguments.

\begin{theorem}
[Error estimate for the pressure variable]
\label{thm:pressure}
Let Assumptions \ref{Assump:cont:phase}, \ref{Assump:mesh},  \ref{Assump:disc:phase}, \ref{assump:reg:add} and \ref{Assump:reg:aux} hold. 
Let  $(\bu, \widehat{p}, \varphi)$ be the solution  of problem~\eqref{form1}, let $(\bu_h, \widehat{p}_h, \varphi_h)$ be the solution  of problem~\eqref{disc:form1} and let $m= \min\{k, \ell-1\}$. 
Assume further that $\widehat{p} \in L^2(0, T; H^k(\O))$.
Then the for a.e. $t\in (0,T]$ following error estimate holds	
\[
\|\widehat{p}-\widehat{p}_h\|_{L^{2}(0,t;Q)} \leq 
\widetilde{C}_{{\rm est}, p} h^{m} 
\]
where  $\widetilde{C}_{{\rm est}, p}$ is a positive constant depending only on the data, on the solution $(\bu, \widehat{p}, \varphi)$, on the physical parameters, on the constants $C_T$, $\widetilde{C}_T$, $C_\O$ in Assumptions
\ref{Assump:cont:phase}, \ref{Assump:disc:phase}, \ref{Assump:reg:aux}
and is independent of $h$.
\end{theorem}

\begin{proof}
The desired bound follows from the error bounds for the velocity and phase variables in Theorem \ref{converg:semidisc}, along with the inf-sup condition~\eqref{eq:infsup}, the bounds in Proposition \ref{lemma:bound:ch},
the approximation properties in Proposition \ref{approx:virtual:velo:phase}.
\end{proof}

As already mentioned, the main novelty of the present paper is the spatial discretization, thus we focused mainly on the first
(spatial) source of error, as shown in Theorems \ref{converg:semidisc} \ref{thm:pressure}. However, in order to detail the behaviour of the fully-discrete method, we state the following result.

\begin{theorem}[Error estimate for the fully-discrete scheme]\label{converg:fullydisc}
Let Assumptions \ref{Assump:cont:phase}, \ref{Assump:mesh},  \ref{Assump:disc:phase}, \ref{assump:reg:add} and \ref{Assump:reg:aux} hold. 
Let $(\bu, \widehat{p}, \varphi)$ be the solution  of problem~\eqref{form1}, let $\{(\bu_h^n, \widehat{p}_h^n, \varphi_h^n)\}_{n=1}^N$ be the solution  of problem~\eqref{fully:disc:schm1} and let $m= \min\{k, \ell-1\}$. 
Assume further that $\widehat{p} \in L^2(0, T; H^k(\O))$.
Then the for all $n=1, \dots, N$  the following error estimate holds	
\[
\begin{aligned}
& \|\bu- \bu_h\|_{\ell^{\infty}(0,t_n;\bL^2(\O))} +   \|\varphi- \varphi_h\|_{\ell^{\infty}(0,t_n;L^2(\O))} 	+\\
& \quad + \|\bu- \bu_h\|_{\ell^{2}(0,t_n;\bV)} +  \|\varphi- \varphi_h\|_{\ell^{2}(0,t_n;H)}  \leq 
\widehat{C}_{{\rm est}}	
(h^{m} + \tau) 
\\
&\|\widehat{p}-\widehat{p}_h\|_{\ell^2(0,t_n;Q)} \leq 
\widehat{C}_{{\rm est}, p} (h^{m} + \tau) 
\end{aligned}
\]
where  $\widehat{C}_{{\rm est}}$  and $\widehat{C}_{{\rm est}, p}$  are positives constant depending only on the data, on the solution $(\bu, \widehat{p}, \varphi)$, on the physical parameters, on the constants $C_T$, $\widetilde{C}_T$ $C_\O$ in Assumptions
\ref{Assump:cont:phase} and Assumption \ref{Assump:f-disc:phase}, \ref{Assump:reg:aux}
and are independent of $h$ and $\tau$.
\end{theorem}

\begin{proof}
The  proof follows from the results established in the previous sections,
combined with standard arguments for fully discrete schemes
\cite{Vacca-parabolico,VK2023}, and the use of Young’s inequality together with
the discrete Gr\"{o}nwall inequality (cf. Lemma~\ref{discrete:gronwall}).
\end{proof}

We conclude this section with two remarks concerning the optimal convergence rates and further properties of the proposed schemes.
\begin{remark}[Optimal error bounds and mass conservation property]\label{rm:orders}
It is observed that the error estimates for the velocity and the phase variable established in Theorem~\ref{converg:semidisc} do not depend directly on the pressure variable. This follows from the fact that the proposed numerical scheme is exactly divergence-free, a property that is enforced at the discrete level through the construction of the virtual element velocity space (cf. Section~\ref{subsec:space:velo-pressure}).

Moreover, in order to achieve optimal convergence orders for all variables (cf. Theorems~\ref{converg:semidisc} and~\ref{converg:fullydisc}) and to ensure mass conservation (cf. Theorems \ref{theo:property:semi} and ~\ref{theo:property:fully}), a natural choice of the polynomial degrees is given by $m=k=\ell-1$. In the numerical experiments, the scheme is implemented using the pairs $(k,\ell)=(1,2)$ and $(k,\ell)=(2,3)$.
\end{remark}

\begin{remark} [Chemical potential] 
\label{rm:chemical}
At the continuous level, the \emph{chemical potential} is defined as
\[
w = -\Delta \varphi + \varepsilon^{-2} f(\varphi).
\]
We note that the virtual element scheme~\eqref{fully:disc:schm} does not provide a direct approximation of this quantity. Nevertheless, using the tools developed in this work, it is possible to construct an approximation of the chemical potential at each time level $t_n$, for $n=1,\ldots,N$. More precisely, we introduce the following projection-based postprocessing formula for the discrete chemical potential:
\begin{equation*}
	\widetilde{w}_h^n
	= - \Pi_{h}^{\ell-2} \Delta \varphi_h^n
	+ \varepsilon^{-2} f\!\left( \Pi_{h}^{\ell} \varphi_h^n \right).
\end{equation*}
By exploiting the approximation properties of the projections, together with the convergence result of Theorem~\ref{converg:fullydisc}, we obtain
\[
\|w-\widetilde{w}_h\|_{\ell^2(0,t_n;L^2(\Omega))}
\leq \widehat{C}_{{\rm est}, w} h^{\min\{k,\ell-1\}} + \tau,
\]
where the positive constant $\widehat{C}_{{\rm est}, w}$  depends on the same parameters as in Theorem~\ref{converg:fullydisc}.
\end{remark}

%% file: 06_numericalResults.tex
\section{Numerical experiments}\label{sec:numericalResults}
In this section, we present two numerical experiments that validate the theoretical properties and the practical performance of the proposed method. 
In the first test we exhibit that the method provides the expected optimal convergence order, whereas in the second test  we validate the proposed scheme on a benchmark problem.

We recall that the time discretization is carried out using the backward Euler 
method. Accordingly, for each time step $n = 1,\dots,N$, the resulting nonlinear 
system is solved via a full Newton iteration. The stopping criterion is based on 
the $\ell^{\infty}$-norm of the relative residual, with a fixed tolerance of \(10^{-6}\).
In the numerical tests, in line with Remark \ref{rm:orders}, we set $(k,\ell)=(1,2)$ and $(k,\ell)=(2,3)$.
\subsection{Test 1. History convergence}\label{Test1} 
This numerical experiment is a non-physical test designed to verify the convergence of the fully-discrete scheme \eqref{fully:disc:schm1}. 
To this end, we assess the method's accuracy by using  the following computable error quantities:
	\begin{equation}\label{compu:errors}
		\begin{split}
{\tt err}(\bu_h,H^1)&:= 
\frac{\left(\sum_{E \in \Omega_h}\| \bnabla \bu(\cdot, T)- \boldsymbol{\Pi}^{0,k-1}_E  \bnabla  \bu_h(\cdot, T)\|_{0,E}^2\right)^{1/2}}{\|\bnabla \bu(\cdot, T)\|_{0,\O}},	
\\
  	{\tt err}(\widehat{p}_h,L^2) &:= \frac{\|\widehat{p}(\cdot, T)-  \widehat{p}_{h}(\cdot, T)\|_{0,\O}}{\|\widehat{p}(\cdot, T)\|_{0,\O}},
  	\\
	{\tt err}(\varphi,H^2)&:= 
	\frac{\left(\sum_{E \in \Omega_h}\| \bD^2 \varphi(\cdot, T)- \boldsymbol{\Pi}^{0,\ell-2}_E  \bD^2  \varphi_h(\cdot, T)\|_{0,E}^2\right)^{1/2}}{\|\bD^2 \varphi(\cdot, T)\|_{0,\O}}.	
		\end{split}
	\end{equation}
Notice that the convergence results in Section \ref{sec:error:analysis} yield error bounds in the space-time norm, and therefore do not provide explicit bounds for all the aforementioned error quantities.
However the best rate of convergence that one can
expect in the above norms corresponds to the one given by the interpolation estimates in Proposition \ref{approx:virtual:velo:phase}.

We consider Problem \eqref{model_a} on the square domain $\O = (0,1)^2$ and final time $T=1$.  
The physical parameters are all set to $1$, and the forcing terms (the right hand side of the third equation in \eqref{model_a} is not set to zero), the boundary and the initial conditions are are chosen according to the exact analytical solution
\begin{equation*}
	\begin{split}
\bu(x,y,t)&= \begin{pmatrix}
 \quad (t+1)\, \exp(x) \left( \sin(y) + y\cos(y) - x\sin(y) \right) \\
- (t+1)\, \exp(x) \left( \cos(y) + x\cos(y) + y\sin(y) \right)
\end{pmatrix}, \quad
p(x,y,t)= (2t+1)\: \sin(\pi x) \: \cos(4 \pi y), \\ 
\varphi(x,y,t) &= (t+1) \:  \cos(2\pi x)\: \cos(2\pi y).
	\end{split}
\end{equation*}
An affine dependence on the time variable $t$ was adopted to minimize the influence of the temporal error and to better assess the spatial error.

The domain $\Omega$ is partitioned using different
families of polygonal meshes (see Figure~\ref{meshes}): (a) \texttt{TRIANGULAR MESHES}; 
(b) \texttt{QUADRILATERAL MESHES};  (c) \texttt{VORONOI MESHES}.
\begin{figure}[!h]
	\centering
	\subfloat[\texttt{TRIANGULAR MESH}]{
		\includegraphics[width=0.30\textwidth]{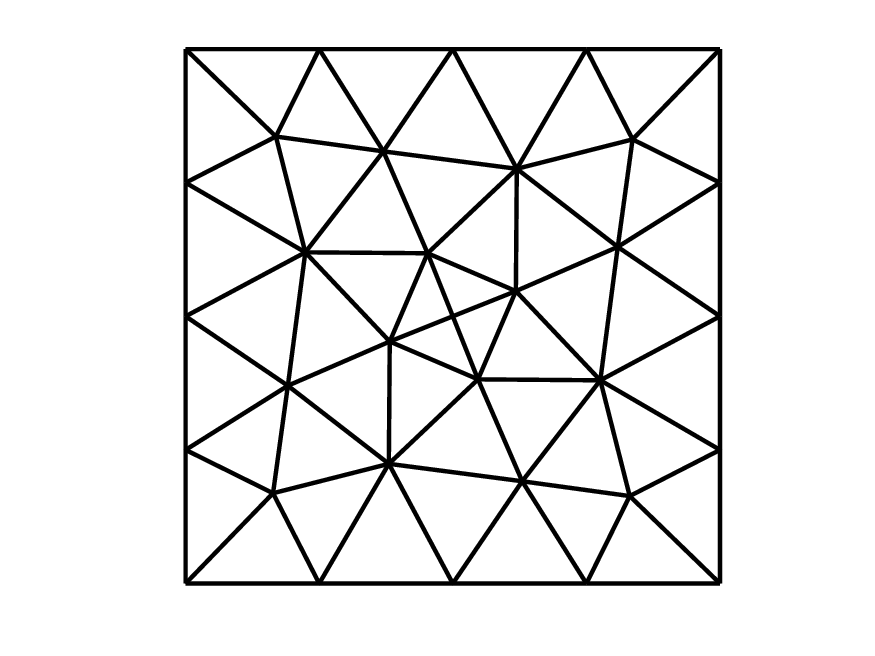}
		\label{mesh:tri}
	}\hfill
	\subfloat[\texttt{QUADRILATERAL MESH}]{
		\includegraphics[width=0.30\textwidth]{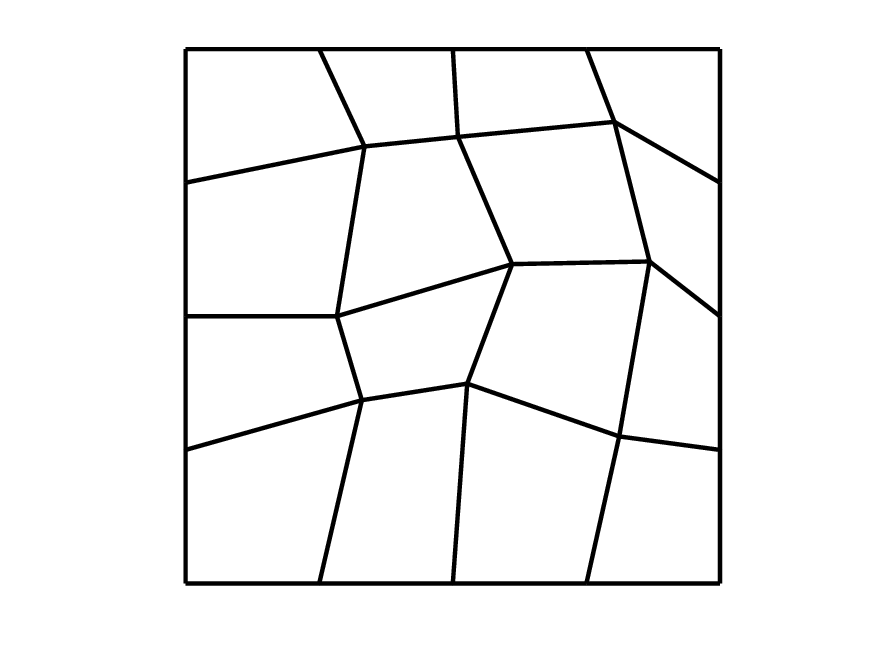}
		\label{mesh:quad}
	}\hfill
	\subfloat[\texttt{VORONOI MESH}]{
		\includegraphics[width=0.30\textwidth]{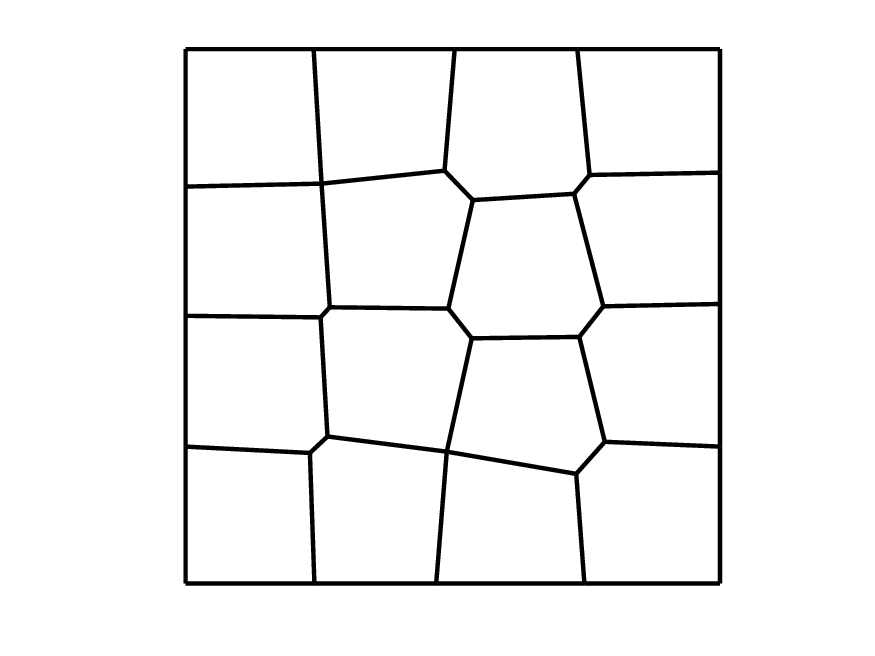}
		\label{mesh:voro}
	}
	\caption{Test 1. Sample of the adopted meshes.}
	\label{meshes}
\end{figure}
The spatial and 
temporal discretizations are refined simultaneously. 
Specifically, for each mesh family we consider the refinements 
$h = 1/4$, $1/8$, $1/16$, $1/32$, and we apply the same uniform refinements $ \tau$ in time. 
In Figure \ref{plot-lines-converg}   we display the errors ${\tt err}(\bu_h,H^1)$, ${\tt err}(\widehat{p}_h,L^2)$ and ${\tt err}(\varphi,H^2)$
for the fully-discrete scheme \eqref{fully:disc:schm1} with
$(k,\ell)=(1,2)$ and $(k,\ell)=(2,3)$
for the sequences of aforementioned meshes.
\begin{figure}[!h]
	\begin{center}
		\includegraphics[height=4.2cm, width=6.5cm]{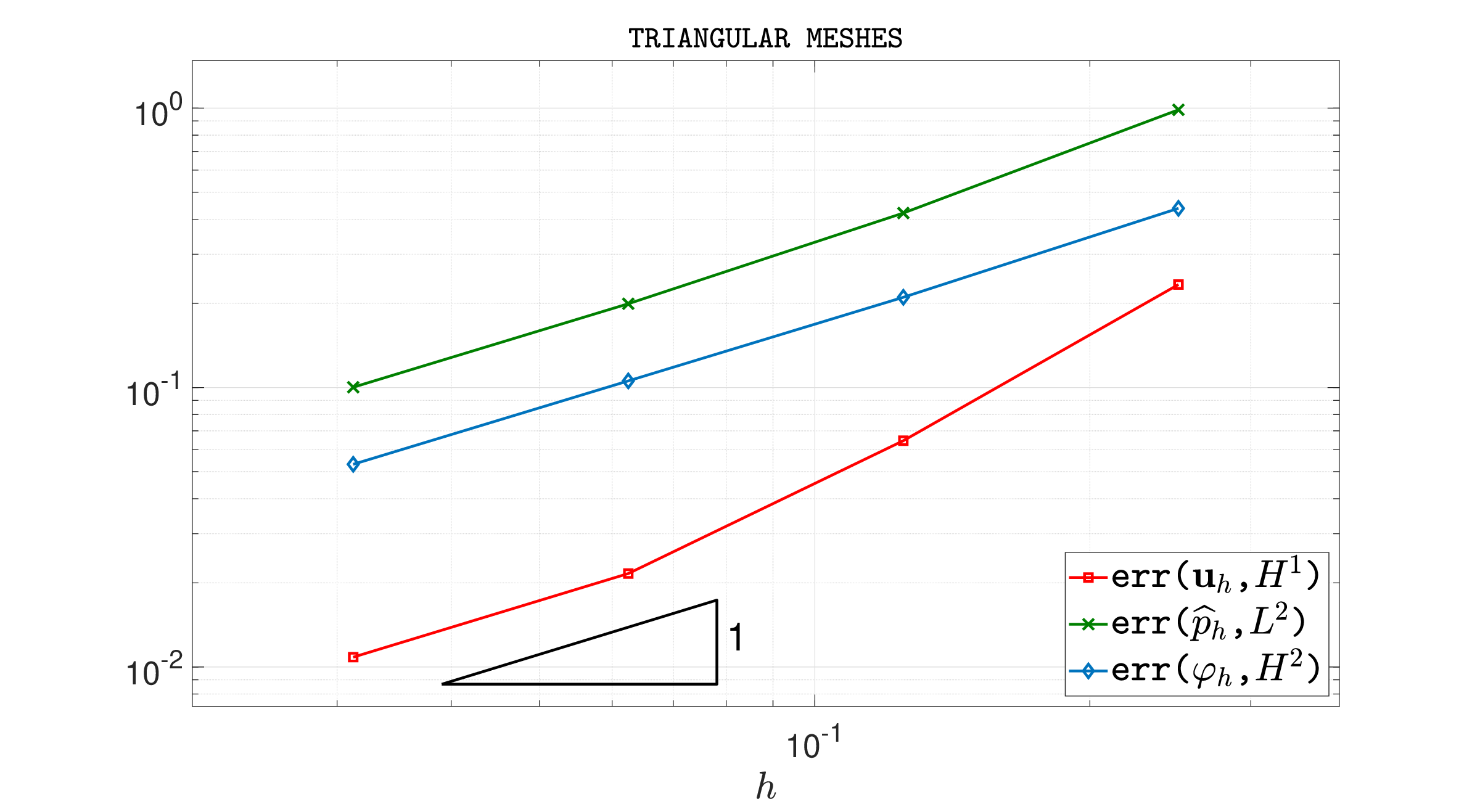}\hspace*{-0.40cm}
	  \includegraphics[height=4.2cm, width=6.5cm]{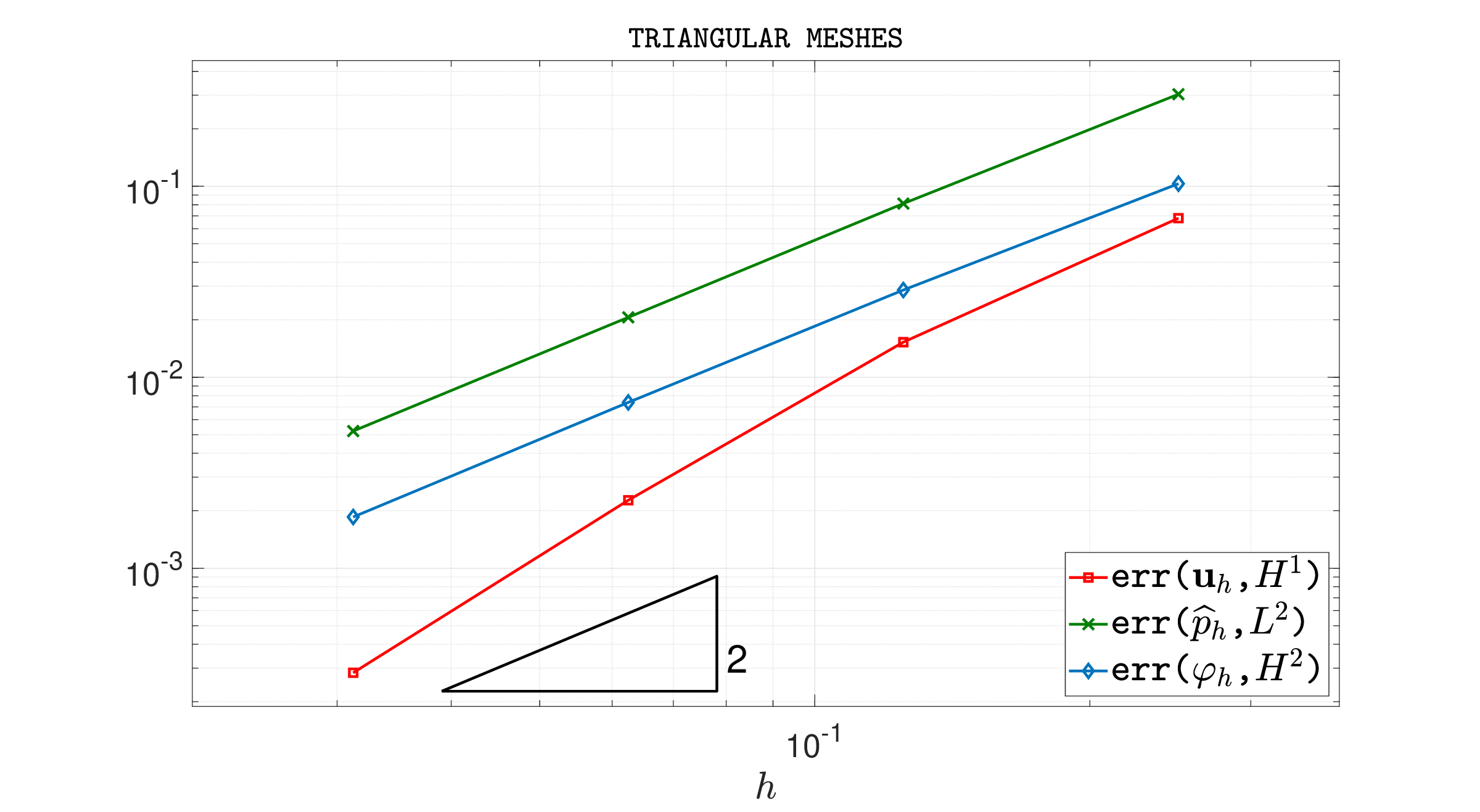}\\  	
		\includegraphics[height=4.2cm, width=6.5cm]{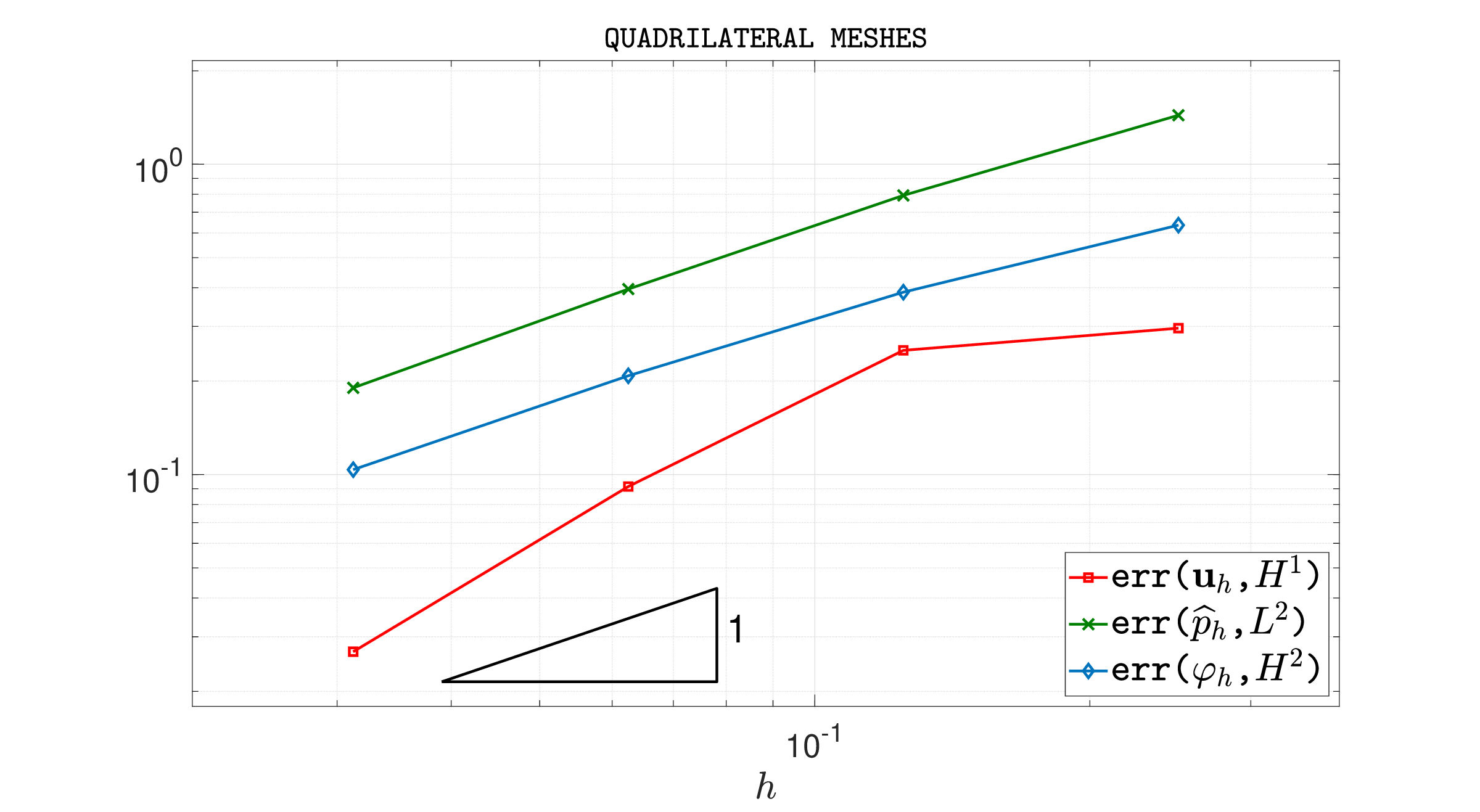}\hspace*{-0.40cm}
		\includegraphics[height=4.2cm, width=6.5cm]{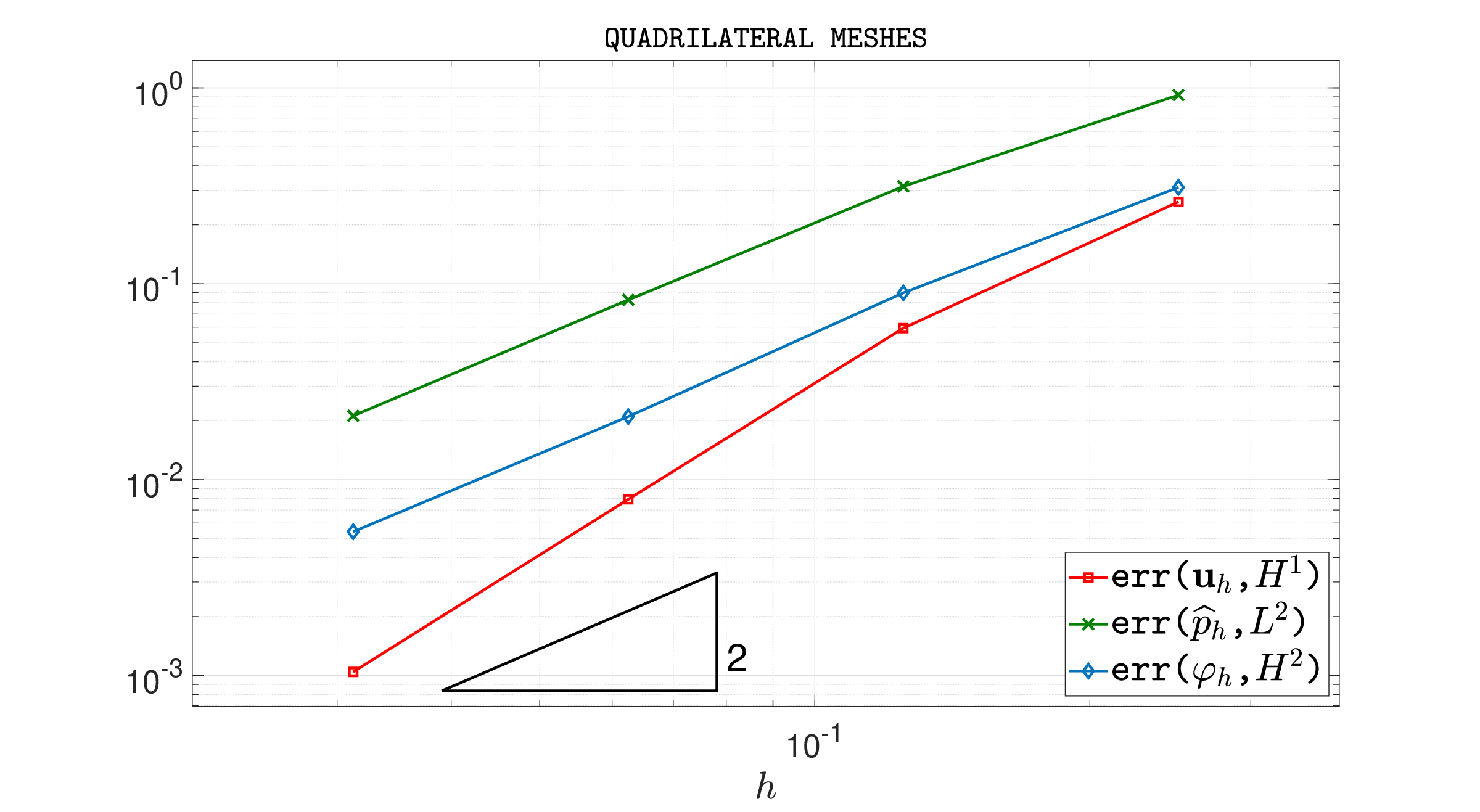}\\
		\includegraphics[height=4.2cm, width=6.5cm]{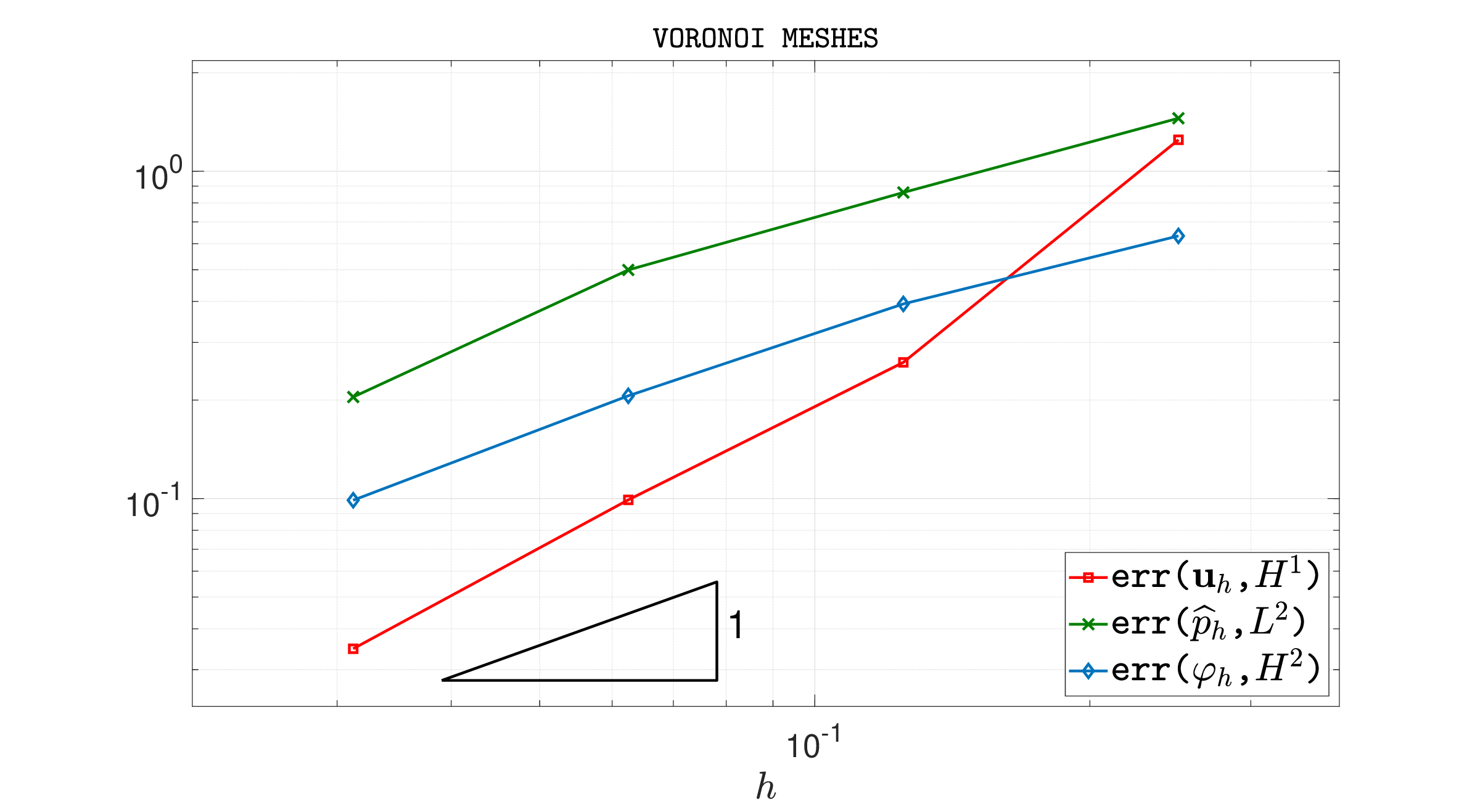}\hspace*{-0.40cm}
		\includegraphics[height=4.2cm, width=6.5cm]{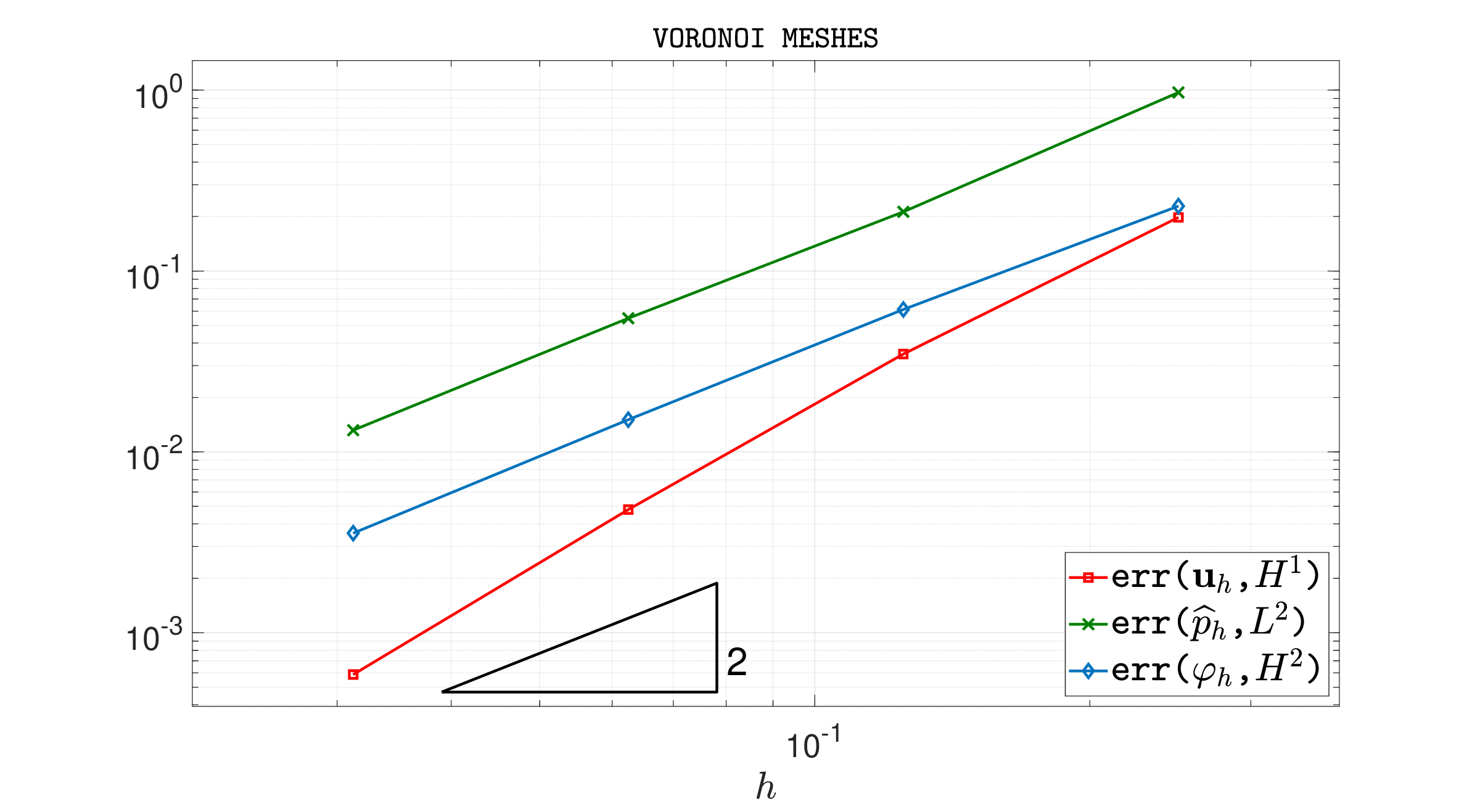}
	\end{center}
\caption{Test 1. Convergence histories of the VEM errors \eqref{compu:errors} at final time $T=1$. 
	The left panel reports the results for $(k,\ell)=(1,2)$, whereas the right panel 
	corresponds to $(k,\ell)=(2,3)$.}
	\label{plot-lines-converg} 
\end{figure}
The results exhibit excellent agreement with the optimistic bound in Proposition \ref{approx:virtual:velo:phase}, namely 
$\mathcal{O}(h)$ and $\mathcal{O}(h^2)$ for 
$(k,\ell)=(1,2)$ and $(k,\ell)=(2,3)$, respectively.

\subsection{Test 2. An elliptical fluid bubble}
Following the configuration of Problem \eqref{model_a} in~\cite[Test 1]{Feng2006}, we consider  the domain $\Omega = (-0.4,0.4)^2$ and set the physical parameters to $\nu = 1$, $\lambda = \gamma = 0.1$, and $\varepsilon = 2 \times 10^{-2}$. The initial velocity is taken as $\bu_0 = \mathbf{0}$, while the initial phase field is defined by
\[
\varphi_0(x_1,x_2)
= \tanh\!\left(\frac{x_1^2}{0.01}+\frac{x_2^2}{0.0225}-1\right).
\]
To adequately resolve the diffuse interface, we use a time step $\tau =  10^{-6}$ and a uniform Cartesian mesh with $64^2$ elements.  
The final simulation time is set to $T = 2 \times 10^{-4}$, and we employ polynomial degrees $(k,\ell) = (1,2)$.

The zero level set of $\varphi_0$, which characterizes the initial interface between the two fluids, corresponds to the ellipse
\[
\frac{x_1^2}{0.01}+\frac{x_2^2}{0.0225}=1,
\]
representing an elliptical bubble immersed in a surrounding fluid.

Figure~\ref{fig:snapshops:phase:Test2} shows snapshots of the computed phase field $\varphi_h^n$ corresponding to six selected times, namely $t_n=5 \times 10^{-6}, 5 \times 10^{-5}, 1\times  10^{-4}, 1.5 \times 10^{-4}, 1.75 \times 10^{-4}, 2 \times 10^{-4}$. In these plots, the red color represents the value $\varphi_h^n = 1$, while the blue color corresponds to $\varphi_h^n = -1$. These visualizations allow us to follow the evolution of the interface and to appreciate the qualitative behavior of the solution over time. We observe that the initial elliptical bubble rapidly evolves into a more circular shape. This behavior is in excellent agreement with the results reported in \cite{Feng2006}. 
\begin{figure}[h!]
	\centering
\subfloat[$\varphi_h=\varphi_0$ at $t_n=5 \times 10^{-6}$]{\includegraphics[width=0.3\linewidth, height=0.25\textwidth]{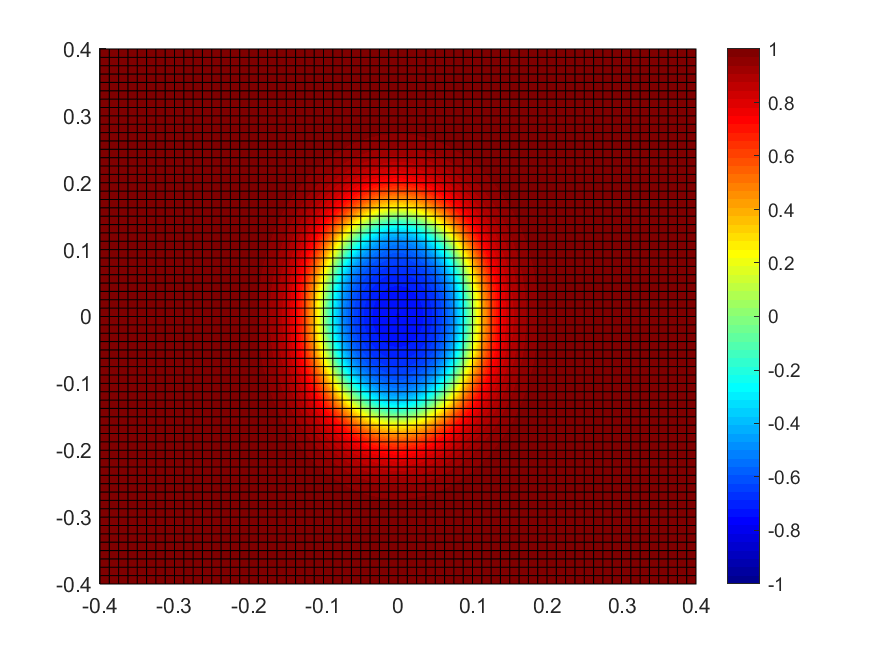}}\hfill
\subfloat[$\varphi_h$ at $t_n=5 \times 10^{-5}$]{\includegraphics[width=0.3\linewidth, height=0.25\textwidth]{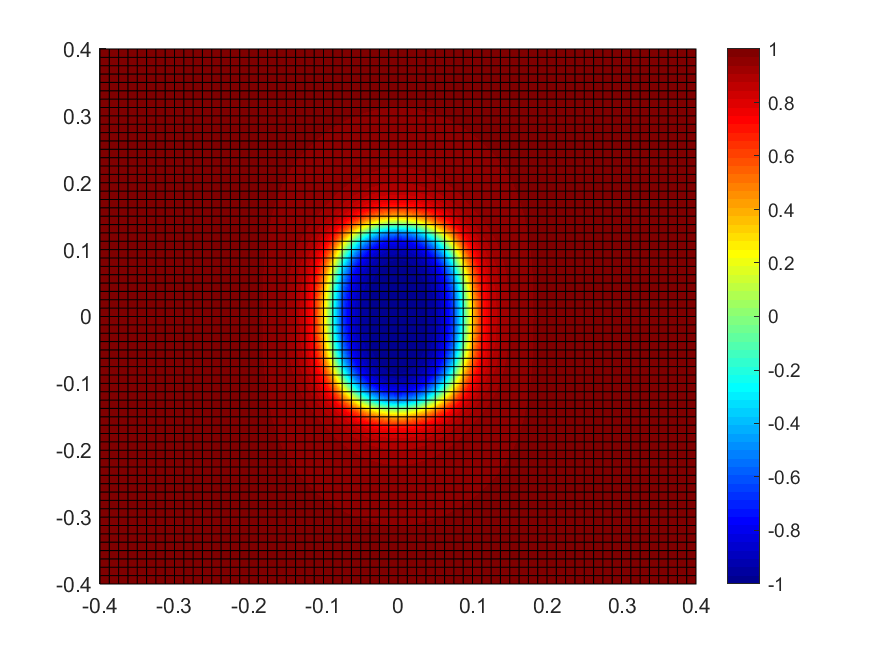}}\hfill
\subfloat[$\varphi_h$ at $t_n=1 \times 10^{-4}$]{\includegraphics[width=0.3\linewidth, height=0.25\textwidth]{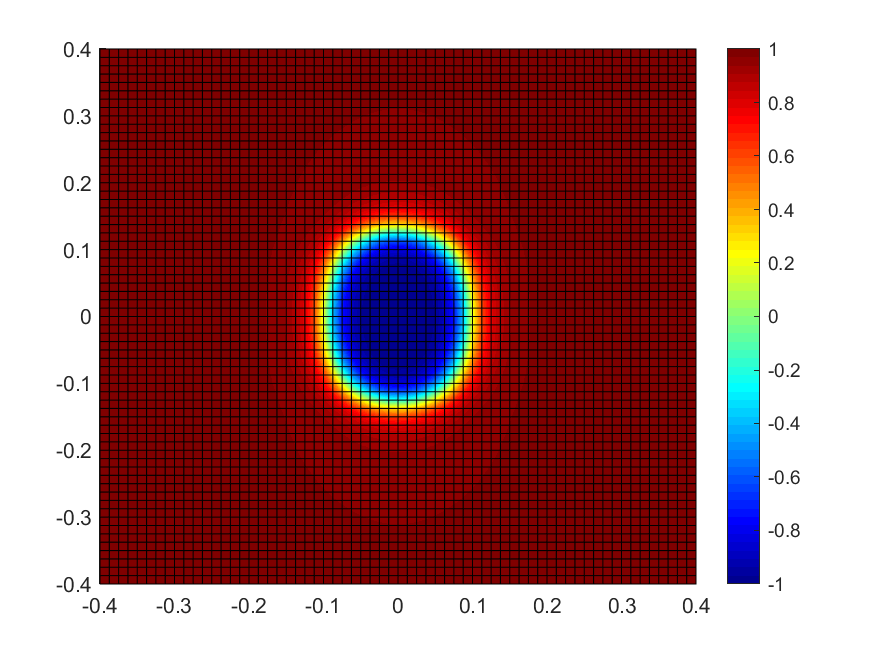}}\\	
\subfloat[$\varphi_h$ at $t_n=1.5\times 10^{-4}$]{\includegraphics[width=0.3\linewidth, height=0.25\textwidth]{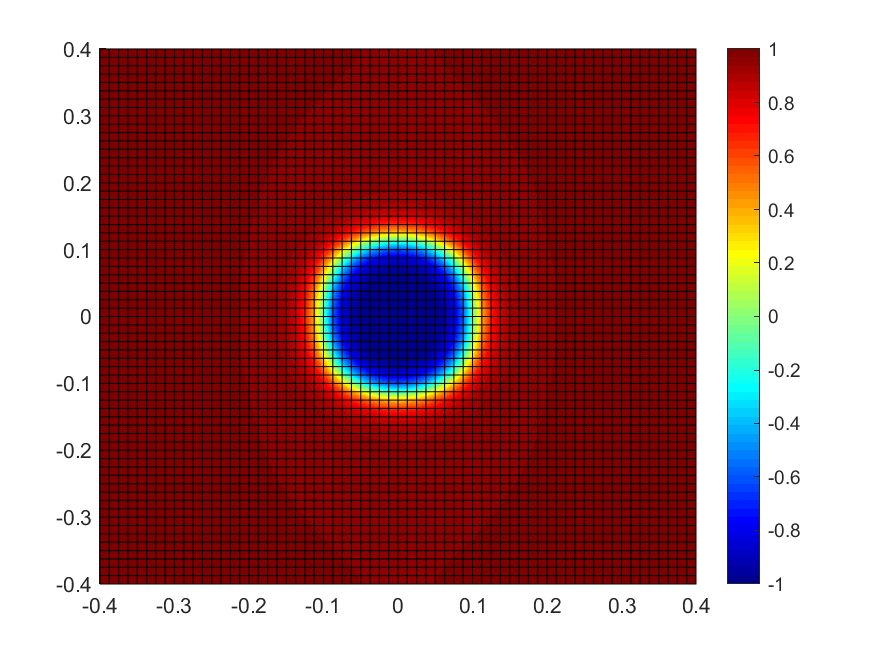}}\hfill
\subfloat[$\varphi_h$ at $t_n=1.75 \times 10^{-4}$]{\includegraphics[width=0.3\linewidth, height=0.25\textwidth]{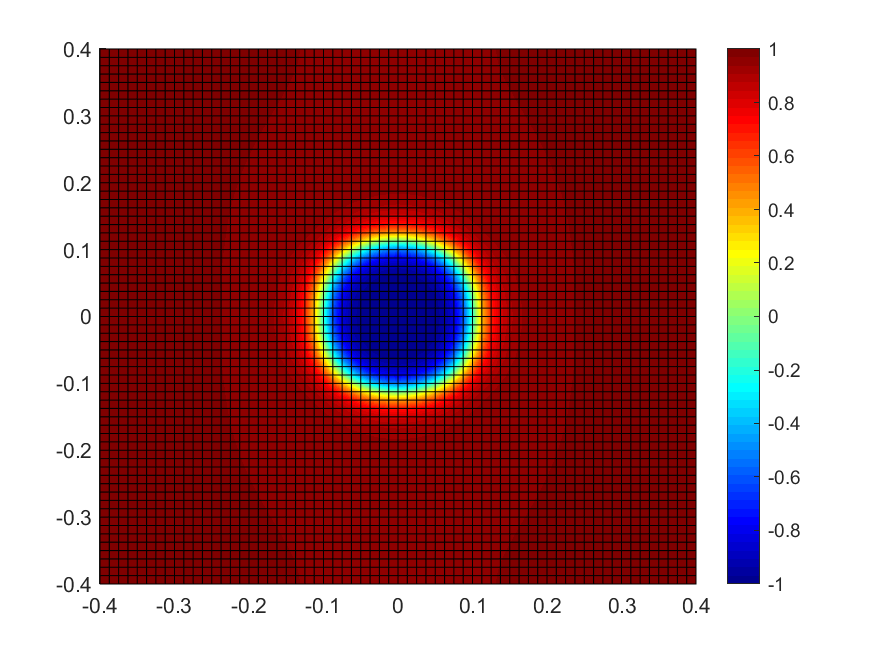}}\hfill
\subfloat[$\varphi_h$ at $t_n=2 \times 10^{-4}$]{\includegraphics[width=0.3\linewidth, height=0.25\textwidth]{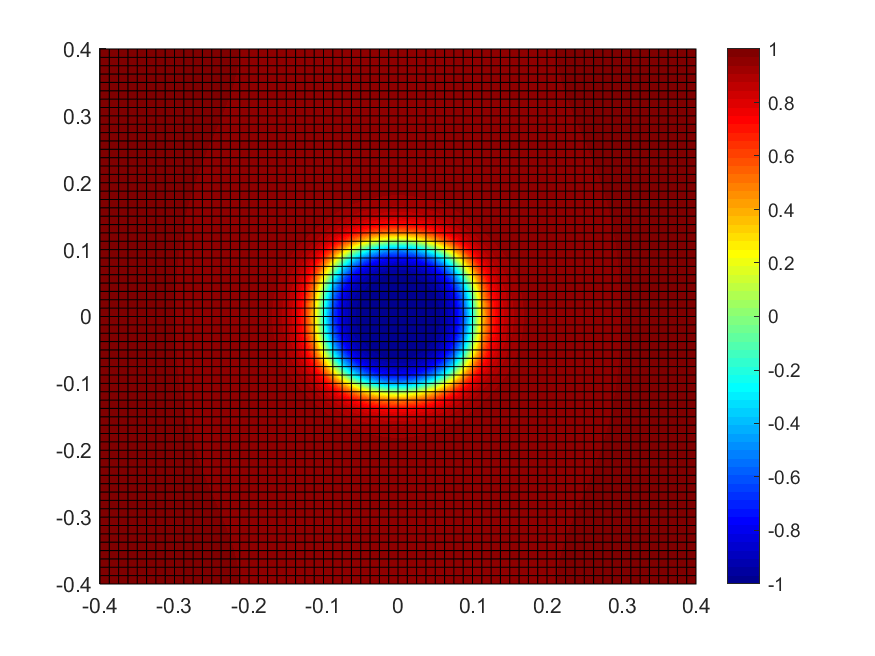}}
\caption{Test 2. Snapshots of the numerical phase solutions $\varphi^n_h$  at different time instants $t_n$, using the fully-discrete VE scheme~\eqref{fully:disc:schm}, $\nu=1$, $\lambda=\gamma=0.1$, $\varepsilon=2 \times 10^{-2}$, and $(k,\ell)=(1,2)$.}
	\label{fig:snapshops:phase:Test2}
\end{figure}

On the other hand, Figure~\ref{fig:snapshops:velo:Test2} presents snapshots of the arrow and streamline plots of the computed velocity field $ \bu^n_h$ at the time steps listed above. We observe that fluid vortices begin to form shortly after the initial time step, and these vortices become more pronounced as time progresses. Once again, these results are in complete agreement with those reported in \cite{Feng2006}. Moreover, Figure~\ref{fig:snapshops:ph:Test2} displays the approximated pseudo-pressure field at the same time instances. 

\begin{figure}[h!]
	\centering
\subfloat[$\bu_h$ at $t_n=5 \times 10^{-6}$]{\includegraphics[width=0.3\linewidth, height=0.25\textwidth]{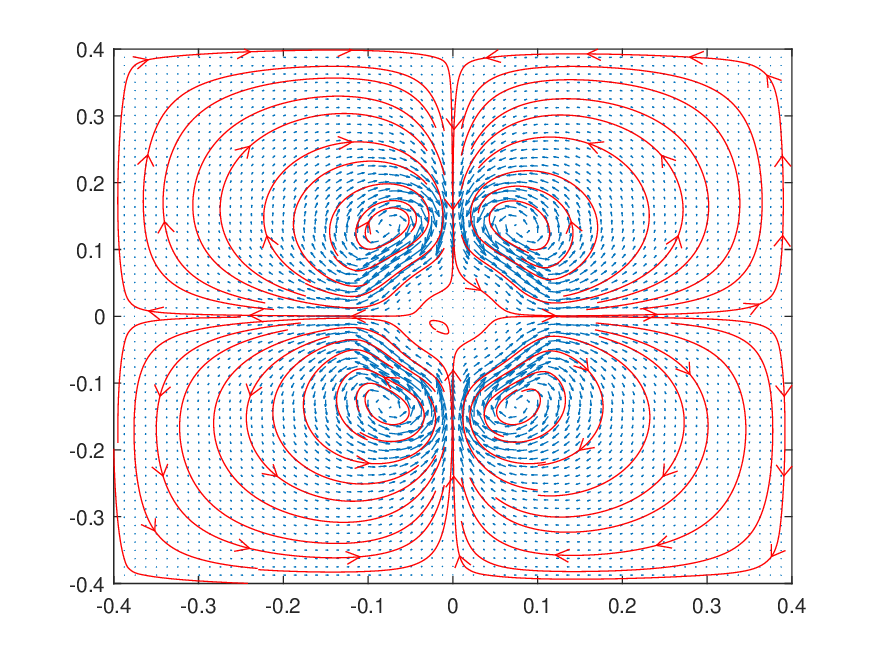}}\hfill
\subfloat[$\bu_h$ at $t_n=5 \times 10^{-5}$]{\includegraphics[width=0.3\linewidth, height=0.25\textwidth]{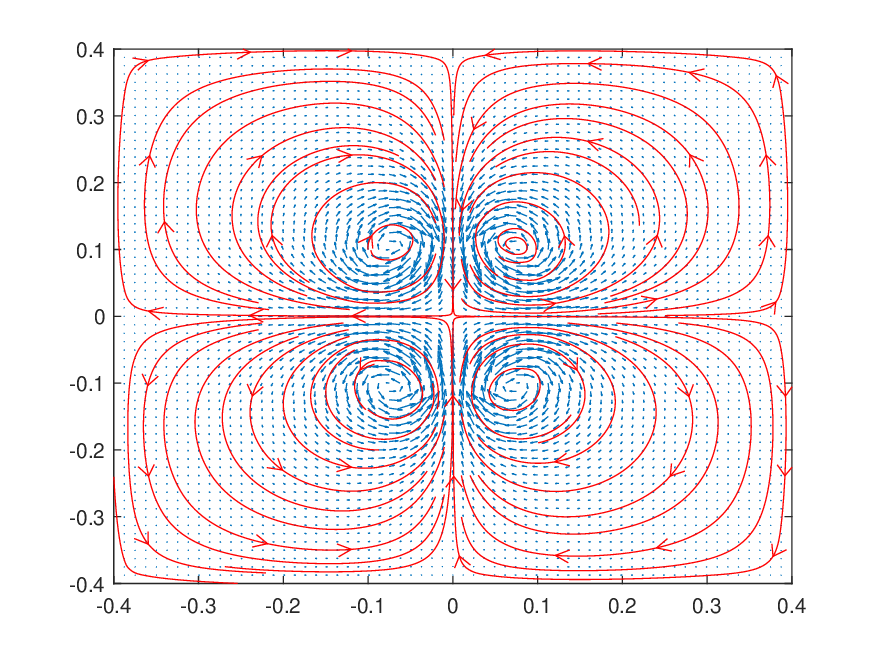}}\hfill
\subfloat[$\bu_h$ at $t_n=1 \times 10^{-4}$]{\includegraphics[width=0.3\linewidth, height=0.25\textwidth]{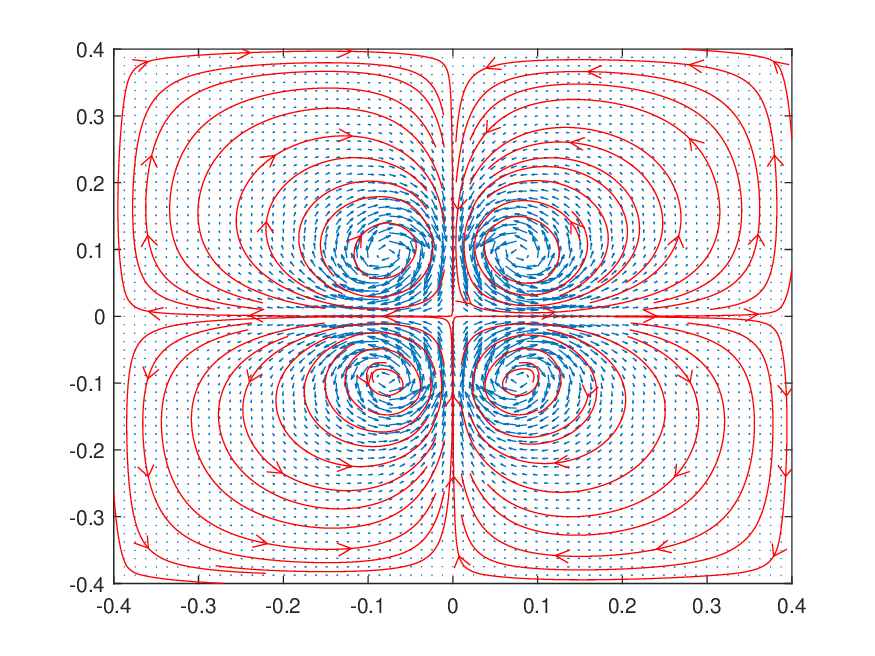}}\\	
\subfloat[$\bu_h$ at $t_n=1.5 \times 10^{-4}$]{\includegraphics[width=0.3\linewidth, height=0.25\textwidth]{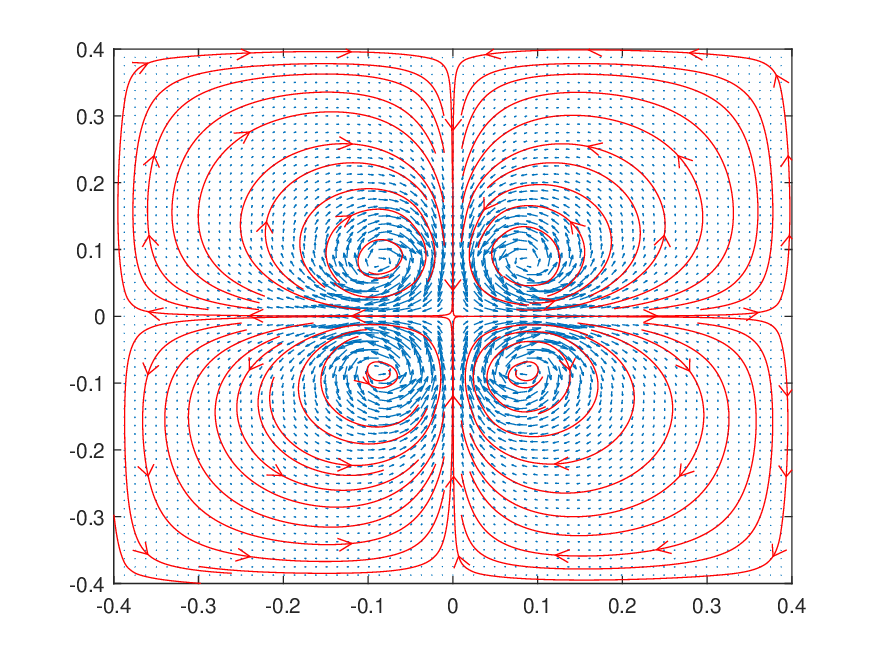}}\hfill
\subfloat[$\bu_h$ at $t_n=1.75 \times 10^{-4}$]{\includegraphics[width=0.3\linewidth, height=0.25\textwidth]{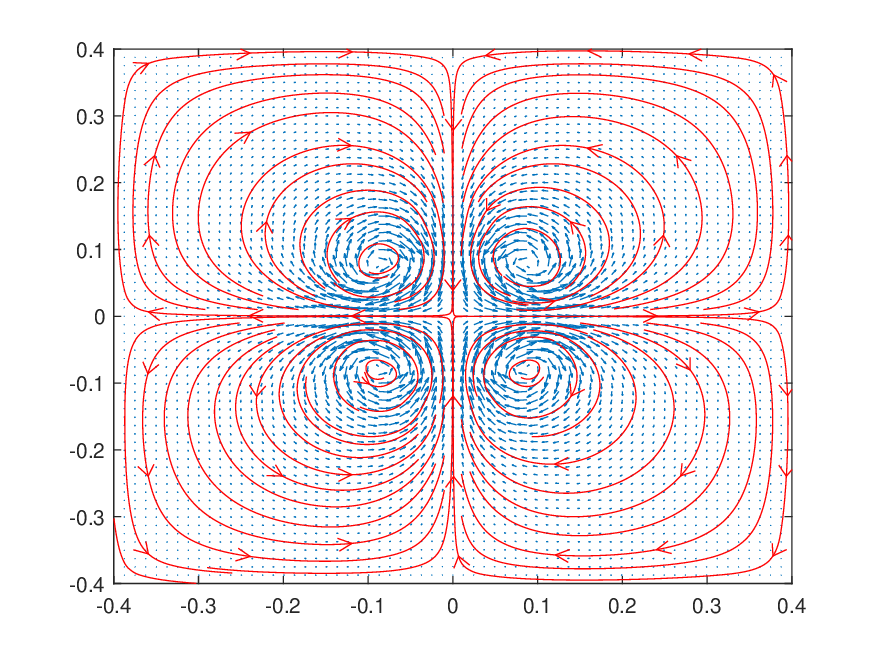}}\hfill
\subfloat[$\bu_h$ at $t_n=2 \times 10^{-4}$]{\includegraphics[width=0.3\linewidth, 
height=0.25\textwidth]{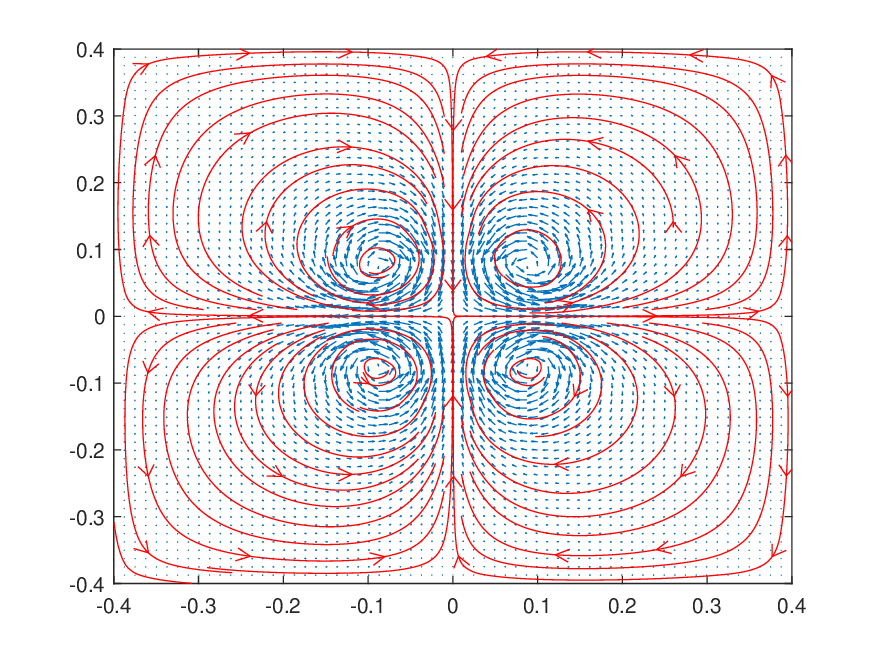}}
	\caption{Test 2. Snapshots and streamlines of the numerical velocity solutions $\bu^n_h$ at different time instants, using the fully-discrete VE scheme~\eqref{fully:disc:schm}, $\nu=1$, $\lambda=\gamma=0.1$, and $\varepsilon=2 \times 10^{-2}$, $(k,\ell)=(1,2)$.}
	\label{fig:snapshops:velo:Test2}
\end{figure}

\begin{figure}[h!]
	\centering
	\subfloat[$\widehat{p}_h=\varphi_0$ at $t_n=5 \times 10^{-6}$]{\includegraphics[width=0.3\linewidth, height=0.25\textwidth]{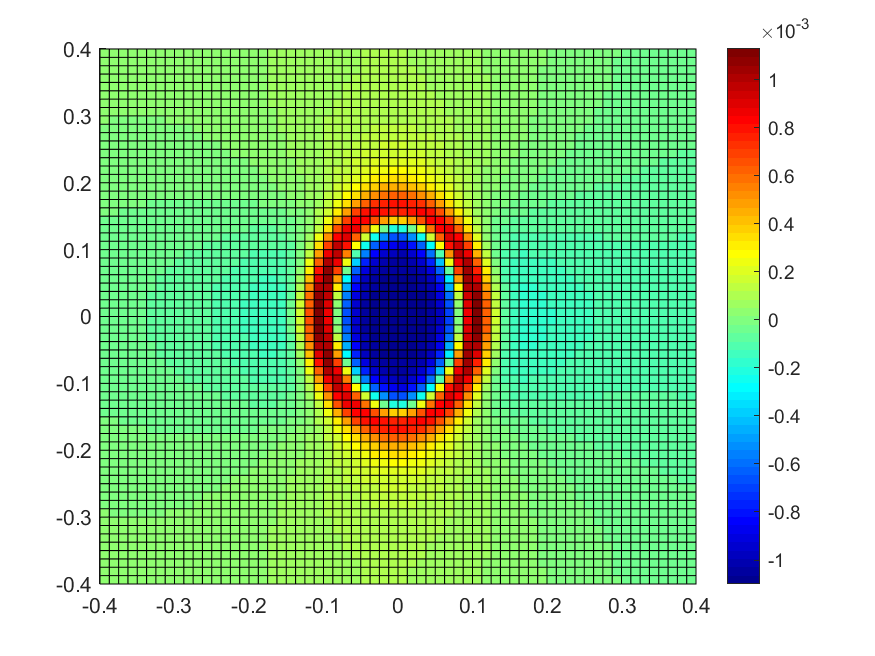}}\hfill
	\subfloat[$\widehat{p}_h$ at $t_n=5 \times 10^{-5}$]{\includegraphics[width=0.3\linewidth, height=0.25\textwidth]{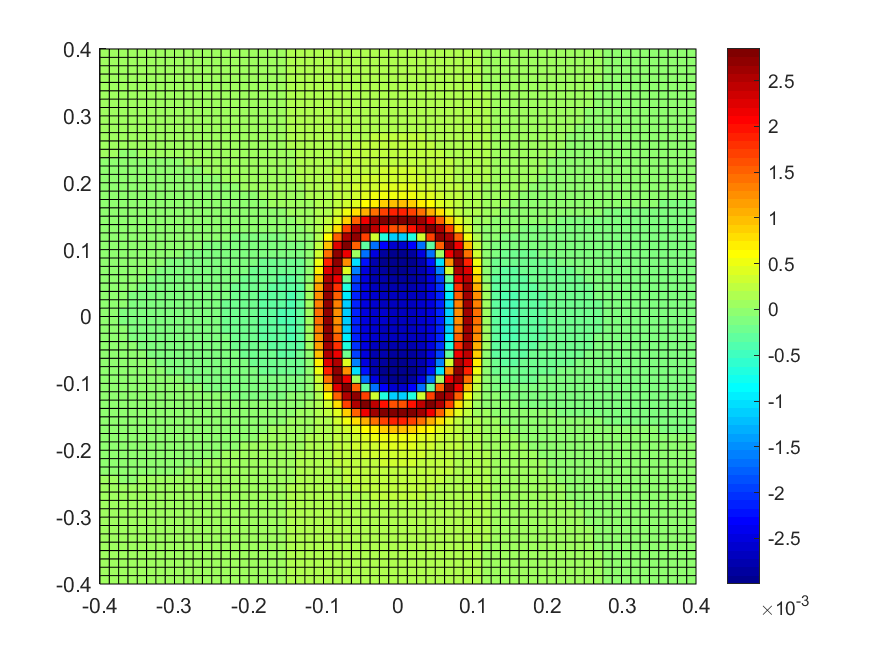}}\hfill
	\subfloat[$\widehat{p}_h$ at $t_n=1 \times 10^{-4}$]{\includegraphics[width=0.3\linewidth, height=0.25\textwidth]{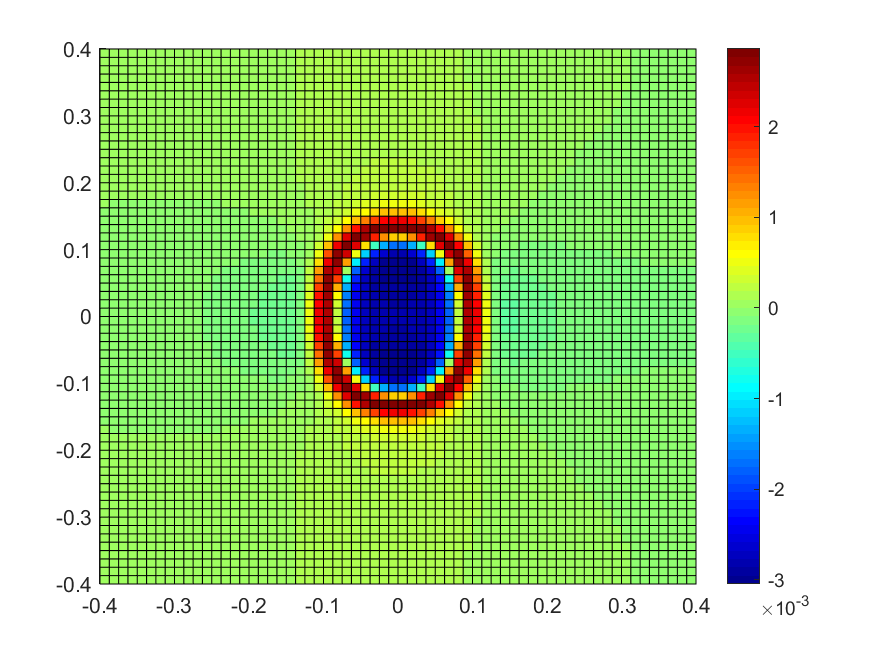}}\\	
	\subfloat[$\widehat{p}_h$ at $t_n=10^{-4}$]{\includegraphics[width=0.3\linewidth, 
		height=0.25\textwidth]{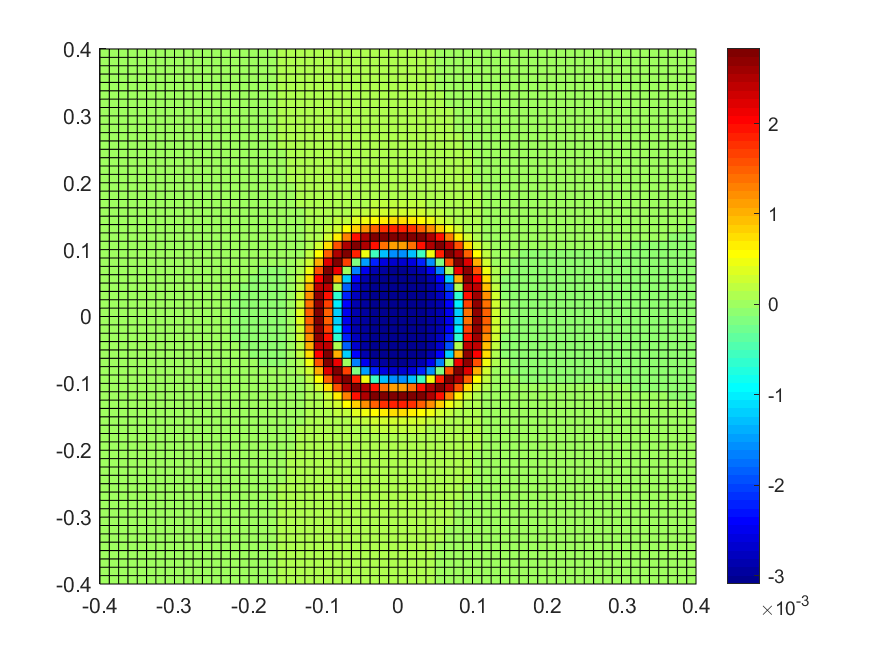}}\hfill
	\subfloat[$\widehat{p}_h$ at $t_n=1.75 \times 10^{-4}$]{\includegraphics[width=0.3\linewidth, height=0.25\textwidth]{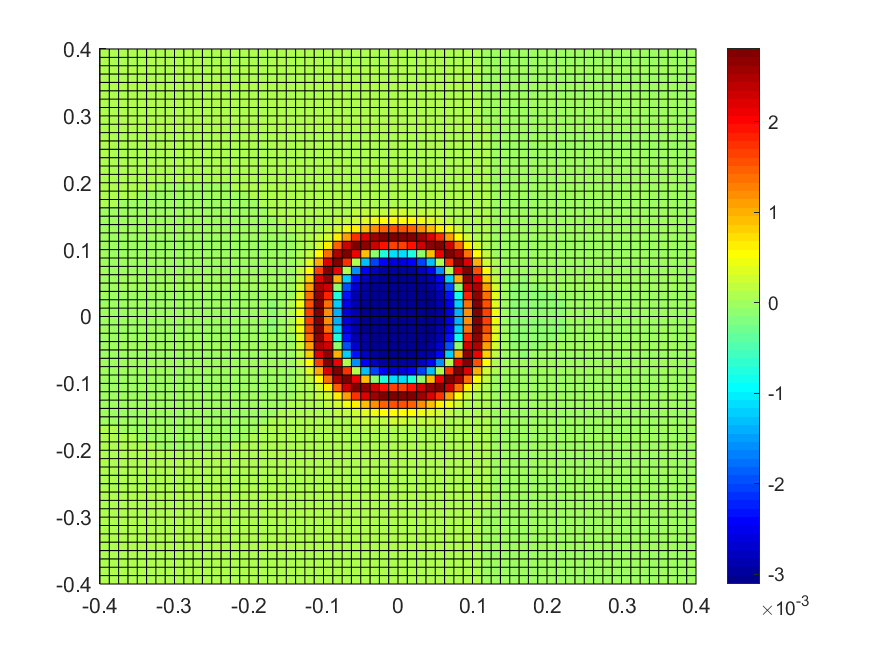}}\hfill
	\subfloat[$\widehat{p}_h$ at $t_n=2 \times 10^{-4}$]{\includegraphics[width=0.3\linewidth, height=0.25\textwidth]{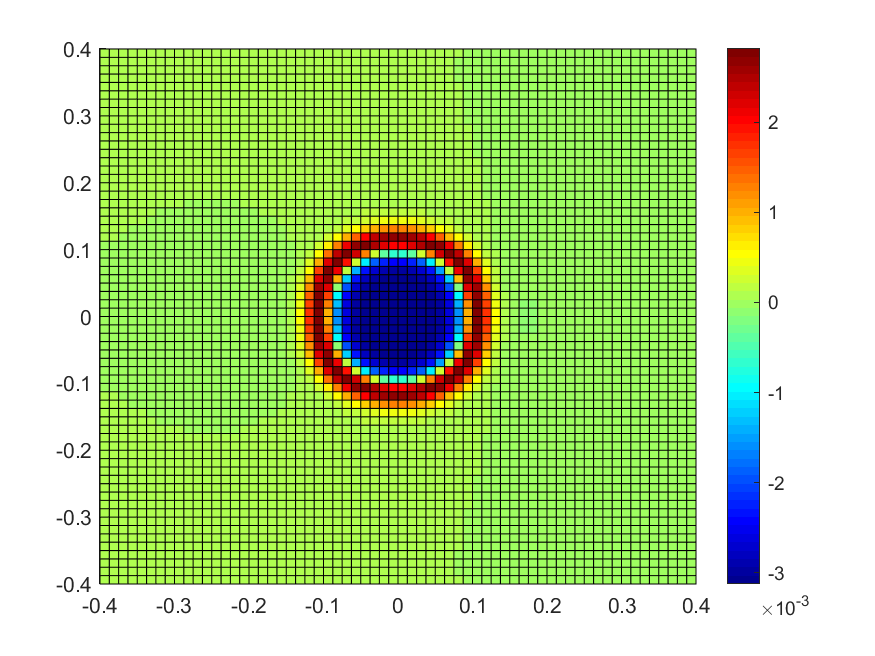}}
	\caption{Test 2. Snapshots  of the numerical pseudo-pressure solution $\widehat{p}^n_h$ at different time instants, using the fully-discrete VE scheme~\eqref{fully:disc:schm}, $\nu=1$, $\lambda=\gamma=0.1$, and $\varepsilon=2 \times 10^{-2}$, $(k,\ell)=(1,2)$.}
	\label{fig:snapshops:ph:Test2}
\end{figure}

The total mass $(\varphi_h^n,1)_\Omega$ 
remains constant throughout the entire simulation, namely $(\varphi_h^n,1)_\Omega \approx 0.53977092$ for $n=0,\ldots,N$.  This numerical evidence confirms the mass-conserving property of our fully-discrete scheme, in full agreement with the result established in Theorem~\ref{theo:property:fully}. 

%% file: 07_Conclusions.tex
\section{Concluding remarks}\label{sec:conclusions}
In this work, we have developed and analyzed semi- and fully discrete virtual element schemes for the Navier--Stokes-Cahn--Hilliard system. The analysis is based on a novel variational formulation involving only the velocity--pressure pair and the phase field, which allows us to eliminate the chemical potential as a primary unknown and, consequently, to reduce the number of unknowns in the discrete problem. The proposed schemes combine divergence-free and $C^1$-conforming virtual elements of arbitrary order with a backward Euler time discretization.

A key ingredient of our approach is the introduction of a discrete skew-symmetric trilinear form for the convective term in the Cahn--Hilliard equation, which guarantees mass conservation and ensures a discrete energy stability property. The theoretical analysis delivers optimal error estimates for the velocity, pressure, and phase variables.

We have presented two numerical experiments that confirm the accuracy, stability, and flexibility of the method across different polynomial degrees and polygonal meshes, fully supporting the theoretical findings and demonstrating its effectiveness for two-phase incompressible fluid flows. In particular, the  second experiment verify the mass conservation property (cf. Theorem~\ref{theo:property:fully}). Promising extensions of this work include the design and analysis of schemes that satisfy a discrete energy law (see Remarks~\ref{remark:properties} and \ref{Remark:mass:conserv}). Achieving this poses several challenges and constitutes an appealing direction for future research.